\documentclass{article}
\usepackage[latin1]{inputenc}
\usepackage[T1]{fontenc}
\usepackage{courier}
\usepackage[english]{babel}
\usepackage{listings}
\usepackage{amssymb}
\usepackage{amsmath}
\usepackage{amsthm}
\usepackage{bm}
\usepackage{graphicx}
\usepackage[margin=2.5cm]{geometry}

\usepackage{color}

\newtheorem{theorem}{Theorem}[section]
\newtheorem{lemma}[theorem]{Lemma}
\newtheorem{corollary}[theorem]{Corollary}

\newenvironment{Proof}{\textit{Proof.}}{$\hfill\Box$\\}

\DeclareMathOperator{\rg}{rg}
\DeclareMathOperator{\sign}{sign}
\DeclareMathOperator{\supp}{supp}
\DeclareMathOperator{\kernel}{ker}

\newcommand{\RR}{\mathbb{R}}
\newcommand{\inner}[2]{\langle #1,#2\rangle}

\newcommand{\ourtitle}{Testable uniqueness conditions for empirical assessment of undersampling levels in  total variation-regularized x-ray CT}

\newcommand{\spikes}{\textbf{spikes}}
\newcommand{\signedspikes}{\textbf{signed-spikes}}
\newcommand{\trununif}{\textbf{truncated-uniform}}
\newcommand{\altproj}{\textbf{alternating-projection}}
\newcommand{\altprojnonneg}{\textbf{alternating-projection-nonneg}}

\newcommand{\pnf}{\textup}
\newcommand{\ellone}{\pnf{L1}}
\newcommand{\atv}{\pnf{ATV}}
\newcommand{\itv}{\pnf{ITV}}

\newcommand{\dmat}{\textrm{\bf D}}
\newcommand{\dlinop}{\mathcal{D}}

\newcommand{\sysmat}{\textrm{\bf{A}}}
\newcommand{\sino}{\bm b}
\newcommand{\im}{\bm x}
\newcommand{\yvec}{\bm y}
\newcommand{\zvec}{\bm z}
\newcommand{\zcapvec}{\bm Z}
\newcommand{\wvec}{\bm w}
\newcommand{\vvec}{\bm v}

\begin{document}

\title{\ourtitle}

\author{Jakob S. J\o{}rgensen\footnote{Corresponding author. Email: jakj@dtu.dk. Department of Applied Mathematics and Computer Science, Technical University of Denmark, Richard Petersens Plads, 2800 Kgs. Lyngby, Denmark. }~, Christian Kruschel\footnote{Institute for Analysis and Algebra, Technical University of Braunschweig, Pockelsstra\ss{}e 14, 38106 Braunschweig, Germany.}~, and Dirk A. Lorenz$^{\dagger}$}

\maketitle

\begin{abstract}
  We study recoverability in fan-beam computed tomography (CT) with sparsity and total variation priors:
 how many underdetermined linear measurements suffice for recovering images of given sparsity?
Results from compressed sensing (CS) establish such conditions for, e.g., random measurements, but not for CT.
Recoverability is typically tested by checking whether a computed solution recovers the original. 
This approach cannot guarantee solution uniqueness and the recoverability decision therefore depends on the optimization algorithm.
We propose new computational methods to test recoverability by verifying solution uniqueness conditions.
Using both reconstruction and uniqueness testing we empirically study the number of CT measurements sufficient for recovery on new classes of sparse test images.
We demonstrate an average-case relation between sparsity and sufficient sampling and observe a sharp phase transition as known from CS, but never established for CT.
In addition to assessing recoverability more reliably, we show that uniqueness tests are often the faster option. 
\end{abstract}


\section{Introduction}
Regularization methods for tomographic reconstruction that exploit sparsity have been in the focus of research recently.
Motivated by the theory of compressed sensing (CS)~\cite{candes2006compressivesampling,Donoho2006} many papers proposed to use sparse or total variation (TV) regularization to compute tomographic reconstructions from underdetermined measurements~\cite{SidkyTV:06,sidky2008image, ritschl2011improved}. One promising goal is unchanged or even improved reconstruction quality from a significantly reduced sampling effort, thereby lowering the necessary radiation dose in medical computed tomography (CT) and scanning time in e.g. materials science and non-destructive testing.

Compressed sensing offers methodologies to predict under what circumstances it is possible to compute exact reconstructions from underdetermined linear measurements.
Usually these conditions depend both on the measurement matrix and on the signal class that is considered.
Roughly spoken, a standard result reads as follows: All vectors that are sparse enough can be reconstructed exactly from underdetermined measurements with a random matrix (e.g. all entries independently, identically Gaussian distributed) by computing the solution of the linear system that has the smallest $\ell^1$-norm~\cite{CandesTao2006nearoptimal}.
There are also results~\cite{needell2013cstv} that state exact recovery conditions for TV regularization~\cite{rudin1992nonlinear}. 
An overview of CS recovery guarantees can be found in \cite{FoucartRauhut:2013}.
It is generally acknowledged, however, that existing guarantees either do not apply to or give extremely pessimistic bounds in deterministic sampling contexts \cite{elad2010sparse}. In particular for CT, \cite{PetraSchnoerr2014,Denitiu2013,Joergensen_eqconpap_v2_arxiv:2014} describe the lack of general guarantees, while \cite{PetraSchnoerr2014,Denitiu2013} derive preliminary average-case results for certain restricted special geometries known as discrete geometry; however these results do not cover regular sampling patterns in CT, such as parallel-beam and fan-beam geometries.

In our recent work \cite{Joergensen_TMI:2013,Joergensen_eqconpap_v2_arxiv:2014} we have been interested in establishing conditions on sparsity and sampling levels sufficient for image recovery with regular CT sampling patterns. In particular, \cite{Joergensen_TMI:2013} suggests a link between gradient sparsity and sufficient sampling for accurate TV-reconstruction.
In \cite{Joergensen_eqconpap_v2_arxiv:2014}, we carried out empirical studies of the average sufficient number of CT fan-beam projection views for $\ell^1$-recovery as function of image sparsity. Using a phase diagram similar to the Donoho-Tanner \cite{DonohoTanner:2009} phase diagram 
we showed that $\ell^1$-recovery often admits sharp phase transitions
as the sampling level is increased, and that the critical sampling level increases with the number of image nonzeros. 

The present work considerably expands on the results of \cite{Joergensen_eqconpap_v2_arxiv:2014} by addressing two limitations. First, while $\ell^1$-regularization is useful for CT, TV-regularization is often a more successful sparsity prior for CT, because many objects have a piecewise constant appearance. The present work extends to study recovery using anisotropic and isotropic TV-regularization and as a step towards this proposes ways to generate images of a desired gradient sparsity.	
Second, both $\ell^1$ and TV reconstruction is possibly subject to non-unique solutions. The approach of \cite{Joergensen_eqconpap_v2_arxiv:2014} considers an image to be uniquely recoverable  if numerical solution of the relevant optimization problem recovers the original image, but does not consider whether the computed solution is unique. In the present work, we derive uniqueness tests that can be used to computationally verify $\ell^1$ and TV solution uniqueness.
We then compute phase diagrams using both reconstruction and uniqueness tests to verify the $\ell^1$ recovery results from \cite{Joergensen_eqconpap_v2_arxiv:2014} and for anisotropic and isotropic versions of TV. In all cases, we observe a pronounced average-case relation between sparsity and the sufficient sampling level for recovery as well as a sharp phase transition. We also compare the reconstruction and uniqueness test approaches in terms of computing time.

The paper is organized as follows. 
Section~\ref{sec:ctprob} describes the CT imaging model and $\ell^1$-norm and total variation regularization problems in study.
Section~\ref{sec:unique} describes necessary and sufficient conditions for solution uniqueness, while 
Section~\ref{sec:opti} presents our numerical implementations of reconstruction and uniqueness tests. 
Section~\ref{sec:study-design} describes how to generate test images with desired image or gradient sparsity.
Section~\ref{sec:results} presents our results establishing empirically a relation between sparsity and the  average sufficient sampling level for recovery. 
Finally, Section~\ref{sec:conclusion} discusses the results and concludes the paper.

\section{Sparse image reconstruction methods for computed tomography} \label{sec:ctprob}

\subsection{Imaging model}
\label{sec:imag-model}
Imaging by computed tomography (CT) exploits that x-rays are attenuated when passing through matter. The attenuation depends on the material traversed by the x-ray, as described by the so-called linear attenuation coefficient $\mu$, with denser materials generally attenuating more. The intensities $I_0$ and $I$ of an x-ray before and after passing through an object with linear attenuation coefficient $\mu({\bm s})$, as function of the spatial coordinate $\bm s$, can be modeled by Lambert-Beer's law, see, e.g., \cite{BuzugBook:2008}, which in a rearranged form reads
\begin{align}
 - \log \frac{I}{I_0} = \int_L \mu({\bm s}) \mathrm{d} {\bm s}, \label{eq:lambertbeerrearranged}
\end{align}
where $\int_L \mu({\bm s}) \mathrm{d} {\bm s}$ denotes the line integral along the line $L$ describing the x-ray path. 
By means of discretizing the object into $n$ pixels and the data by assuming that $m$ individual x-rays with infinitesimal width are used, a fully discretized imaging model can be written $\sino{} = \sysmat{}\im{}$. Here, $\im{}$ is a vector of length $n$ of all pixel values stacked. $\sysmat{}$ is an $m$-by-$n$ matrix of which the $(i,j)th$ element, $\sysmat{}_{ij}$, equals the path length of the $i$th ray through pixel $j$, such that $\sum_j \sysmat{}_{ij}\im{}_j$ approximates the line integral in \eqref{eq:lambertbeerrearranged} for ray $i$. Each ray only intersects a small number of pixels (for a square $\sqrt{n}$-by-$\sqrt{n}$ array of the order of $\sqrt{n}$), causing the remaining $\sysmat{}_{ij}$ values to be zero and $\sysmat{}$ to be very sparse. $\sino{}$ is a vector of length $m$ with the log-transformed data, i.e., $\sino{}_i =  - \log (I_i/I_0)$ for rays $i = 1,\dots,m$.

In the present work, we consider a 2-D fan-beam geometry with equi-angular projection views acquired from $360^\circ$ around the image. Due to rotational symmetry we consider the image to be the largest inscribed disk within a square $N_\text{side}$-by-$N_\text{side}$ pixel array, hence consisting of $n \approx  \pi/4 \cdot N_\text{side}^2$ pixels. 
The source-to-detector distance is set to $2N_\text{side}$ and the detector has the shape of a circular arc centered at the source and consists of $2N_\text{side}$ detector elements. The number of projection views is denoted $N_\text{v}$ and the fan angle is set to $28.07^\circ$ so that precisely the inscribed disk is covered. The total number of linear measurements is $m = 2N_\text{side}N_\text{v}$ and the $m$-by-$n$ system matrix $\sysmat{}$ is computed using the function \texttt{fanbeamtomo} in the MATLAB package AIR Tools \cite{Hansen2012}.

\subsection{Sparse regularization}
\label{sec:var-reg}

In the context of an underdetermined system $\sino{} = \sysmat{}\im{}$, $\sysmat{} \in \RR^{m\times{}n}$, $m<n$,  one obtains a whole affine space of solutions.
One selects one solution of this space by considering a regularization functional $R$, and computing the minimum-$R$ solution, i.e. the solution of
\[
\min_x R(\im{})\ \text{subject to}\ \sysmat{}\im{}=\sino{}.  \label{eq:eqcon}
\]
In the present work we study
$\ell^1$-norm (\ellone{}), anisotropic TV (\atv{}), and 
isotropic TV (\itv{}) regularization.
First, the $\ell^1$-norm, defined as
  \begin{equation}
\ellone{}:\quad\quad    R(\im{}) = \|\im{}\|_1 = \sum_j |\im{}_j|, \label{eq:ell1}
  \end{equation}
  enforces sparsity of the minimizer. 
  $\ell^1$-norm minimization forms one backbone of compressed sensing and the problem of computing minimum-$\ell^1$-norm solution of underdetermined systems is also known as Basis Pursuit~\cite{Chen:1998}. 

  Second, anisotropic TV can be written using
  a set of vectors $\bm d_i$, $i=1,\dots,N$ (of length $n$ as $\im{}$) and the matrix $\dmat{} = [\bm d_1,\dots,\bm d_N]$ as
  \begin{equation}
    \atv{}:\quad\quad R(\im{}) = \|\dmat{}^T \im{}\|_1 = \sum_{i=1}^N|\bm d_i^T \im{}|.\label{eq:analysis-ell1}
  \end{equation}
  The matrix $\dmat{}$ is called dictionary in this context and the minimum-$R$ solution is seeking a solution in which inner products with the dictionaries entries form a sparse vector, for example to enforce sparsity in the coefficients of a wavelet basis or a dictionary learned from training images.
  We use \atv{} to denote the general case but focus in the present work on the anisotropic TV, which is a special case where $\dmat{}^T$ contains finite-difference approximations of the horizontal and vertical derivatives in each pixel, i.e., $N = 2n$. 
  However, we emphasize that all our theoretical results hold in the general case.
  
 Third, isotropic TV can be written as a special case of the group sparsity problem \cite{Huang2010}. This problem 
 differs in a small but crucial point from \atv{}:
Here we do not simply take the $\ell^1$-norm of $\dmat{}^T\im{}$ as objective function, but we take a mixed $\ell^{1,2}$-norm.
To fix notation, consider a linear mapping $\dlinop :\RR^{r\times p}\to \RR^n$, i.e. the transposed map is $\dlinop^T:\RR^n\to\RR^{r\times p}$, where $r$ is the number of groups and $p$ is the number of pixels in each group. For $\bm Y \in \RR^{r\times p}$ we consider the mixed $\ell^{1,2}$-norm
\[
\|\bm Y\|_{1,2} = \sum_{i=1}^r \|\bm Y_i\|_2 = \sum_{i=1}^r\Big(\sum_{j=1}^p|\bm Y_{i,j}|^2\Big)^{1/2},
\]
where $\bm Y_i\in \RR^p$ denotes the $i$-th row of $\bm Y$.
We then set
  \begin{equation}
    \itv{}: \quad\quad R(\im{}) = \|\dlinop^T \im{} \|_{1,2} = \sum_{i=1}^r \|(\dlinop^T \im{})_i\|_2, \label{eq:iso-TV}
  \end{equation}
  We use \itv{} to denote the general case but focus in the present work on the isotropic TV, for which we take each group to be the horizontal and vertical finite-difference partial-derivative approximations in each pixel, i.e., $p=2$ and $r=n$.   Again, our theoretical results hold in the general case.
  
  Note that the matrix $\dmat{}$ in \atv{} and linear mapping $\dlinop{}$ in \itv{} are closely related for the case of finite differences, namely
  \begin{equation*}
   (\dlinop{}^T\im{})_i = [\bm d_i^T\im{}, \bm d_{i+n}^T\im{}] \quad \text{for all} \quad  i\le n.
  \end{equation*}

\section{Necessary and sufficient conditions for uniqueness of minimizer} \label{sec:unique}
In this section, conditions for the uniqueness of the considered optimization problems are introduced.
We use the following notation.
The complement of an index set $I\subset\{1,...,n\}$ is denoted as $I^c = \{1,...,n\}\backslash I$.
If $\sysmat{}\in\mathbb{R}^{m\times n}, m < n$, then $\sysmat{}_I$ denotes the submatrix of $\sysmat{}$ with columns indexed by $I$.
The transposed of such a submatrix is denoted by $\sysmat{}_I^T$.

\subsection{Uniqueness conditions for \ellone{}}
The following theorem gives necessary and sufficient conditions for a vector $\im{}^*$ to solve \ellone{} uniquely.
\begin{theorem}[\cite{Plumbley:2007}]\label{thm:dual-cert}
Let $\sysmat{}\in\RR^{m\times n}$ with $m<n$ and $\im{}^*\in\RR^n$ with $I = \supp(\im{}^*)$.
Then $\im{}^*$ is the unique solution of 
\begin{equation}\label{eq:BP}
  \ellone{}:\quad\quad\min_{\yvec{}}\|\yvec{}\|_1~\mbox{subject to}~\sysmat{}\yvec{}=\sysmat{}\im{}^*
\end{equation}
if and only if 
\begin{equation}
\kernel(\sysmat{}_I) = \{\bm 0\}
\end{equation}
and there exists $\wvec{}\in\RR^m$ such that
\begin{equation}
  \label{eq:dual-cert}
  \sysmat{}_I^T\wvec{} = \sign(\im{}^*)_I,~~\|\sysmat{}_{I^c}^T\wvec{}\|_\infty < 1.
\end{equation}
\end{theorem}
The theorem has two important consequences: First, the recoverability of a vector $\im{}^*$ only depends on its sign pattern, not the magnitude of its entries.
Second, to know if some vector $\im{}^*$ can be recovered from the measurement $\sysmat{}\im{}^*$ (for some given, fixed $\sysmat{}$), one only needs to check the existence of a vector $\wvec{}\in\RR^m$ such that the equality and the inequality from~\eqref{eq:dual-cert} are fulfilled. As we will show in Section~\ref{sec:study-design}, this can be done by solving an $m$-dimensional linear program.
Since the vector $\wvec{}$ is related to the dual optimization problem of~\eqref{eq:BP}, it will be called \emph{dual certificate} in the following.

A seemingly different necessary and sufficient condition for $\im{}^*$ being the unique \ellone{} solution given in~\cite{Grasmair2011} 
is the existence of a vector $\wvec{}$ such that $\sysmat{}_I^T \wvec{} = \sign(\im{}^*)_I$, $\|\sysmat{}_{I^c}^T \wvec{}\|_\infty \leq 1$ and that $\sysmat{}_J$ is injective where $J = \{j\in\{1,\dots,n\}\ |\ |(\sysmat{}^T\wvec{})_j|=1\}$. This condition may appear weaker but is in fact equivalent to Theorem \ref{thm:dual-cert} and we do not use it any further.

\subsection{Uniqueness conditions for \atv{}}\label{Sec:Anal1Cond}
We extend the previous result to the \atv{} case 
through a straightforward generalization of Theorem~\ref{thm:dual-cert}.
The following theorem basically adapts the result in \cite{haltmeier2013stable}.
\begin{theorem}\label{Sec:Th:Anal1Cond}
Let $\sysmat{}\in\RR^{m\times n}$ with $m<n, \dmat{}\in\RR^{n\times N}$ and $\im{}^*\in\RR^n$ with $I = \supp(\dmat{}^T\im{}^*)$.
Then it holds that $\im{}^*$ is the unique solution of
\begin{equation}\label{eq:l1analysis}
  \atv{}:\quad\quad\min_{\yvec{}}\|\dmat{}^T\yvec{}\|_1~\mbox{subject to}~\sysmat{}\yvec{}=\sysmat{}\im{}^*
\end{equation}
if and only if
\begin{align}
\kernel(\sysmat{})\cap\kernel(\dmat{}_{I^c}^T) = \{\bm 0\}\label{Th:OptCond:NullSpaceCond}
\end{align}
 and there exists $\wvec{}\in\RR^m$ and $\vvec{}\in\RR^N$ such that
\begin{align*}
\dmat{}\vvec{} =~\sysmat{}^T\wvec{},~~ \vvec{}_I =~\sign(\dmat{}_I^T\im{}^*),~~ \|\vvec{}_{I^c}\|_\infty <~1.
\end{align*}
\end{theorem}
\begin{Proof}
The proof is separated into two parts, each for one direction. First, it will be shown that $\im{}^*$ is the unique solution under the given conditions.
For each $\yvec{}\in\RR^n$ with $\yvec{}\neq \im{}^*$ it holds $\dmat{}_{I^c}^T\yvec{} = \dmat{}_{I^c}^T(\yvec{}-\im{}^*) \neq \bm 0$ since $\yvec{}-\im{}^*$ is a nontrivial null space element of $\sysmat{}$ and \eqref{Th:OptCond:NullSpaceCond} is provided.
Further, with $\bm s = \sign(\dmat{}^T\im{}^*)$, the remaining conditions imply
\begin{align*}
\|\dmat{}^T\im{}^*\|_1 & = \bm s^T\dmat{}^T\im{}^* = (\dmat{}\vvec{})^T\im{}^* = \wvec{}^T\sysmat{}\im{}^* = \wvec{}^T\sysmat{}\yvec{} = \vvec{}^T\dmat{}^T\yvec{}\\
&\le\underbrace{\|\vvec{}_I\|_\infty}_{=1}\|\dmat{}_I^T\yvec{}\|_1 + \underbrace{\|\vvec{}_{I^c}\|_\infty}_{<1}\|\dmat{}_{I^c}^T\yvec{}\|_1 <\|\dmat{}^T\yvec{}\|_1.
\end{align*}
As a result of \eqref{Th:OptCond:NullSpaceCond}, the latter inequality is truly strict, see above. This proves the conditions to be sufficient for \atv{} uniqueness.

Let us assume the vector $\im{}^*$ solves the considered optimization problem uniquely.
Since the optimization problem \eqref{eq:l1analysis} is piecewise linear, for all $\bm h\in\kernel(\sysmat{})\backslash\{\bm 0\}$ it holds that
\begin{align*}
0 < \lim_{t\rightarrow0, t>0}\frac1t\left[\|\dmat{}^T(\im{}^*+t \bm h)\|_1 -\|\dmat{}^T\im{}^*\|_1\right] = \sum_{i\in I}\sign(\bm d_i^T\im{}^*)\bm d_i^T\bm h + \sum_{i\in I^c}|\bm d_i^T\bm h|,
\end{align*}
and, since the latter inequality holds for all nontrivial null space elements, in particular $-\bm h$, it holds that
\begin{align}
\left|\sum_{i\in I}\sign(\bm d_i^T\im{}^*)\bm d_i^T\bm h\right| < \sum_{i\in I^c}|\bm d_i^T\bm h|.\label{Th:OptCond:StrictNSInEqu}
\end{align}
Moreover, it holds $\dmat{}_{I^c}^T\bm h\neq \bm 0$ which implies \eqref{Th:OptCond:NullSpaceCond}.

Finally the rest of the conditions will be proved.
Choose $\bm \eta\in\RR^N$ with $\bm \eta_i = \sign(\bm d_i\im{}^*)$ for $i\in I$ and $\bm \eta_j = 0$ for $j\notin I$ and assume $\dmat{}^T\bm \eta$ is not an element of the range of $\sysmat{}^T$ -- otherwise the proof would be finished.
Let $\kernel\sysmat{}$ be $p$-dimensional and $\{\wvec{}^{(l)}\}_{1\le l\le p}$ be a basis  of $\kernel\sysmat{}$ with $1 = \bm \eta^T\dmat{}^T\wvec{}^{(l)}$; it holds
\begin{align*}
1 = \bm \eta^T\dmat{}^T\wvec{}^{(l)} = \sum_{i\in I}\sign(\bm d_i^T\im{}^*)\bm d_i^T\wvec{}^{(l)}~\mbox{for all}~1\le l\le p.
\end{align*}
In the following, a vector $\tilde{\bm\xi}\in\RR^N$ will be constructed such that $\tilde{\bm\xi}^T\dmat{}^T\wvec{}^{(l)} = -1$ holds for all $1\le l\le p$; hence, the vector $\dmat{}^T(\bm \eta + \tilde{\bm\xi})$ is an element of the range of $\sysmat{}^T$ due to its orthogonality to the null space of $\sysmat{}$.
Consider a solution $\tilde{\bm\xi}\in\RR^N$ of the problem
\begin{align*}
\min_{\bm\xi\in\RR^N}\max_{j\in I^c}|\bm\xi_j|~\mbox{subject to}~\bm\eta^T\dmat{}^T\wvec{}^{(l)} = -1~\mbox{for all}~1\le l\le p
\end{align*}
and $\bm q^*\in\RR^p$ as a solution of its dual problem
\begin{align*}
\min_{\bm q\in\RR^p}-\sum_{i=1}^p \bm q_i~\mbox{ subject to}~\sum_{j\in I^c}\left|\sum_{l=1}^p \bm q_l \bm d_j^T\wvec{}^{(l)}\right|\le1,
\end{align*}
then for all $1\le l\le p$ it holds
\begin{align*}
\max_{j\in I^c}|\tilde{\bm\xi}_j| = \left|\sum_{i=1}^p \bm q^*_i\right| =  \left|\sum_{i\in I}\sign(\bm d_i^T\im{}^*)\bm d_i^T\sum_{i=1}^p \bm q^*_i\wvec{}^{(l)}\right| < \sum_{j\in I^c}\left|\sum_{l=1}^p \bm q^*_l \bm d_j^T\wvec{}^{(l)}\right|\le1
\end{align*}
since $\sum_{i=1}^p \bm q^*_i\wvec{}^{(l)}$ satisfies \eqref{Th:OptCond:StrictNSInEqu}.
This proves the remaining conditions as necessary.
\end{Proof}

We also call the pair $(\vvec{},\wvec{})$ of vectors a \emph{dual certificate} for \atv{}. 
Note that adding a kernel element of $\dmat{}^T$ does not affect the property of being uniquely recoverable by \atv{}:

\begin{corollary} \label{cor:nullvector}
  In the setting of Theorem~\ref{Sec:Th:Anal1Cond}, let $\bm h\in\ker(\dmat{}^T)$ and let $\im{}^*$ be the unique solution of~\eqref{eq:l1analysis}.
  Then $\tilde \im{}=\im{}^*+\bm h$ is also a unique solution of~\eqref{eq:l1analysis}, with $\im{}^*$ replaced by $\tilde \im{}$.
\end{corollary}
\begin{Proof}
  Just note that $\supp(\dmat{}^T\im{}^*) = \supp(\dmat{}^T\tilde \im{})$ and that if $(\vvec{},\wvec{})$ is a dual certificate for $\im{}^*$ it is also a dual certificate for $\tilde \im{}$.
\end{Proof}

\subsection{Uniqueness conditions for \itv{}}

The following theorem gives sufficient conditions on $\sysmat{}$, $\dlinop$ and $\im{}^*$ such that $\im{}^*$ is the unique \itv{} solution; as far as we know, it is unknown if they are also necessary.

\begin{theorem} \label{thm:isotv}
  Let $\sysmat{}\in\RR^{m\times n}$, $m<n$, $\dlinop:\RR^{r\times p} \to \RR^n$ and
  $\im{}^*\in\RR^{n}$ and denote by $I = \{i\in\{1,\dots r\} :\ (\dlinop^T\im{}^*)_i\neq \bm 0\}$.  Then it holds that $\im{}^*$ is the unique solution
  of
  \begin{equation}
  \label{eq:isotv}
  \itv{}:\quad\quad\min_{\yvec{}} \|\dlinop^T\yvec{}\|_{1,2}\ \text{subject to}\ \sysmat{}\yvec{} = \sysmat{}\im{}^*
\end{equation}
if there exists $\wvec{}\in\RR^m$ and $\bm Y\in\RR^{r\times p}$ such that $\dlinop \bm Y = \sysmat{}^T \wvec{}$ and
  \begin{enumerate}
  \item There exists $\bm Y\in \RR^{r\times p}$ such that 
    \begin{equation*}
      \bm Y_i = \frac{(\dlinop^T \im{}^*)_i}{\|(\dlinop^T \im{}^*)_i\|_2} \ \text{ for }\ i\in I,\qquad
      \|\bm Y_i\|_2< 1\ \text{ for }\ i\notin I
    \end{equation*}
    and
  \item with $S = \{\vvec{}\in \RR^n\ :\
    (\dlinop^T \vvec{})_i = \bm 0\ \text{for}\ i\notin I\}$ it holds that $\sysmat{}\vvec{} = \sysmat{}\yvec{}$
    implies $\yvec{}=\vvec{}$ for all $\yvec{},\vvec{}\in S$.
  \end{enumerate}
\end{theorem}

\begin{Proof}
  On $\RR^{r\times p}$ we have the usual inner product $\inner{\bm Y}{\bm Z} =
  \sum_{i=1}^r\sum_{j=1}^p \bm Y_{ij}\bm Z_{ij}$. Then it holds for the $\bm Y$
  defined above that
  \[
  \|\dlinop^T\yvec{}\|_{1,2} = \inner{\bm Y}{\dlinop^T\yvec{}} = \inner{\dlinop \bm Y}{\yvec{}} = \inner{\sysmat{}^T \wvec{}}{\yvec{}}
  = \inner{\wvec{}}{\sysmat{}\yvec{}}.
  \]
  Also note that obviously, $\yvec{}\in S$.
  
  Now consider a vector $\vvec{}\in\RR^n$ that is feasible
  for~(\ref{eq:isotv}), i.e. it holds that $\sysmat{}\vvec{} = \sysmat{}\im{}^*$. If $\vvec{}$ would
  be an element of $S$, then, by assumption, $\vvec{} = \im{}^*$. Hence, if
  $\vvec{}\neq \yvec{}$, then $\vvec{}\notin S$, i.e. there is at least one index
  $i_0\notin I$ such that $(\dlinop^T\vvec{})_{i_0} \neq \bm 0$. Consequently we get
  \begin{align*}
    \|\dlinop^T\im{}^*\|_{1,2} & = \inner{\wvec{}}{\sysmat{}\yvec{}} = \inner{\wvec{}}{\sysmat{}\vvec{}} = \inner{\sysmat{}^T \wvec{}}{\vvec{}} = \inner{\dlinop \bm Y}{\vvec{}}\\
    & = \sum_{i=1}^r \bm Y_i\cdot (\dlinop^T\vvec{})_i \qquad\text{[`$\cdot$' denotes the inner product in $\RR^d$]}\\
    & \leq \sum_{i=1}^r\|\bm Y_i\|_2\|(\dlinop^T\vvec{})_i\|_2 
     = \underbrace{\|\bm Y_{i_0}\|_2}_{<1}\underbrace{\|(\dlinop^T\vvec{})_{i_0}\|_2}_{>0} + \sum_{i\neq i_0}\underbrace{\|\bm Y_i\|_2}_{\leq 1}\|(\dlinop^T\vvec{})_i\|_2\\
    & < \sum_{i=1}^r \|(\dlinop^T\vvec{})_i\|_2 = \|\dlinop^T\vvec{}\|_{1,2}.
  \end{align*}
  In other words, every feasible $\vvec{}$ different from $\im{}^*$ has a larger
  objective value.
\end{Proof}

Whereas geometrical interpretations of the conditions in Theorem \ref{Sec:Th:Anal1Cond} for \atv{} and Theorem \ref{thm:isotv} for \itv{} are difficult to establish and may require a high level of comprehension, the \ellone{} conditions in Theorem \ref{thm:dual-cert} can be explained easily. For, say, a full rank matrix $A\in\mathbb{R}^{m\times n}$ with $m < n$ and $x^*\in\mathbb{R}^n$ with $I = \mathrm{supp}(x^*)$ and $|I| = k \le m$, the conditions \eqref{eq:dual-cert} can be seen as the intersection of the affine space $\mathrm{rg}(A^T)$ with the $n$-dimensional hypercube $[-1,+1]^n$.
Indeed, the affine space cuts the interior of an $(n-k)$-dimensional face of the hypercube; which face is sliced depends on the sign pattern of $x^*$.

\section{Numerical implementation of reconstruction and uniqueness tests} \label{sec:opti}

\subsection{Reconstruction problems}
The three regularized reconstruction problems \ellone{}, \atv{} and \itv{}
are solved numerically by a primal-dual interior-point method using MOSEK \cite{MOSEK}. Our motivation for this choice of method is to ensure that the optimization problem is solved accurately. MOSEK achieves this by producing a certificate of optimality for the returned numerical solution, i.e., a primal-dual solution pair with duality  gap numerically close to zero.

Across the large number of optimizations done for the present study, it is our experience that other methods such as accelerated first-order methods~\cite{tseng1991applications,NesterovBook2004} and primal-dual methods~\cite{ChambollePock:2011,esser2010primaldual} are less reliable in actually arriving at a solution of the equality-constrained problem accurate enough to reliably assess whether it is equal to the original image. 

We solve \ellone{} as a linear program (LP) by introducing $\bm q \in \RR^n$ to bound $\im{}$:
\begin{equation}
 \min_{\im{},\bm q}\; \mathbf{1}^T \bm q \qquad \text{subject to} \qquad \sysmat{}\im{} = \sino{} \quad \text{and} \quad -\bm q \leq \im{} \leq \bm q .
\end{equation}
In a similar fashion, \atv{} can be solved as an LP. By defining $\bm z = \dmat{}^T\im{}$ and using $\bm q$ for bounding $\bm z$ we can solve the problem as
\begin{equation}
 \min_{\im{},\bm z,\bm q}\; \mathbf{1}^T \bm q \qquad \text{subject to} \qquad \sysmat{}\im{} = \sino{} \quad \text{and} \quad \bm z = \dmat{}^T\im{} \quad \text{and} \quad -\bm q \leq \bm z \leq \bm q.
\end{equation}
\itv{} can be recast as a conic optimization problem, which can also be solved by MOSEK. Again, we introduce the bounding vector $\bm q \in \RR^n$ and can solve the problem as
\begin{equation}
 \min_{\im{},\bm q}\; \mathbf{1}^T \bm q \qquad \text{subject to} \qquad \sysmat{}\im{} = \sino{} \quad \text{and} \quad \|(\dlinop^T \im{})_j\|_2 \leq \bm q_j \quad \text{for} \quad j = 1,\dots,n,
\end{equation}
in which each of the $n$ inequalities specify a quadratic conic constraint.

\subsection{Uniqueness tests}

As stated by Theorem \ref{thm:dual-cert}, we can show that $\im{}^*$ is the unique \ellone{} minimizer if and only if $\sysmat{}_I$ is injective and there exists a $\wvec{}\in \RR^m$ such that $\sysmat{}_I^T \wvec{} = \sign(\im{}^*)_I$ and $\|\sysmat{}_{I^c}^T \wvec{} \|_\infty < 1$. Injectivity is tested by evaluating whether $\sysmat{}_I$ has full column rank. The second condition can be tested by minimizing $\|\sysmat{}_{I^c}^T \wvec{}\|_\infty$ with respect to $\wvec{}$ while respecting the equality constraint $\sysmat{}_I^T \wvec{} = \sign(\im{}^*)_I$. By splitting the infinity-norm into a two-sided inequality constraint, this problem can be solved as an LP,
\begin{align}
 \min_{t,\wvec{}} \quad& t \\
\text{subject to}\quad & -t \mathbf{1} \leq \sysmat{}_{I^c}^T\wvec{} \leq t\mathbf{1}\\
&\sysmat{}_I^T \wvec{} = \sign(\im{}^*)_I,
\end{align}
which we accomplish by use of MOSEK's primal-dual interior-point method.
For the optimal solution $(t^\star,\wvec{}^\star)$, by definition we have the smallest possible $t = \|\sysmat{}_{I^c}^T\wvec{}\|_\infty$. If $t^\star$ is not smaller than one, then no $\wvec{}$ exists with smaller $t$. We therefore declare $\im{}^*$ the unique minimizer if $t<1$, and if $t\geq 1$, $\im{}^*$ cannot be the unique minimizer. Numerically, we test whether $t^\star < 1 - \epsilon$ for $\epsilon = 10^{-5}$ to ensure that the inequality is satisfied strictly. Technically, by doing this we risk falsely rejecting solution uniqueness of some $\im{}^*$, namely the ones for which $1-\epsilon < t^\star<1$. However, we found the decision on uniqueness to be robust to other choices of $\epsilon$, so we believe this is not a problem in practice.

Regarding \atv{}, we can show solution uniqueness of a vector $\im{}^*$ by use of Theorem \ref{Sec:Th:Anal1Cond}.
The condition of zero-intersection of $\ker(\sysmat{})$ and $\ker(\dmat{}_{I^c}^T)$ 
can be checked numerically by evaluating whether the matrix $(\sysmat{};\dmat{}_{I^c}^T)$ has full rank, where semicolon means vertical concatenation. Similar to the \ellone{} case, the second condition can be tested by solving in MOSEK the LP,
\begin{align}
 \min_{t,\vvec{},\wvec{}} \quad& t \\
\text{subject to}\quad & -t \mathbf{1} \leq \vvec{}_{I^c} \leq t\mathbf{1}\\
&\sysmat{}_I^T \wvec{} = \dmat{}_I\vvec{}_I + \dmat{}_{I^c}\vvec{}_{I^c} \\
&\vvec{}_I = \sign(\dmat{}_I^T\im{}^*),
\end{align}
and assessing whether the optimal $t^\star$ is smaller than $1 - \epsilon$.

For isotropic TV we can show solution uniqueness by use of Theorem \ref{thm:isotv}. We let $\bm Y_I$ and $\bm Y_{I^c}$ denote $\bm Y$ restricted to rows $I$ and $I^c$, and similarly for $\dlinop_I$ and $\dlinop_{I^c}$. Given $\im{}^*$ we construct a $\bm Y$ and a $\wvec{}$ satisfying the requirements 
by solving the conic program
\begin{align}
 \min_{t,\bm Y,\wvec{}} \quad& t \\
 \text{subject to}\quad &\dmat{}_{I^c}\bm Y_{I^c} -\sysmat{}^T\wvec{} = -\dmat{}_I \bm Y_I \\
 & \bm Y_i = \frac{(\dlinop^T \im{}^*)_i}{\|(\dlinop^T \im{}^*)_i\|_2} \quad \text{for} \quad i \in I\\
 &\|\bm Y_i\|_2 \leq t  \quad \text{for} \quad  i \in I^c.
\end{align}
If the optimal value $t^\star$ is greater than $1 - \epsilon$, then part (1) of Theorem \ref{thm:isotv} is not fulfilled and we cannot show uniqueness. If instead $t^\star < 1-\epsilon$ we proceed to part (2). 

The set $S$ can be equivalently described as the 
kernel of $\dmat{}_{I^c}^T$. Letting $\textrm{\bf K}$ denote a matrix of basis vectors of the kernel, any $\vvec{} \in S$ can be represented using a coefficient vector $\bm c$ such that $\vvec{} = \textrm{\bf K}\bm c$. We numerically test the injectivity requirement of $\sysmat{}$ on $S$ by evaluating whether $\sysmat{}\textrm{\bf K}$ has full column rank. If true, we have shown solution uniqueness.

\section{Generation of sparse test images}\label{sec:study-design}

\subsection{Images for \ellone{}}
\label{sec:spikes}

The \spikes{} class will be used to test for exact recovery of sparse
images from tomographic measurements by \ellone{} (\ref{eq:ell1}). Since we are interested in recovery in dependence on
the sparsity, we follow the usual approach~\cite{blanchard2011compressed} and build test
images consisting of a given number of spikes at random positions and
with entries sampled from a uniform distribution on $[0,1]$. The \spikes{} images hence are non-negative. 
We also consider a signed version, \signedspikes{}, with the only differences that entries are sampled from the uniform distribution on $[-1,1]$. 
Fig. \ref{fig:spikesexamples} shows example images of each class at a range of relative sparsity values.

\begin{figure}[tb]
\newcommand{\ww}{0.19}
\includegraphics[width=\ww\linewidth]{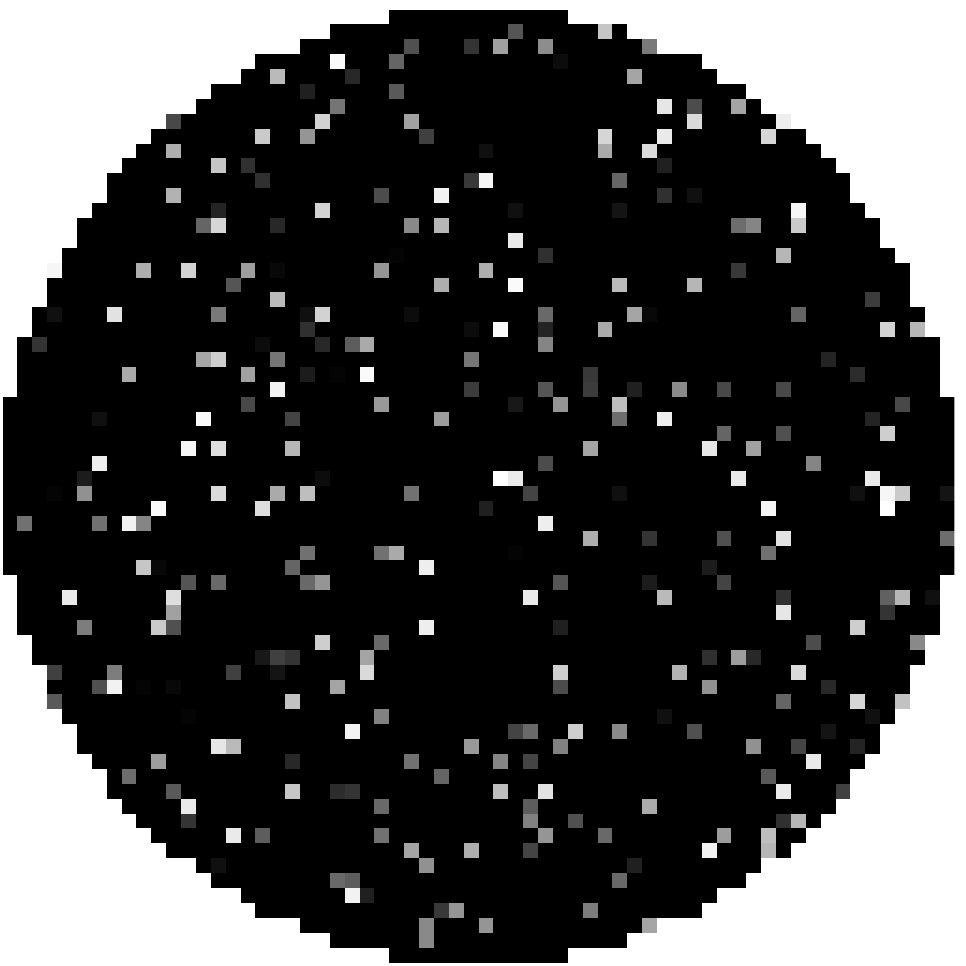}
\includegraphics[width=\ww\linewidth]{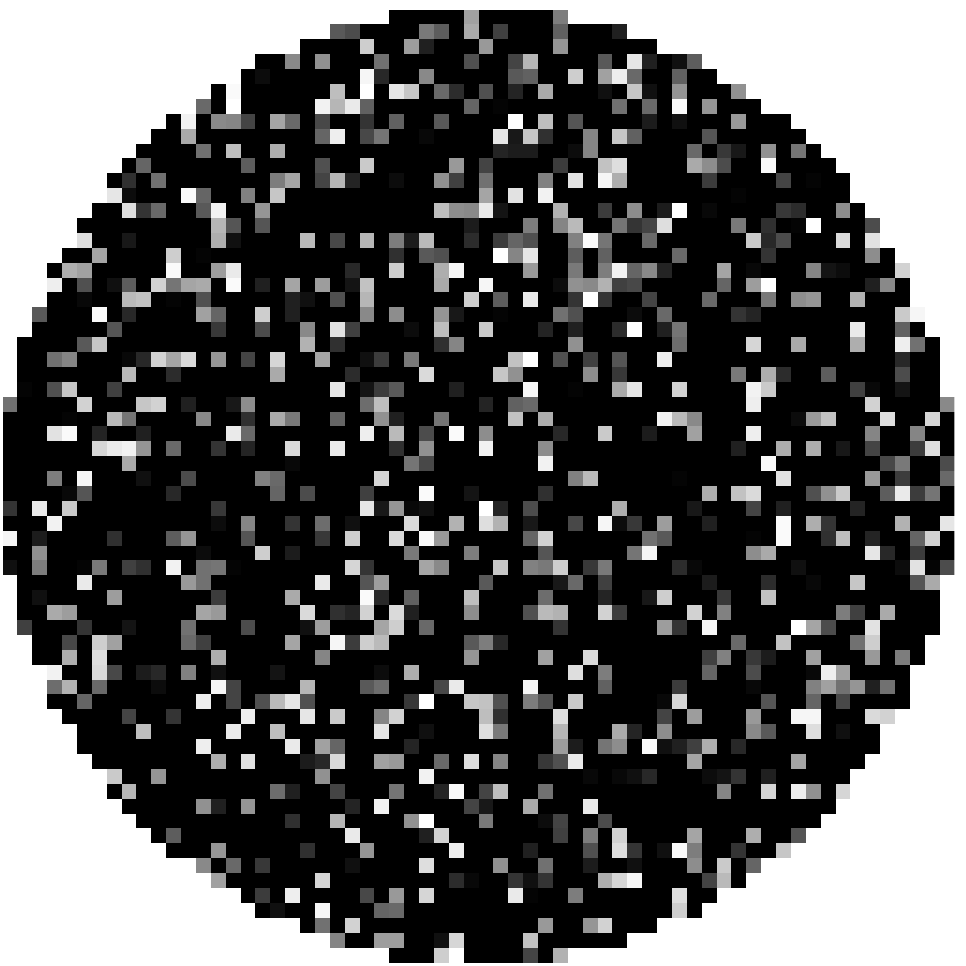}
\includegraphics[width=\ww\linewidth]{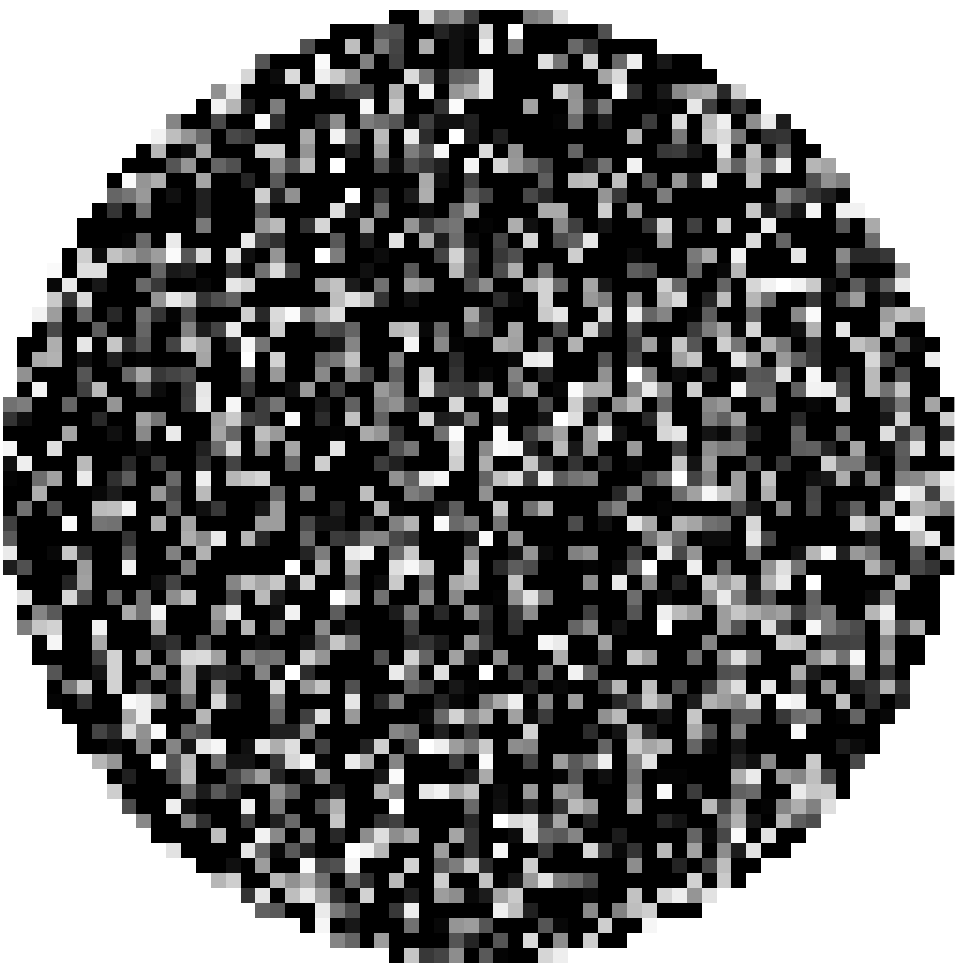}
\includegraphics[width=\ww\linewidth]{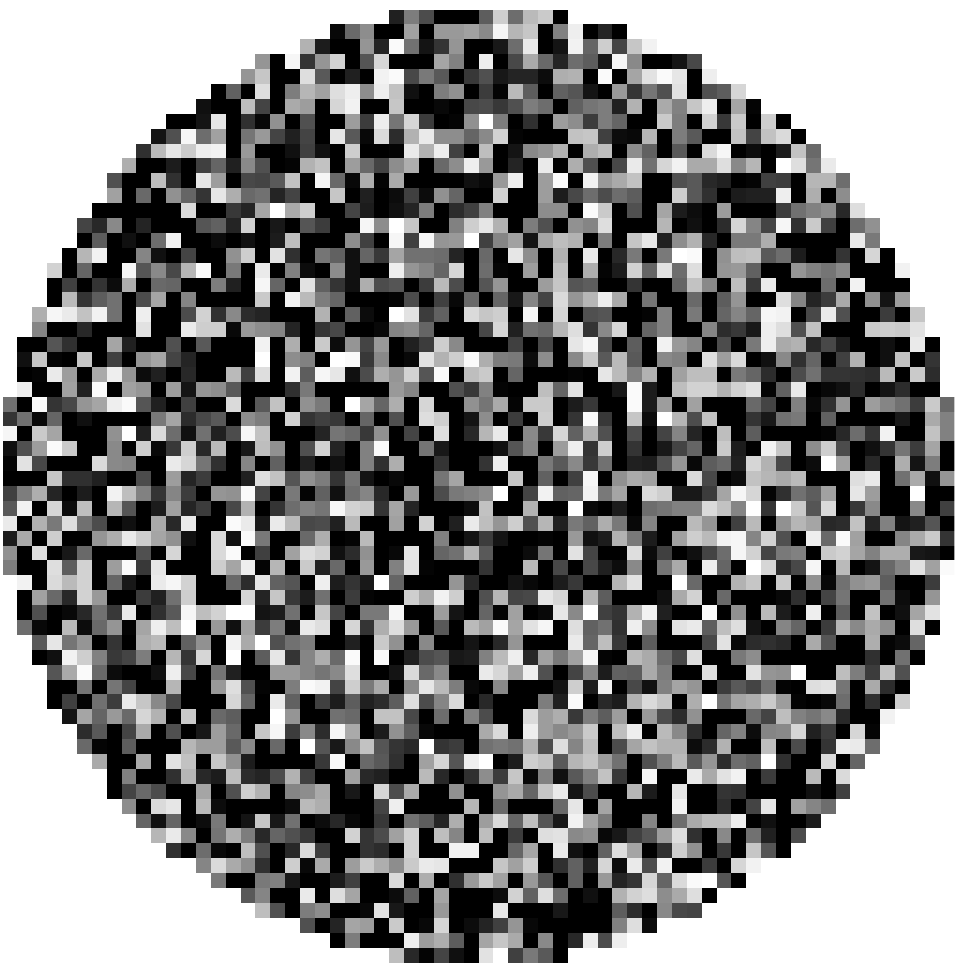}
\includegraphics[width=\ww\linewidth]{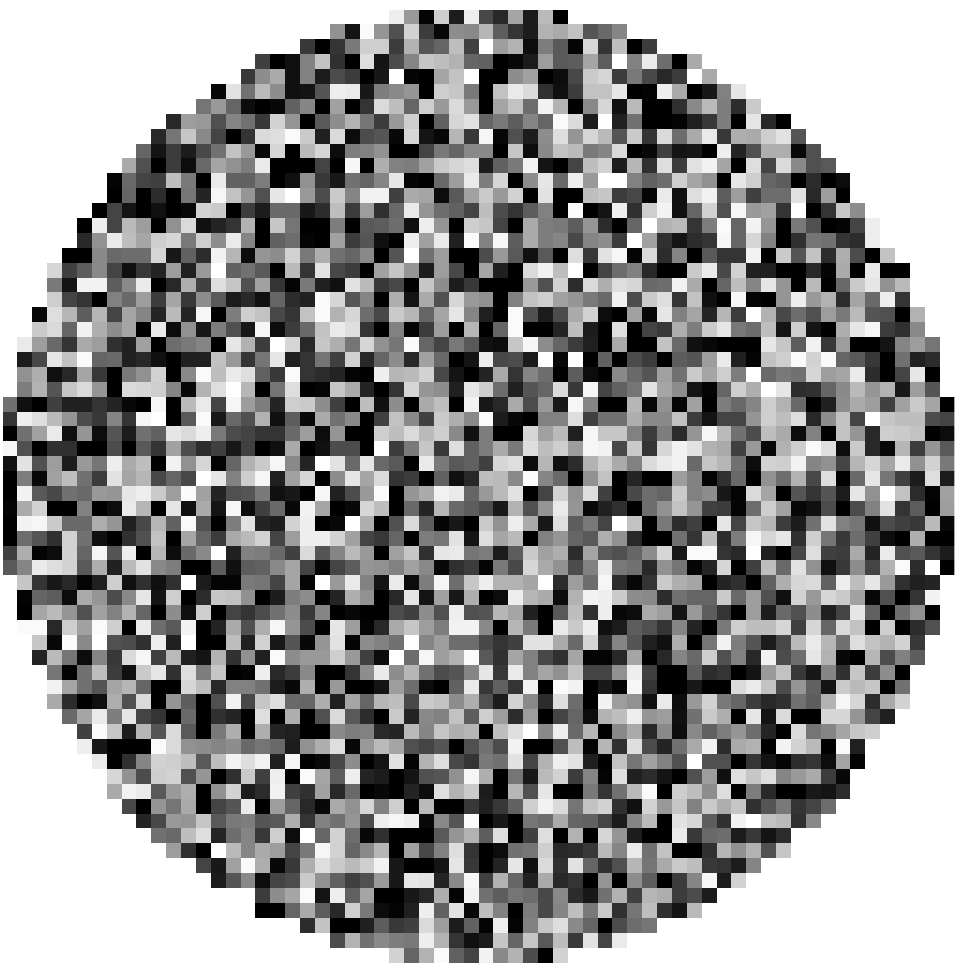}\\
\includegraphics[width=\ww\linewidth]{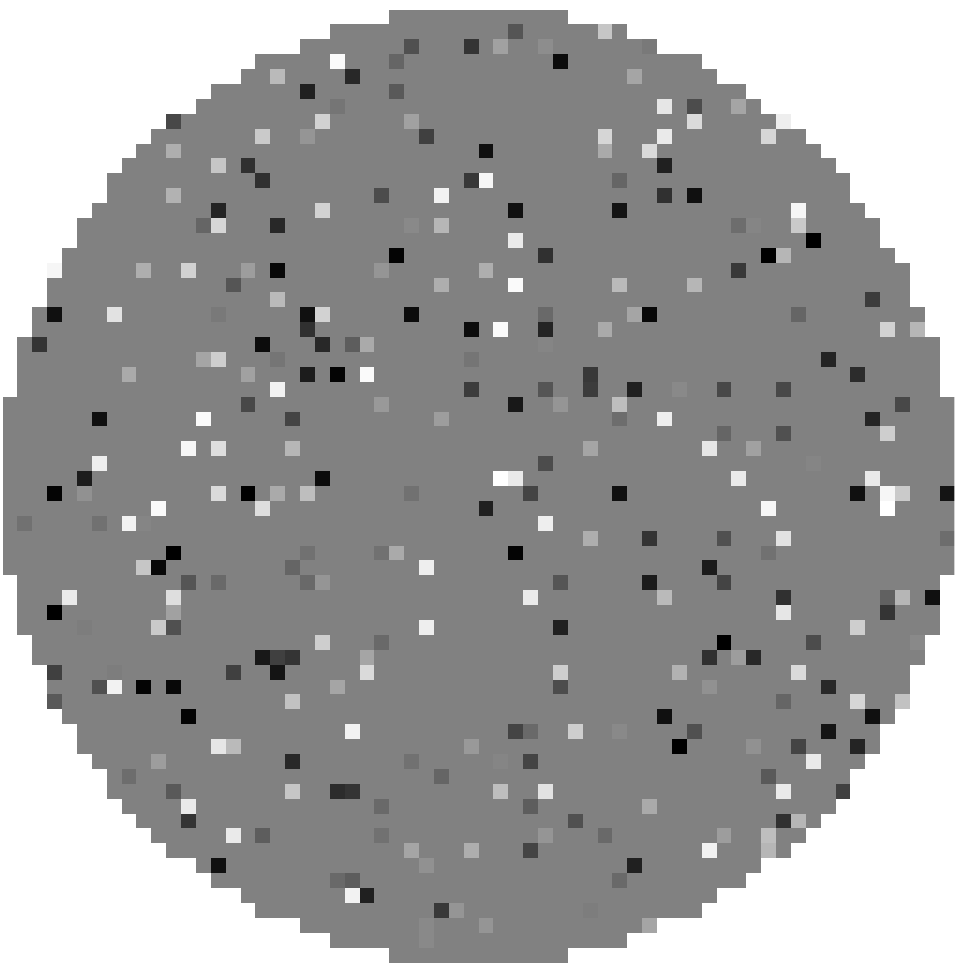}
\includegraphics[width=\ww\linewidth]{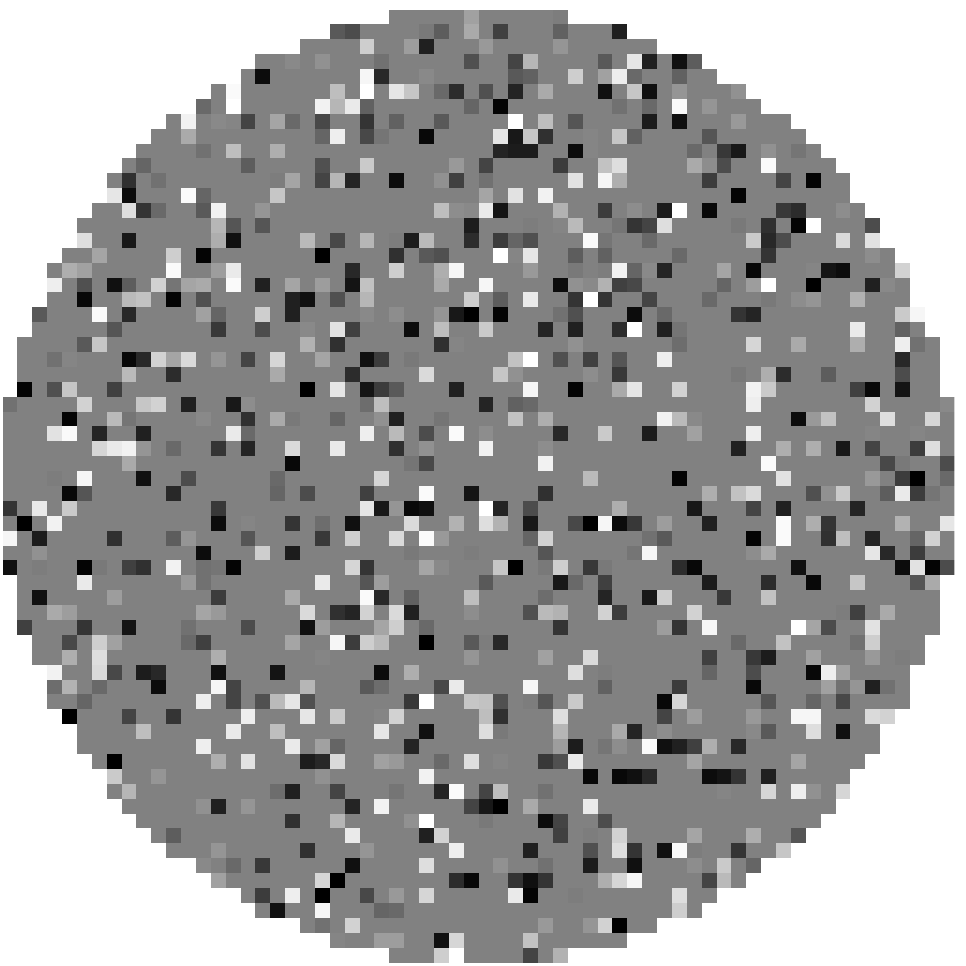}
\includegraphics[width=\ww\linewidth]{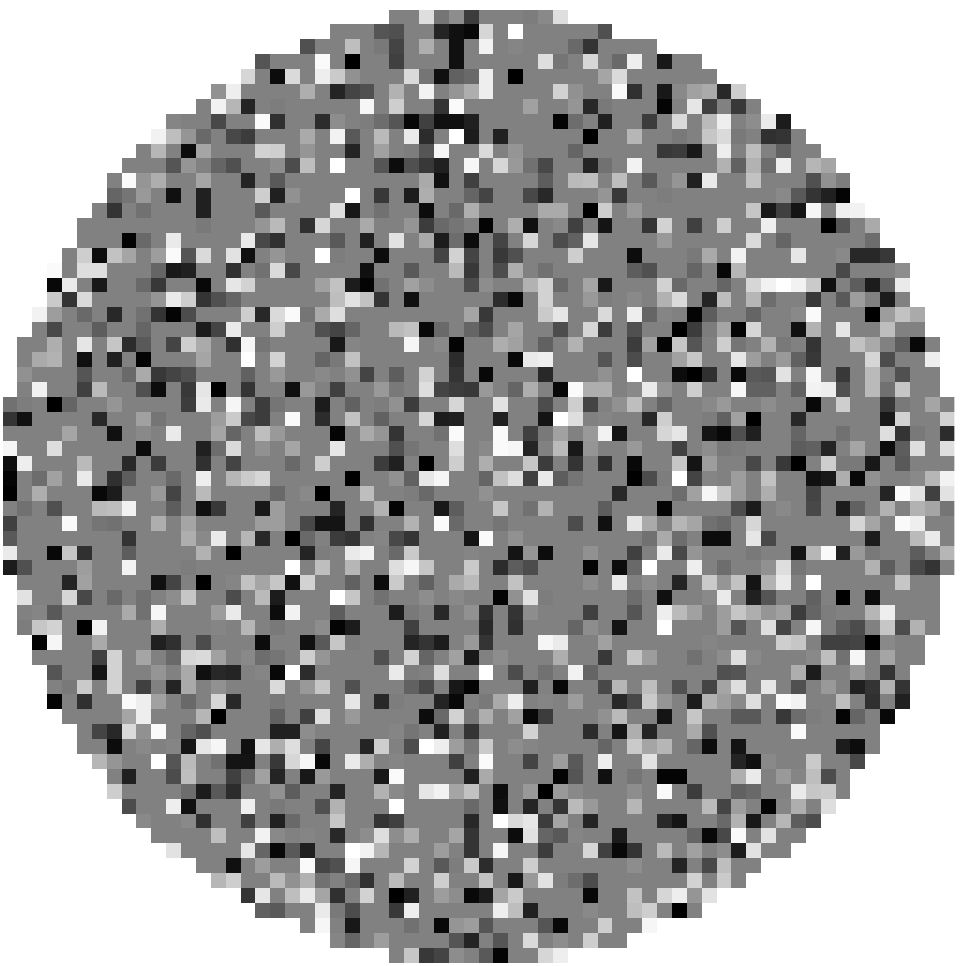}
\includegraphics[width=\ww\linewidth]{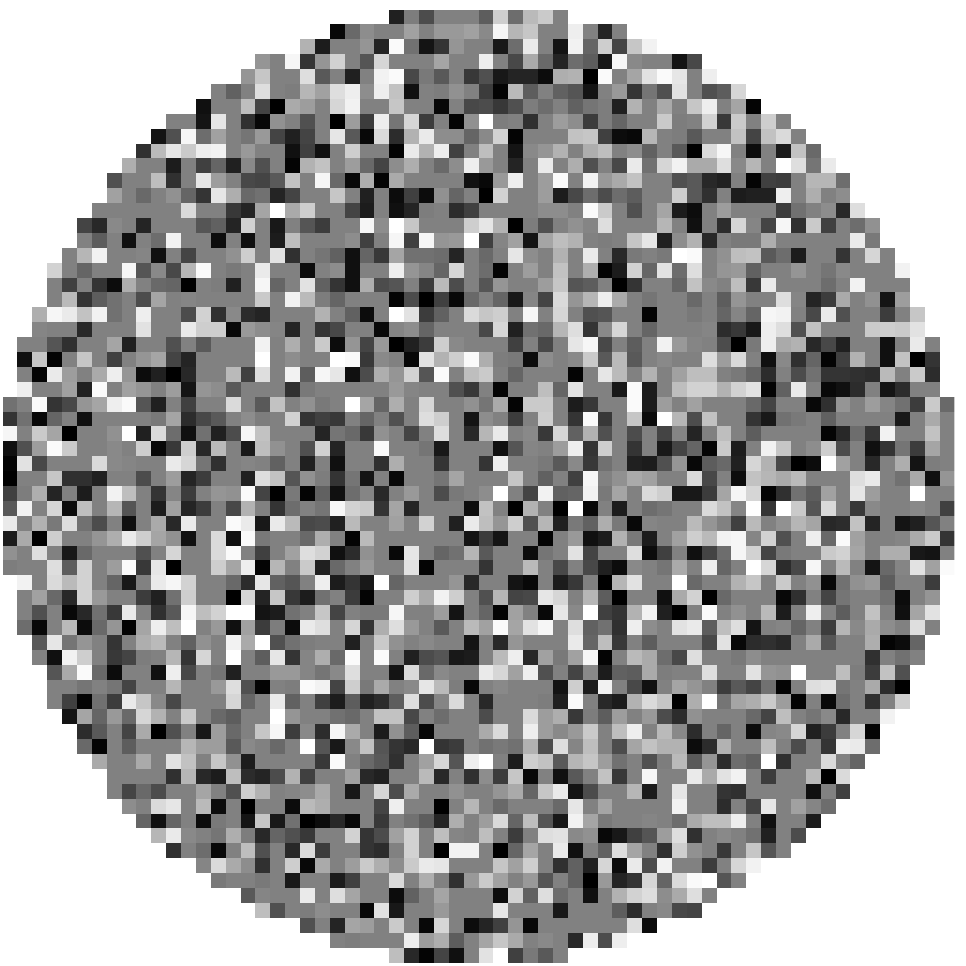}
\includegraphics[width=\ww\linewidth]{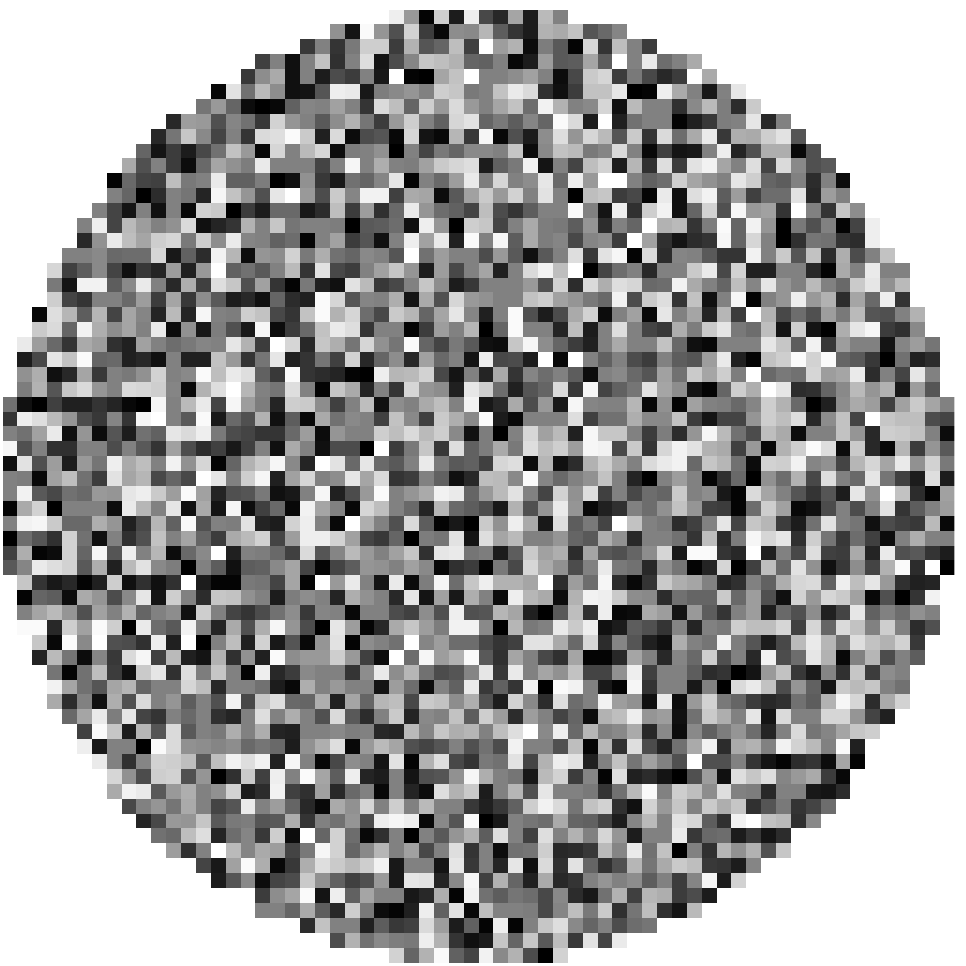}
\caption{Top: Realization of images from the \spikes{} class of relative sparsity $\kappa = k/n$ values $0.1$, $0.3$, $0.5$, $0.7$, $0.9$, gray-scale $[0,1]$. Bottom: The same for the \signedspikes{} class, gray-scale $[-1,1]$.\label{fig:spikesexamples}}
\end{figure}

\subsection{Images for \atv{}}
For \atv{} we wish to generate test images having a prescribed value of $\|\dmat{}^T\im{}\|_0$, where $\|\cdot\|_0$ is the cardinality. Due to the operator this is a less trivial task than in the directly sparse case. For certain operators this can be accomplished using a technique from \cite{Nam2013cosparse} but as pointed out in that work, the technique does not apply to the finite-difference operator $\dmat{}^T$. Here, we present two methods for this purpose. The first method, \trununif{}, produces test images that \emph{in expectation} achieve the target sparsity, while the second, \altproj, produces test images of precisely the target sparsity.

\begin{figure}[htb]
\newcommand{\ww}{0.19}
\includegraphics[width=\ww\linewidth]{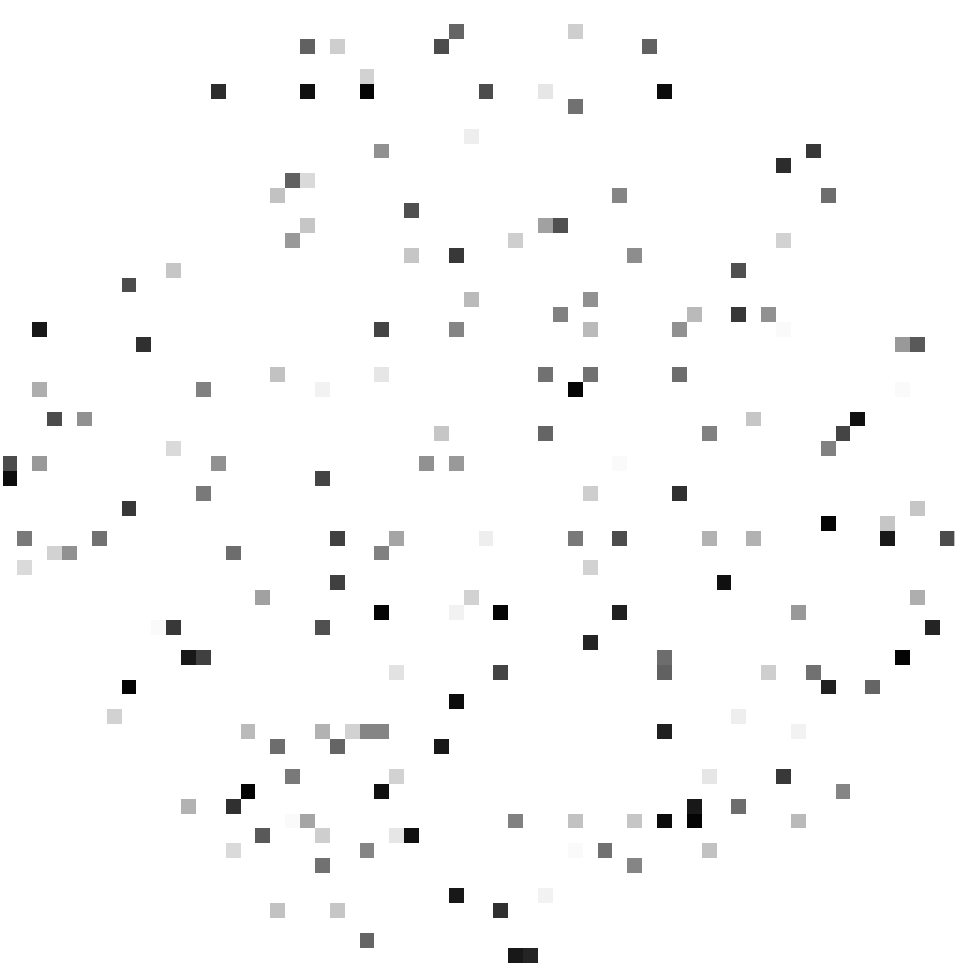}
\includegraphics[width=\ww\linewidth]{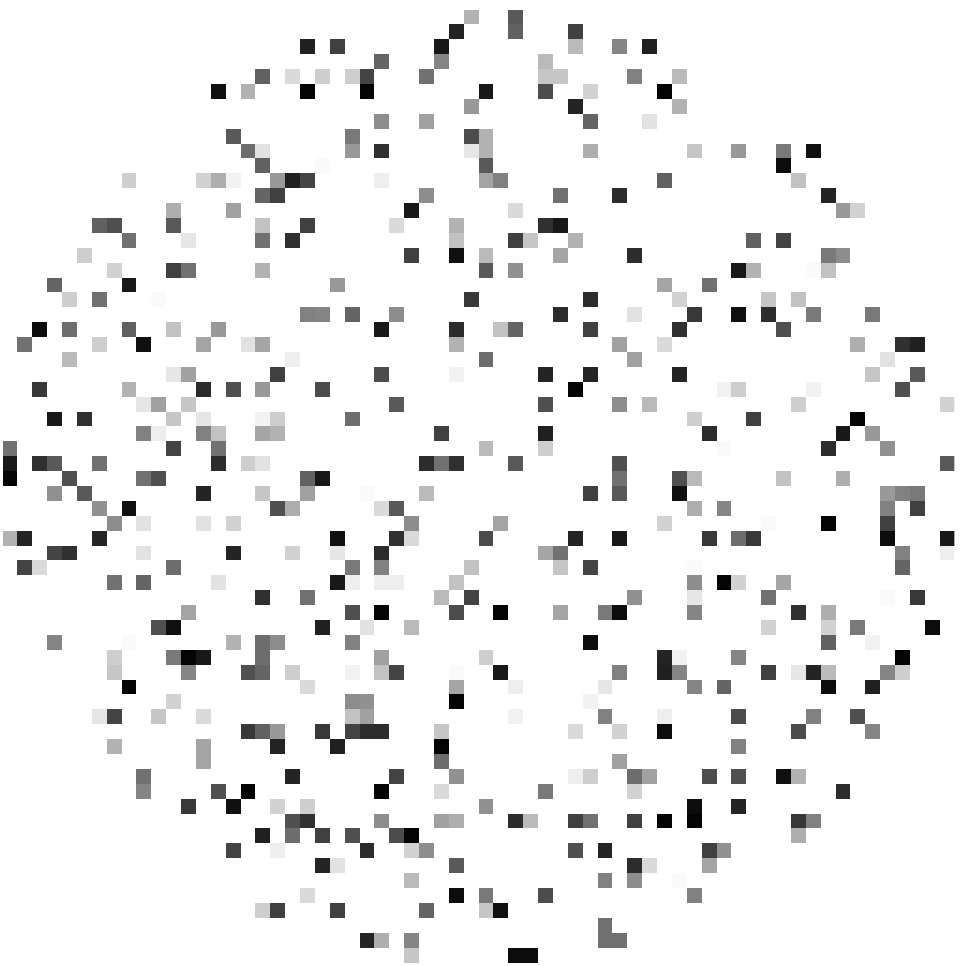}
\includegraphics[width=\ww\linewidth]{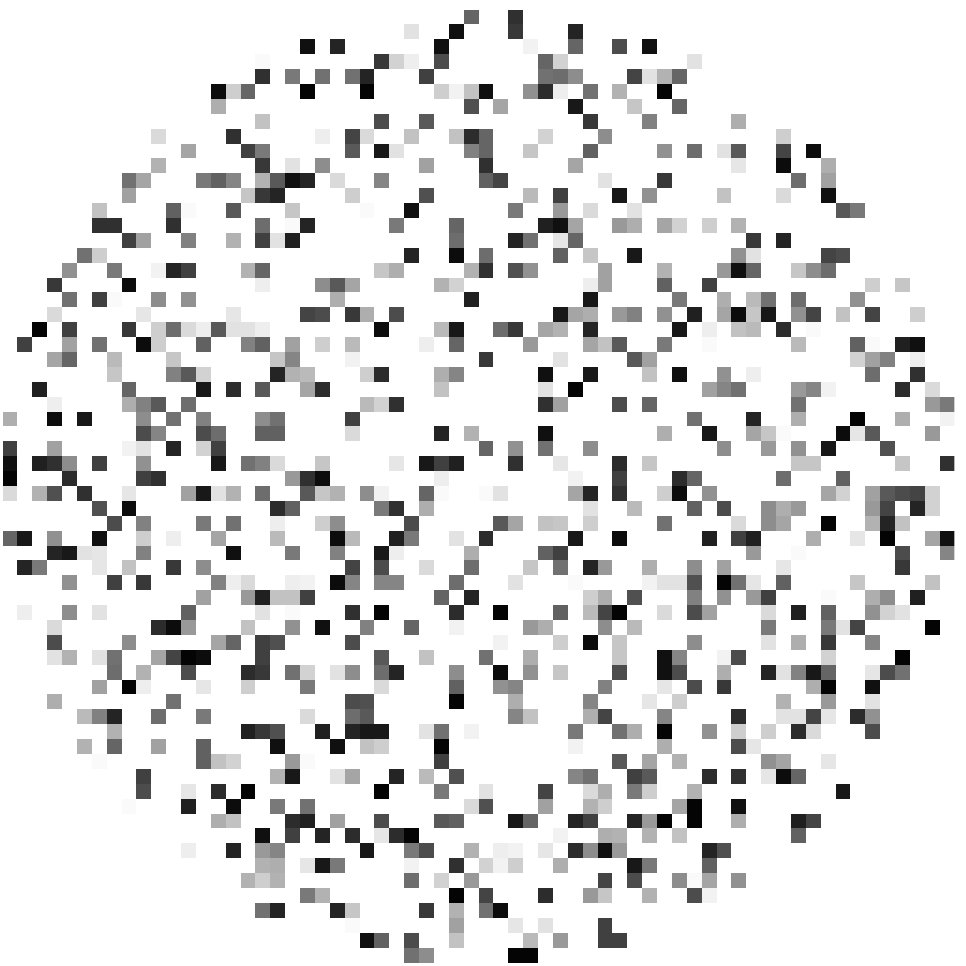}
\includegraphics[width=\ww\linewidth]{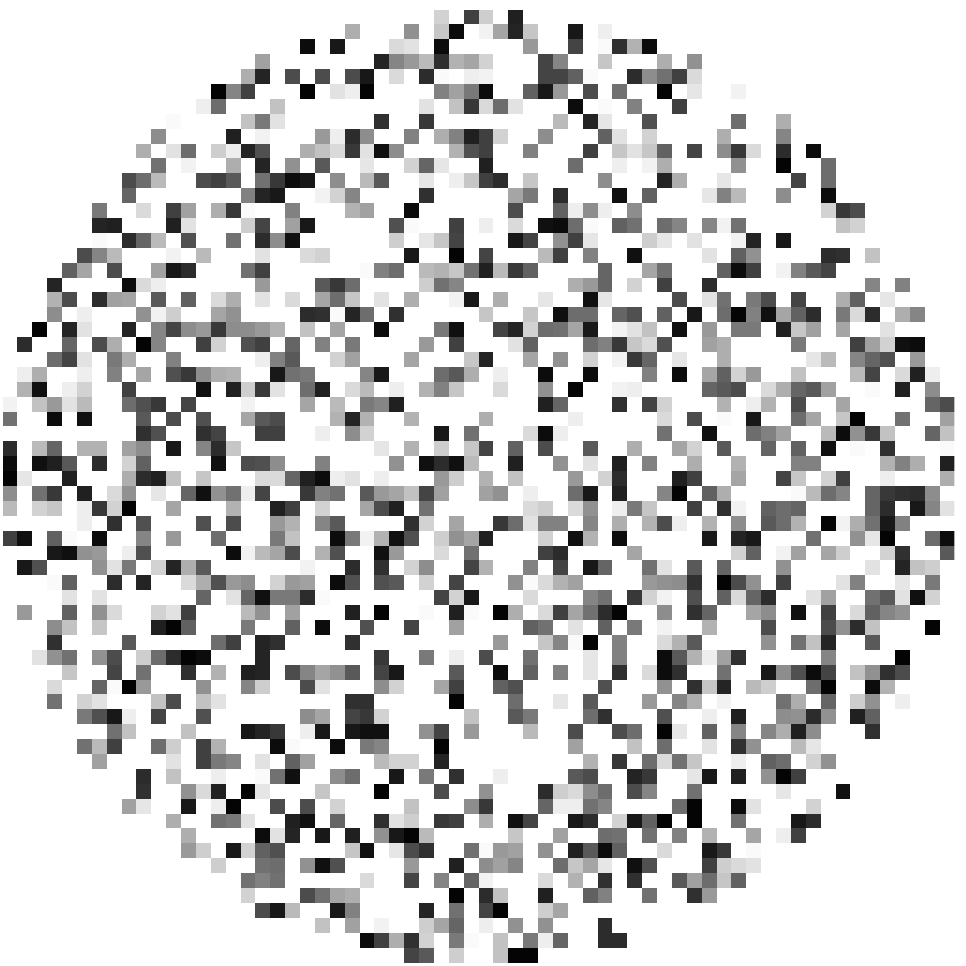}
\includegraphics[width=\ww\linewidth]{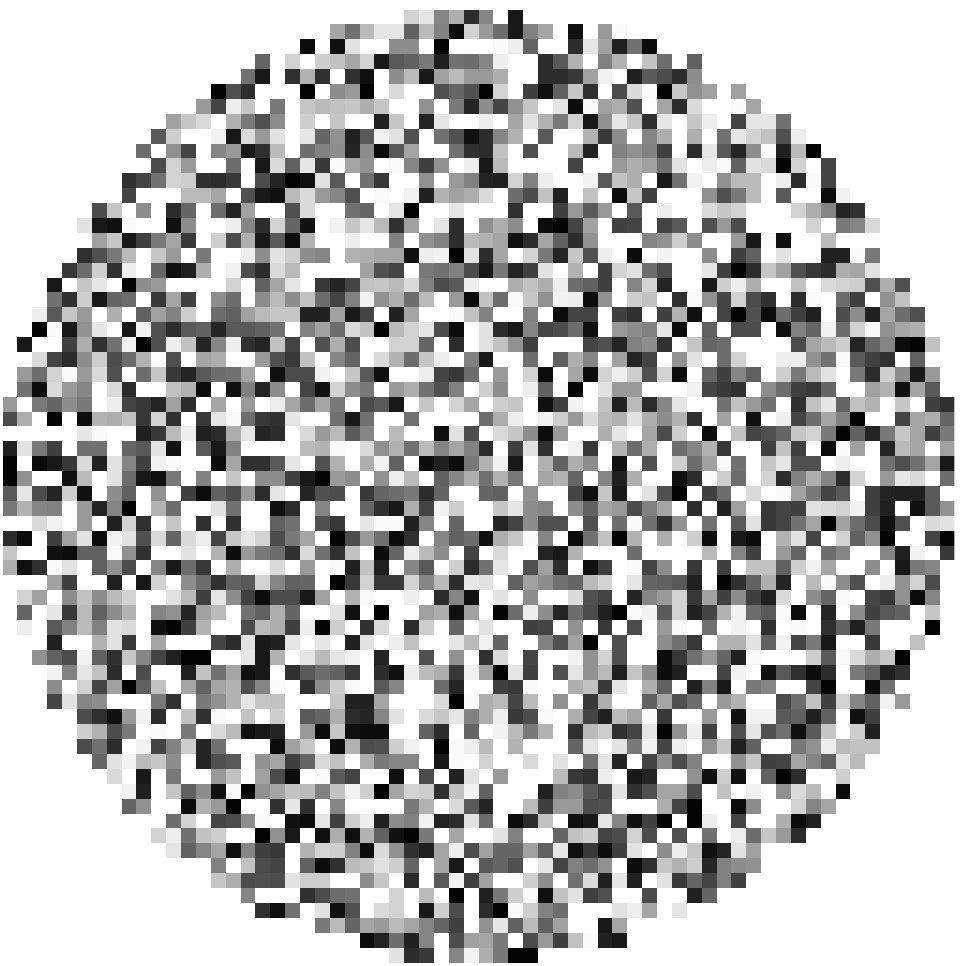}\\
\includegraphics[width=\ww\linewidth]{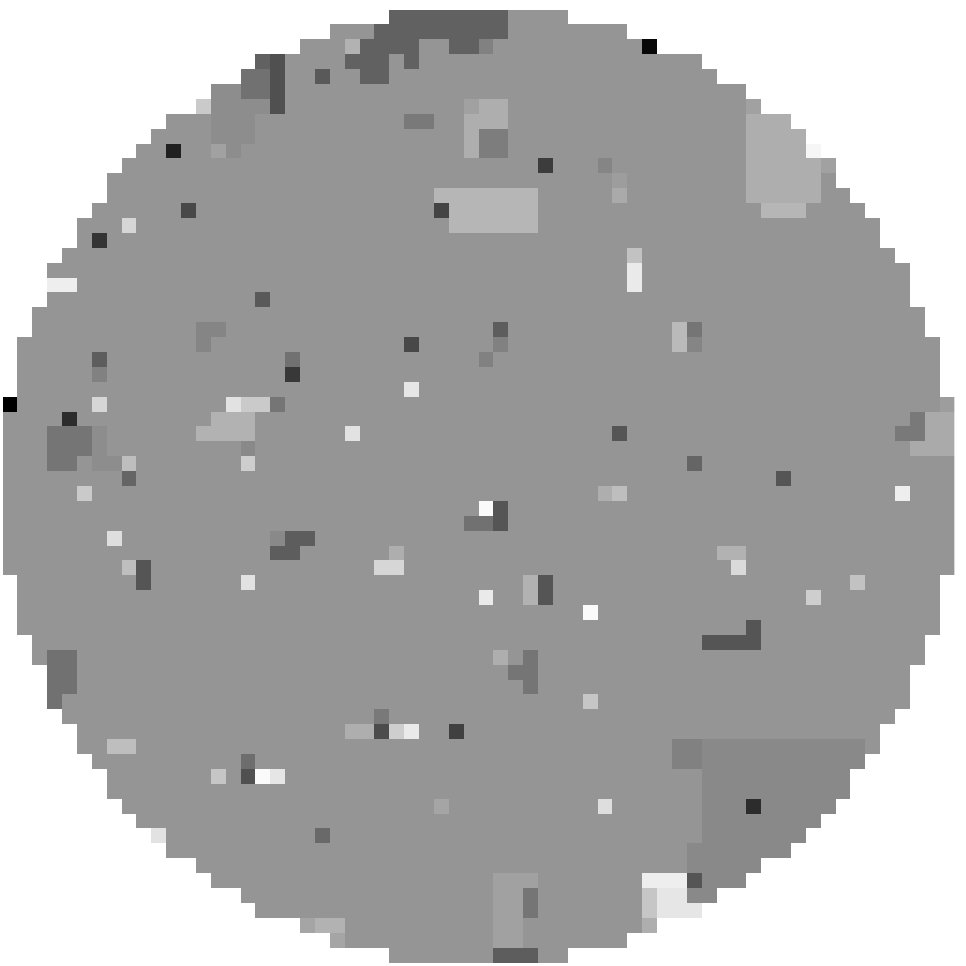}
\includegraphics[width=\ww\linewidth]{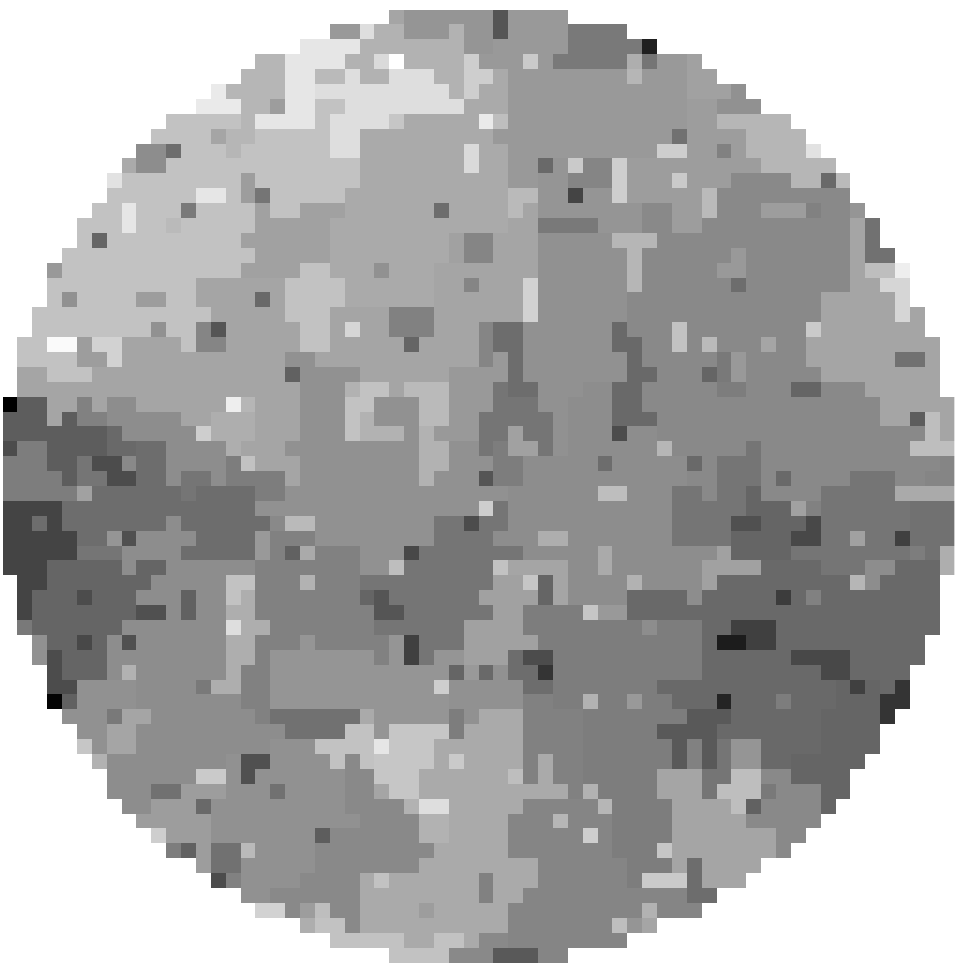}
\includegraphics[width=\ww\linewidth]{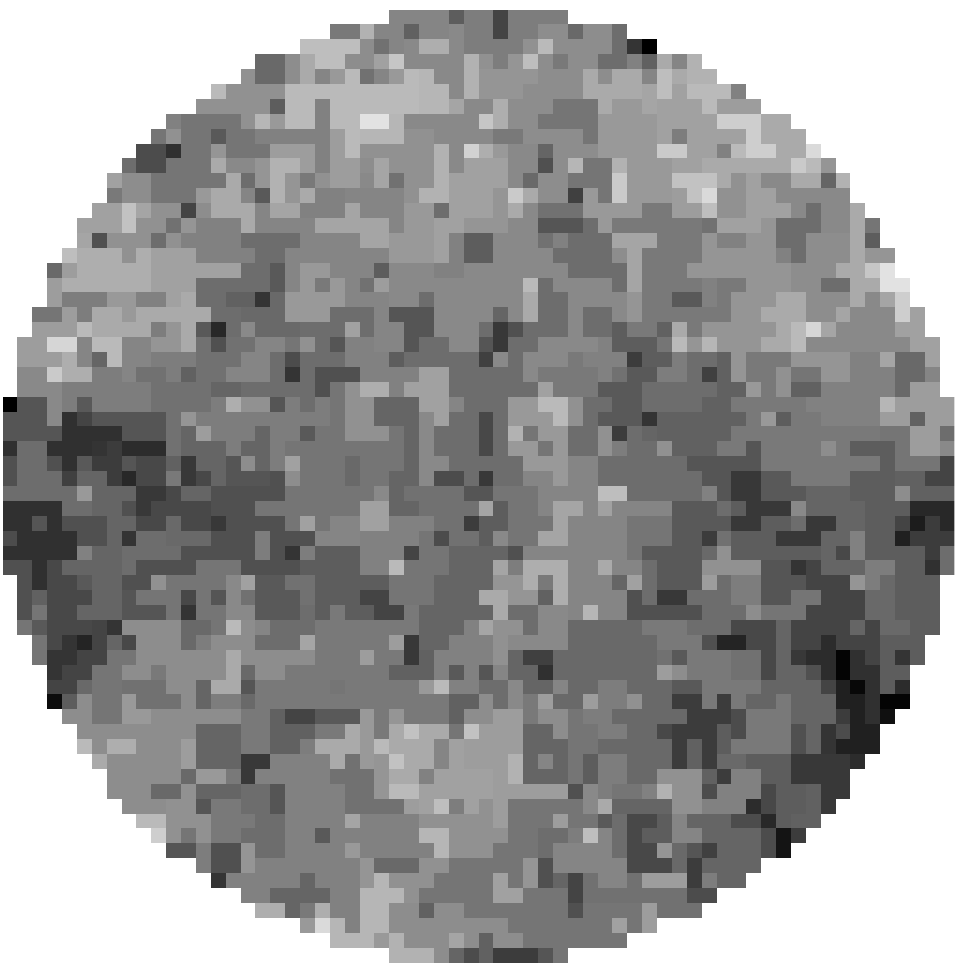}
\includegraphics[width=\ww\linewidth]{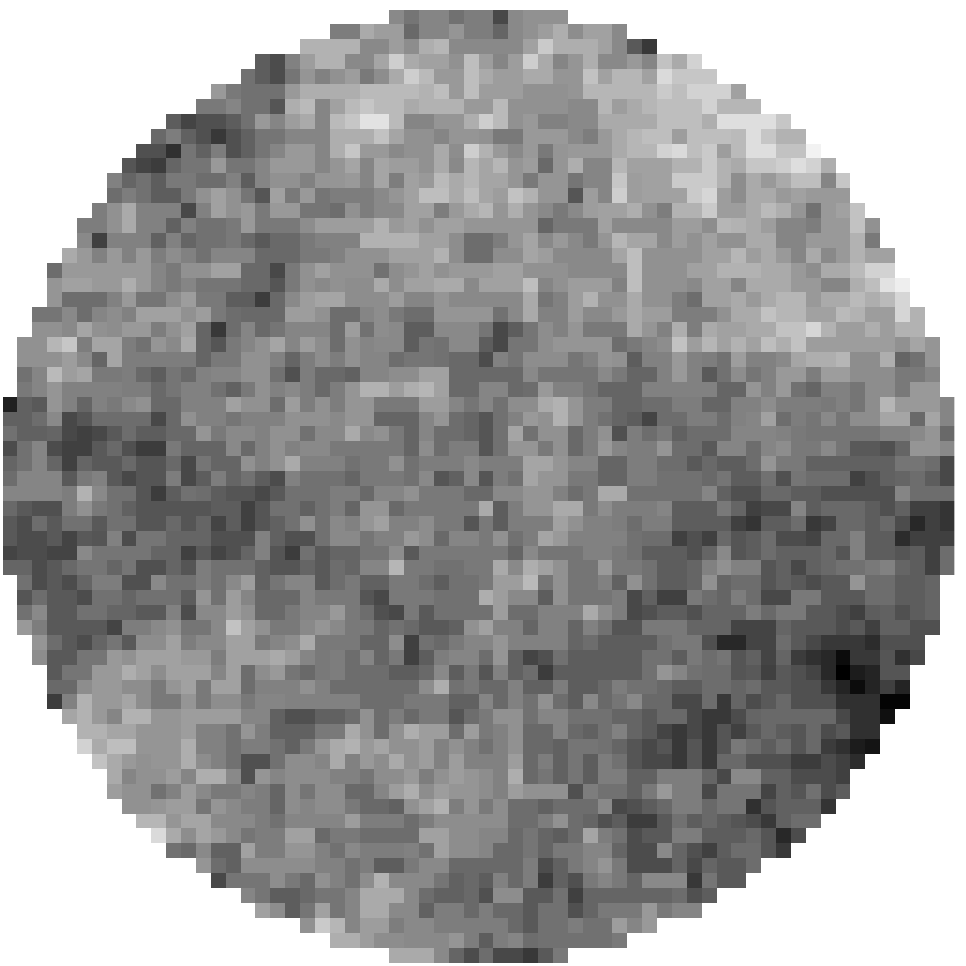}
\includegraphics[width=\ww\linewidth]{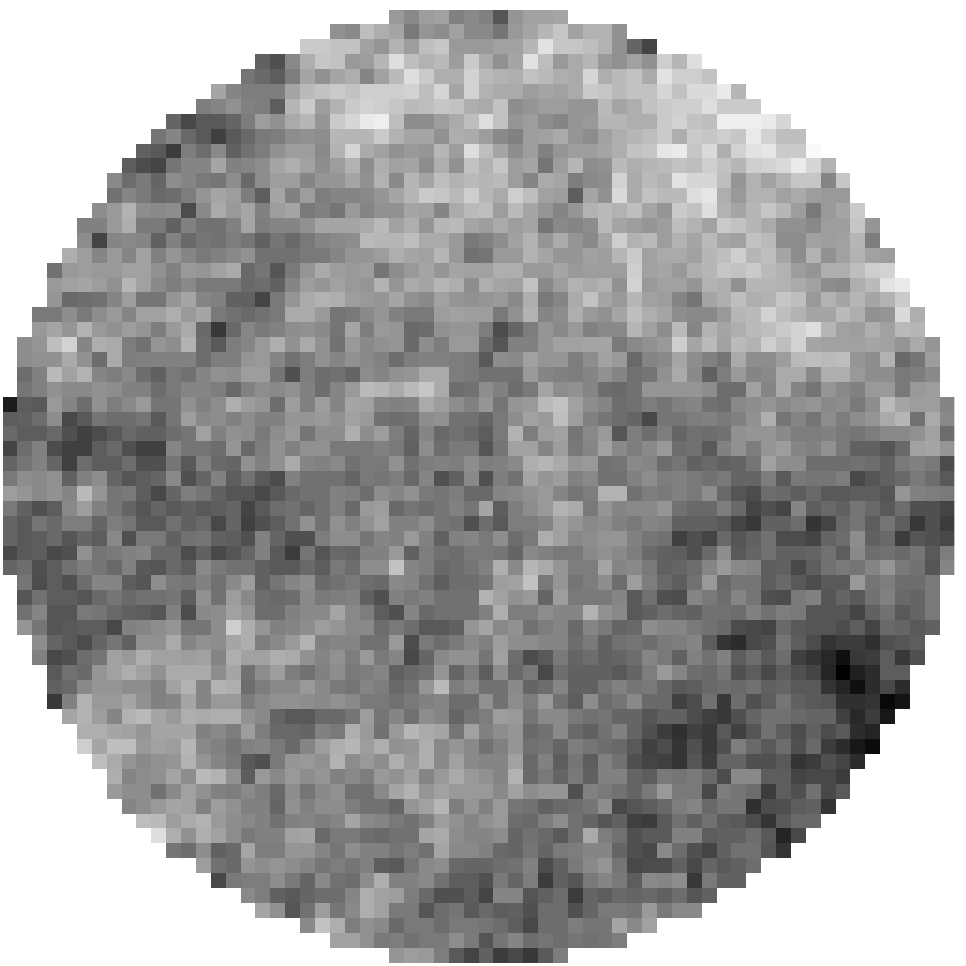}\\
\includegraphics[width=\ww\linewidth]{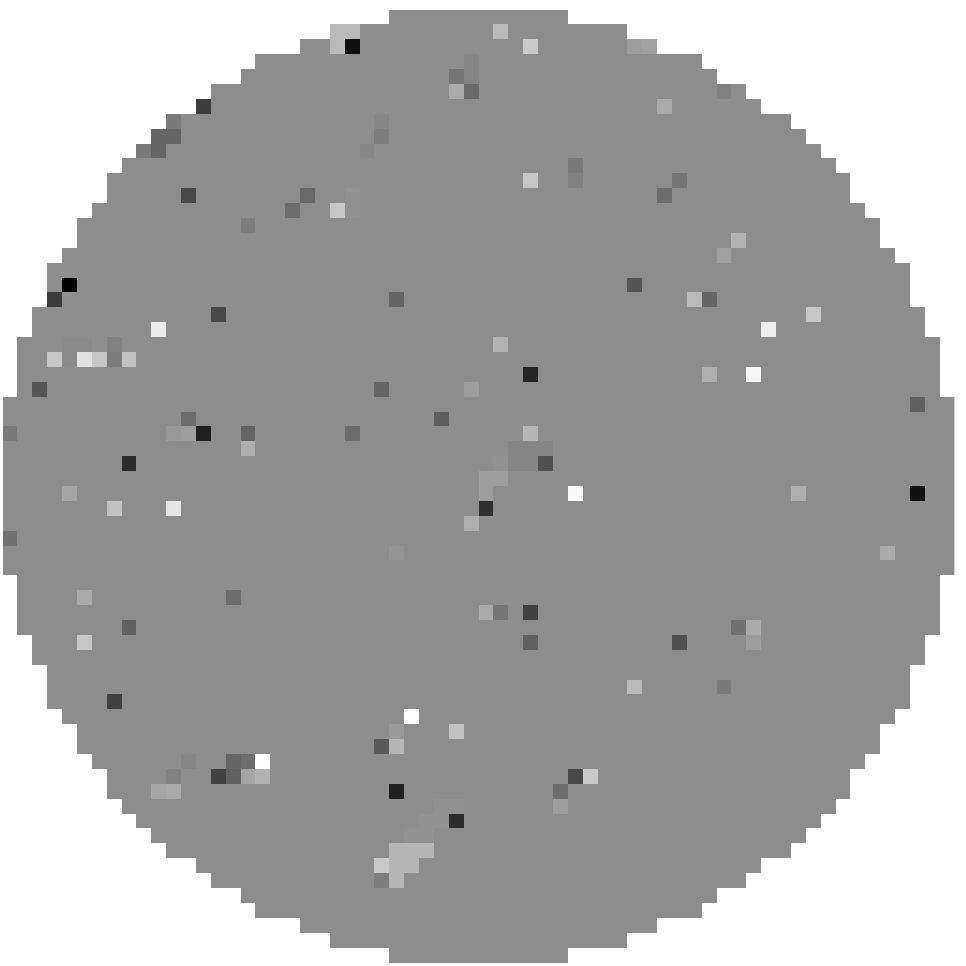}
\includegraphics[width=\ww\linewidth]{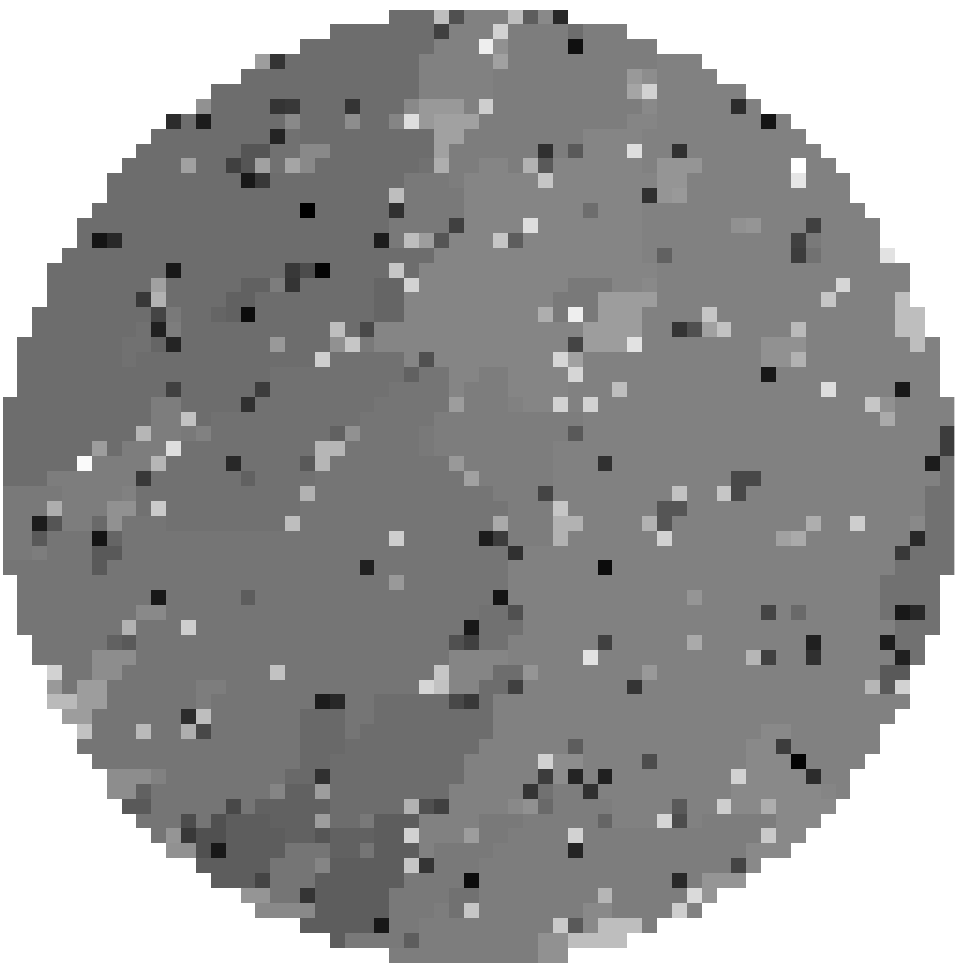}
\includegraphics[width=\ww\linewidth]{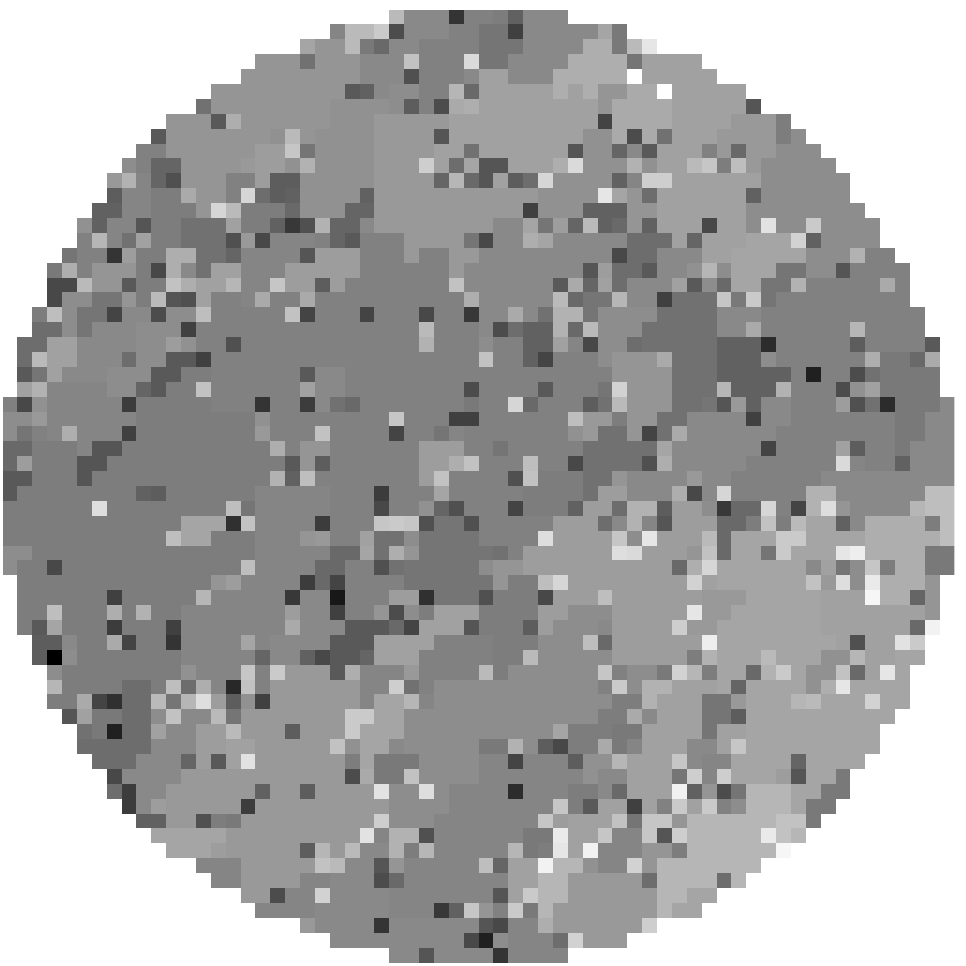}
\includegraphics[width=\ww\linewidth]{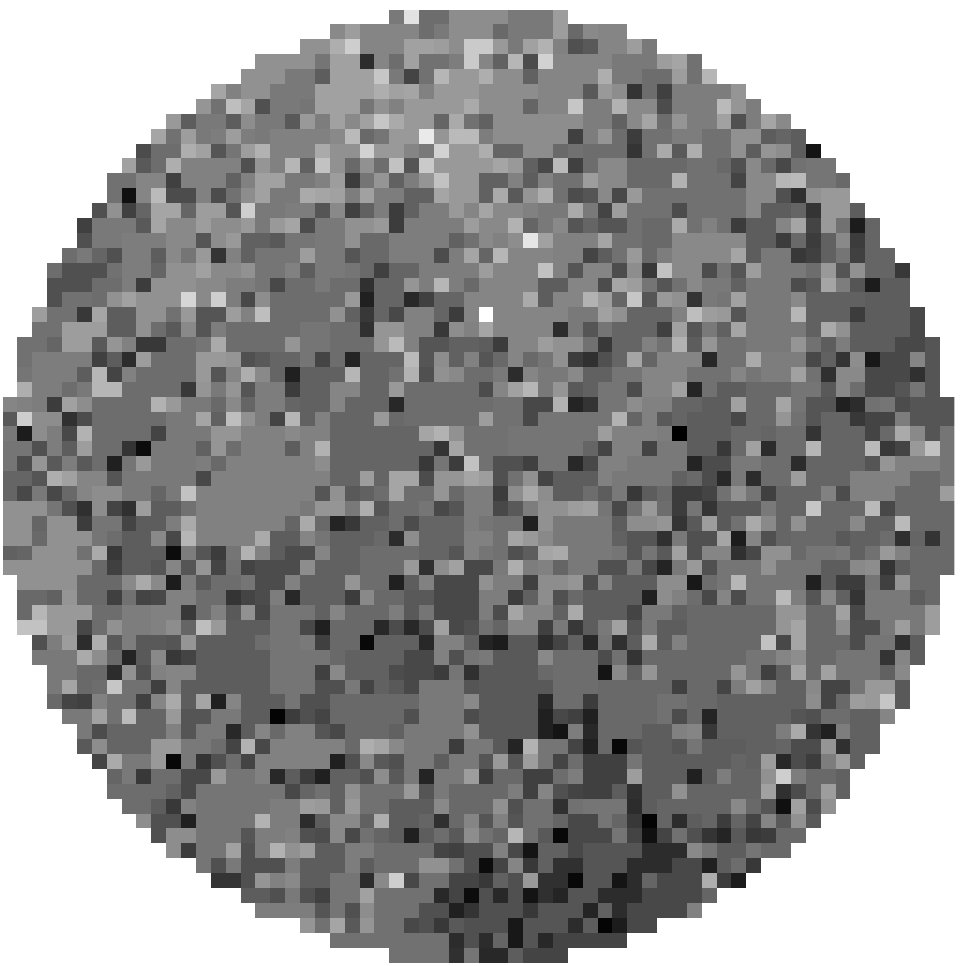}
\includegraphics[width=\ww\linewidth]{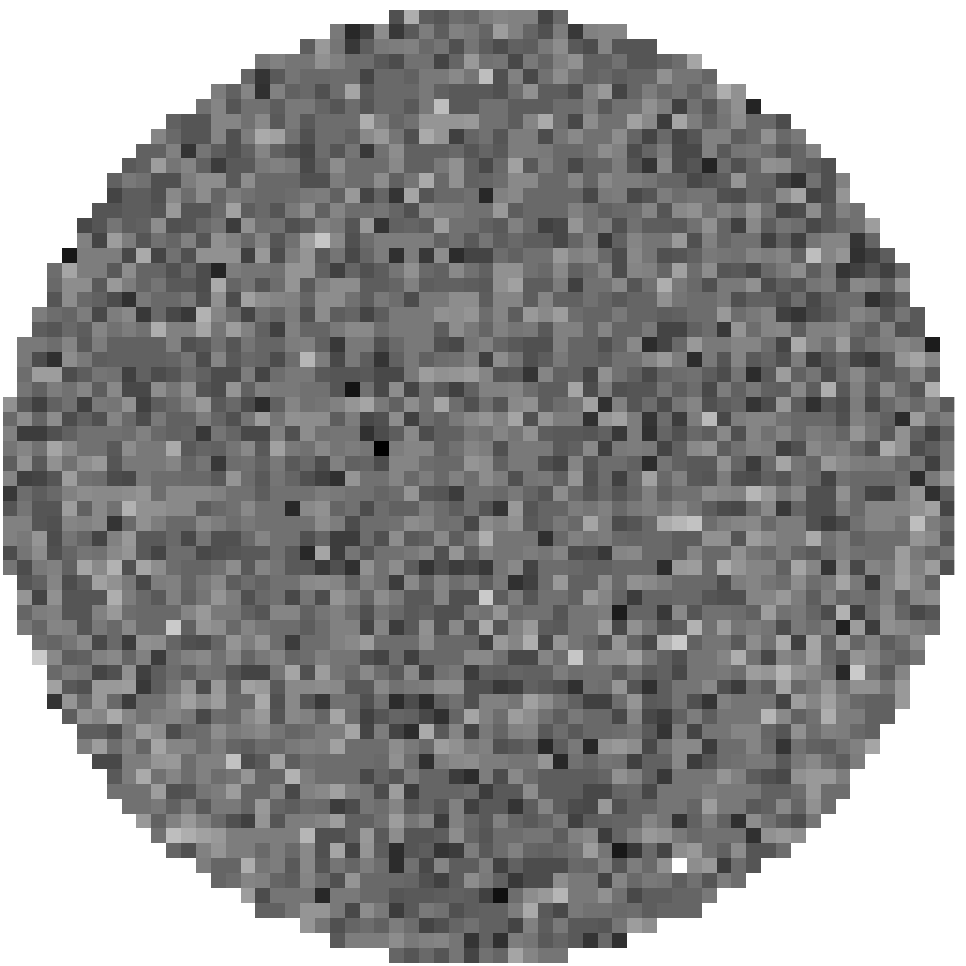}
\caption{Top: A set of \trununif{} images of relative sparsity $\kappa = k/n$ values $0.2$, $0.6$, $1.0$, $1.4$, $1.8$, gray-scale $[0,1]$. Middle: A set of `\altproj{} (anisotropic) images of relative sparsity values $0.2$, $0.6$, $1.0$, $1.4$, $1.8$, gray-scale $[-1,1]$. Bottom: A set of \altproj{} (isotropic) images of relative sparsity values $0.1$, $0.3$, $0.5$, $0.7$, $0.9$, gray-scale $[-1,1]$. \label{fig:trununifaltprojexamples}}
\end{figure}

\subsubsection{The \trununif{} class}
\label{sec:truncunif}
The \trununif{} class produces images $\im{}$
according to the following heuristic. Given a target number of nonzeros $k$ of
the length-$N$ vector $\dmat{}^T\im{}$, where $\dmat{}^T$ has $|B^c|$ rows that do not correspond to differences across the image boundary, and a number $F$, which is the number of gray values in the image, satisfying
\begin{equation}
 k \leq |B^c| \frac{F-1}{F}, \label{eq:maxk_trununif}
\end{equation}
do the following:
\begin{enumerate}
\item Compute 
$\omega = (1-\sqrt{1-kF/(|B^c|(F-1))})/F$.
\item Partition the interval $[0,1]$ into $F$ intervals $I_1,\dots,I_F$ where intervals $\ell=1,\dots, F-1$ have widths $\omega_\ell = \omega$ and the last one has width  $\omega_F = 1-(F-1)\omega$, i.e. the intervals are $I_\ell = [(\ell-1)\omega,\ell\omega[$ for $\ell=1,\dots,F-1$ and $I_F = [(F-1)\omega,1]$. Denote by $f_\ell$ the midpoint of the $\ell$-th interval. 
\item Assign for every pixel the value $f_\ell$ with probability $\omega_\ell$.
\end{enumerate}
The following lemma shows that the outcome of this method is an image that in expectation has $k$ nonzeros after application of $\dmat{}^T$ 
as desired:
\begin{lemma} \label{lem:trununif}
  Let $\im{}$ be generated by the four steps above and denote $\zvec{} = \dmat{}^T \im{}$.
  Then the expected value of the number of non-zero entries in $\zvec{}$ is $k$.
\end{lemma}
\begin{proof}
   With the described procedure, the $j$th entry of $\im{}$, $j=1,\dots,n$ is a scalar stochastic variable $\bm X_j$, and the corresponding vector stochastic variable $\bm X$ has independent, identically distributed (i.i.d.) elements. We consider also the vector stochastic variable $\zcapvec = \dmat{}^T \bm X$ and the vector stochastic indicator variable $\bm \delta$ with elements
  \begin{align*}
   \bm\delta_i = \left\{
   \begin{array}{rl}
 1 &\mbox{ if $\zcapvec{}_i \neq 0$,} \\
  0 &\mbox{ if $\zcapvec{}_i = 0$,}
       \end{array} \right. \qquad i = 1,\dots, N.
  \end{align*}
The total number of nonzeros in $\zcapvec{}$ is described by the scalar stochastic variable 
\begin{align*}
\delta_\text{total} = \sum_{i=1}^N \bm \delta_i, 
\end{align*}
of which we derive the expected value $E(\delta_\text{total})$. Using linearity of expectation and that 
\begin{align*}
 E(\delta_i) = 0 \cdot P(\zcapvec{}_i = 0) + 1 \cdot P(\zcapvec{}_i \neq 0) = P(\zcapvec{}_i \neq 0), \qquad i = 1,\dots,N,
\end{align*}
we get
\begin{align}
 E(\delta_\text{total}) = \sum_{i=1}^N P(\zcapvec{}_i \neq 0). \label{eq:EKdef}
\end{align}
To compute $P(\zcapvec{}_i \neq 0)$ we distinguish between two cases: $\zcapvec{}_i$ is a finite difference either 1) across the boundary or 2) in the interior of the image. Formally, we partition the indices $\{1,\dots,N\}$ into a boundary set $B$ and a complementary interior set $B^c$.

In case 1), $i\in B$, the choice of boundary conditions (BCs) affects the probability in question. Assuming Neumann BCs, each finite difference across the boundary is zero, i.e., $P(\zcapvec{}_i \neq 0) = 0 ~\text{for}~ i\in B$ 
corresponding to a zero row in $\dmat{}^T$ at indices $B$. Equivalently, these rows can be removed from $\dmat{}^T$, leaving $B$ empty. Assuming zero BCs instead, the probability is instead equal to $1$, since $\zcapvec{}_i$ can only be zero, if the pixel value in question is $0$, which happens with probability $0$. In this paper, we do not consider zero BCs more.

In case 2), we apply the law of total probability over the set
$\{f_1, \dots, f_F\}$,
\begin{align*}
P(\zcapvec{}_i \neq 0) = \sum_{\ell=1}^F P(\zcapvec{}_i \neq 0 \vert \bm X_{\tilde{i}} = f_\ell) P(\bm X_{\tilde{i}} = f_\ell),
\end{align*}
where $\tilde{i}$ denotes the index of the pixel at which $\zcapvec{}_i$ is evaluated.
From step (3) we know that the probability that a given pixel value is $f_\ell$ equals $\omega_\ell$. 
For the conditional probability, we note that $\zcapvec{}_i$ is computed as the difference between $X_{\tilde{i}}$ and a neighboring pixel's value. Since the pixel values are independent, and since $\bm X_{\tilde{i}}$ is given to be in the $\ell$th interval, the probability of having a nonzero $\zcapvec{}_i$ equals the probability of the neighboring pixel's value being outside interval $\ell$, which is $1-\omega_\ell$. Hence,
\begin{align*}
P(\zcapvec{}_i \neq 0) &= \sum_{\ell=1}^F (1-\omega_\ell)\omega_\ell 
= \sum_{\ell=1}^{F-1} (1-\omega)\omega + (1-\omega_F)\omega_F \\
&= (F-1)\omega(2-F\omega) \quad \text{for} \quad i \in B^c.
\end{align*}
Inserting the two cases into \eqref{eq:EKdef} yields
\begin{align*}
 E(\delta_\text{total}) = |B^c| (F-1)\omega(2-F\omega).
\end{align*}
Since we want $E(\delta_\text{total})=k$, we solve this quadratic equation for $\omega$. The smaller of the two solutions guarantees that the sum of all widths is not greater than 1 and is precisely $\omega$ from step (1). Further, 
for $\omega$ to be real-valued we require \eqref{eq:maxk_trununif}.
\end{proof}
 For the case $N_\text{side} = 64$, we get $n = 3228$ pixels within the disk-shaped mask. For \atv{} with Neumann BCs and keeping zero-rows of $\dmat{}^T$ we have $N = 2n$. Due to convexity of the disk-shaped mask we can explicitly compute $|B^c| = 2n - 2N_\text{side}$. In our numerical studies we wish to study images sampled from the entire sparsity range between $\kappa =0$ and $\kappa = 2$ with the maximal relative sparsity of $1.9$. By taking $F=40$ we can achieve images $\im{}$ with sparsity of $k = 6169$ of $\dmat{}^T \im{}$. This corresponds to a relative sparsity of $\kappa = k/n = 1.911$.
Examples of \trununif{} images are shown in Fig. \ref{fig:trununifaltprojexamples} for a range of relative sparsity values.

\subsubsection{The \altproj{} class}
\label{sec:altproj}

Since the \trununif{} class consists of images of a special
structure (namely, they have a prescribed number of different gray
levels) it may be that they also introduce a special behavior in the
recoverability. Hence, we will consider also a different class of test
images. Our goal is again to produce images $\im{}$ 
such
that $\|\dmat{}^T\im{}\|_0 = k$. 
We reformulate the problem as
follows: Find a vector $\vvec{}$ such that $\vvec{}$ is in the range of $\dmat{}^T$ and
that $\|\vvec{}\|_0=k$. If we have found such a $\vvec{}$, then we get a suitable
test image $\im{}$ by solving $\dmat{}^T\im{}=\vvec{}$. 
For a method to construct such a $\vvec{}$, we are inspired by the feasibility problem
\[
\text{find}\ \vvec{}\in \rg \dmat{}^T\cap \{\|\cdot\|_0\leq k\}.
\]
Although the set we are looking at is the intersection of a convex
with a \emph{non-convex} one, recent results indicate that an
alternating projection approach may work~\cite{hesse2013nonconvex}. Hence, we
perform the following iteration:
\begin{enumerate}
\item Choose a random starting point $\vvec{}^0 \in \RR^N$; set $j=0$.
\item Set $\vvec{}^{j+\frac12}$ as orthogonal projection of $\vvec{}^{j}$ onto $\rg
  \dmat{}^T$. With the help of the pseudoinverse $(\dmat{}^T)^\dag$ of $\dmat{}^T$ this is
  written as $\vvec{}^{j+\frac12} = \dmat{}^T(\dmat{}^T)^\dag \vvec{}^{j}$.
\item Set $\vvec{}^{j+1}$ as orthogonal projection of $\vvec{}^{j+\frac12}$ onto
  the set $\{\|\cdot\|_0\leq k\}$: Keep the
  largest $k$ entries of $\vvec{}^{j+\frac12}$ and set the rest to zero. 
  If the projection yields fewer than $k$ nonzeros, project $\vvec{}^{j+\frac12}$ on a set with higher sparsity.
  \item If converged, set $\im{} = (\dmat{}^T)^\dag \vvec{}^{j+1}$; otherwise increment $j$ and go to step 2.
\end{enumerate}

If the method converges, it is guaranteed to produce an image $\im{}$
with the desired properties. However, in practice it does not always
convergence. Hence, we perform a maximum number of iterations (in
the range of a few thousands) and if we do not observe convergence to
a feasible point $\vvec{}$, we restart the method with a different initial
point. Typically, we found that only a few restarts sufficed for producing a
desired image.

We also consider a non-negative version, \altprojnonneg{}, for which an image is generated from an \altproj{} image by shifting all pixel values by the smallest possible positive scalar such that all pixel values become non-negative.

\subsection{Images for \itv{}}
For isotropic TV we basically proceeded similarly to \altproj{}. However, here we
considered the feasibility problem
\[
\text{find}\ \bm Y\in \rg \dlinop^T\cap \{\|\cdot\|_{0,2}\leq k\},
\]
where $\|\cdot \|_{0,2}$ counts the number of rows with nonzero $\ell^2$-norm. Note that the latter set are the images $\im{}$ for which the Euclidean
norm of the gradient has only $k$ non-zero entries. Hence, we modify the iteration to:
\begin{enumerate}
\item Choose a random starting point $\bm Y^0 \in \RR^{r\times{}p}$; set $j=0$.
\item Set $\bm Y^{j+\frac12}$ as orthogonal projection of $\bm Y^{j}$ onto $\rg
  \dlinop^T$. With the help of the pseudoinverse $(\dlinop^T)^\dag$ of $\dlinop^T$ this is
  written as $\bm Y^{j+\frac12} = \dlinop^T(\dlinop^T)^\dag \bm Y^{j}$.
\item Set $\bm Y^{j+1}$ as orthogonal projection of $\bm Y^{j+\frac12}$
  onto the set $\{\|\cdot\|_{0,2}\leq k\}$: Keep the $k$ rows of $\bm Y^{j+\frac12}$ with largest 2-norm 
  and set the rest to zero.
  If the projection yields fewer than $k$ nonzero rows, project $\bm Y^{j+\frac12}$ on a set with higher sparsity.
\item If converged, set $\im{} = (\dlinop^T)^\dag \bm Y^{j+1}$; otherwise increment $j$ and go to step 2.
\end{enumerate}

Examples of both anisotropic and isotropic \altproj{} images are shown in Fig. \ref{fig:trununifaltprojexamples} for a range of relative sparsity values. The relative sparsity values for anisotropic are chosen as twice the isotropic ones to enable a rough comparison between images from each class of `comparable' sparsity.

\section{Numerical experiments} \label{sec:results}
As in \cite{Joergensen_eqconpap_v2_arxiv:2014} we wish to show empirically that CT image reconstruction by sparsity-exploiting methods admit sharp phase transitions as known from compressed sensing. The results in \cite{Joergensen_eqconpap_v2_arxiv:2014} only covered \ellone{}. Here, we extend to \atv{} and \itv{} and construct phase diagrams by solving the reconstruction problem as well as the uniqueness tests.

\subsection{Phase diagrams for \ellone{}}
In \cite{Joergensen_eqconpap_v2_arxiv:2014} it was found that reconstruction by $\ell_1$-minimization for \spikes{} images yields a sharp phase transition. Here, through uniqueness testing we verify this result 
and further extend to the class \signedspikes{} with signed entries. 

We consider images of size $N_\text{side} = 64$ and at each relative sparsity value $\kappa = k/n = 0.025, 0.05$, $0.1$,$0.2$,$0.3,\dots$,$0.9$ we generate $100$ instances. For each instance $\im{}^*$, we generate synthetic CT data $\sino{} = \sysmat{}\im{}^*$ corresponding to $N_\text{v} = 1,...,32$ projection views. At $N_\text{v} \leq 25$ the linear system is underdetermined, while at $N_\text{v} \geq 26$ the system matrix $\sysmat{}$ has full rank and $\im{}^*$ is the unique solution no matter its sparsity. We therefore use $N_\text{v}^\text{suf} = 26$ as a reference point of full sampling at $N_\text{side} = 64$ and define the relative sampling $\mu = N_\text{v}/N_\text{v}^\text{suf}$. For each data set, reconstruction and uniqueness test are run. If the relative error 
 of the computed solution $\im{}_\ellone{}$ w.r.t. $\im{}^*$ is sufficiently small, i.e., $\|\im{}_\ellone{} - \im{}^*\|_2/\|\im{}^*\|_2 < \epsilon$ with $\epsilon = 10^{-4}$, we declare the original perfectly recovered. 

\begin{figure}[tb]
 \newcommand{\wq}{0.35}
\includegraphics[width=\wq\linewidth,trim=0cm 0cm 0cm 0cm, clip]{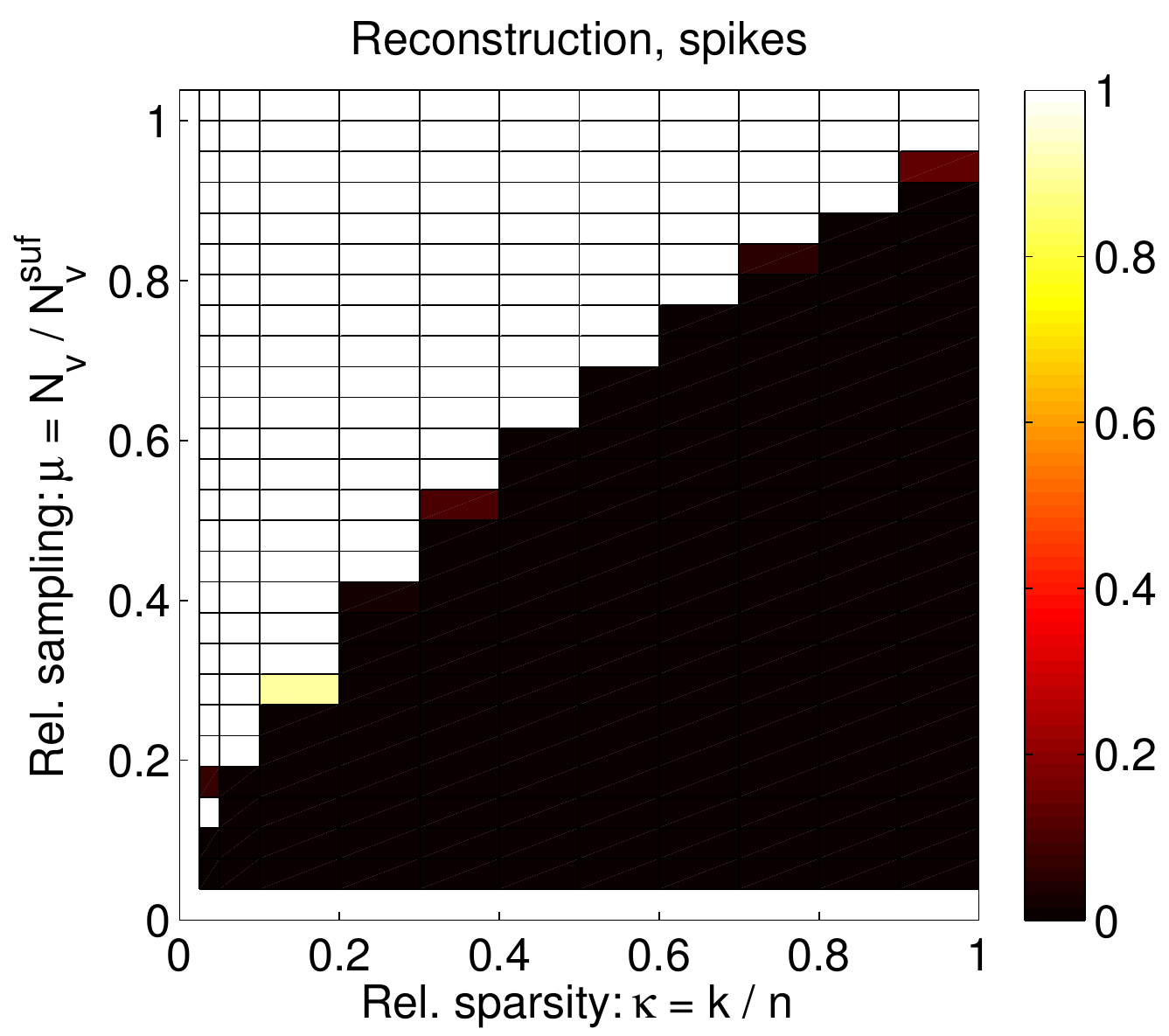}
\includegraphics[width=\wq\linewidth,trim=0cm 0cm 0cm 0cm, clip]{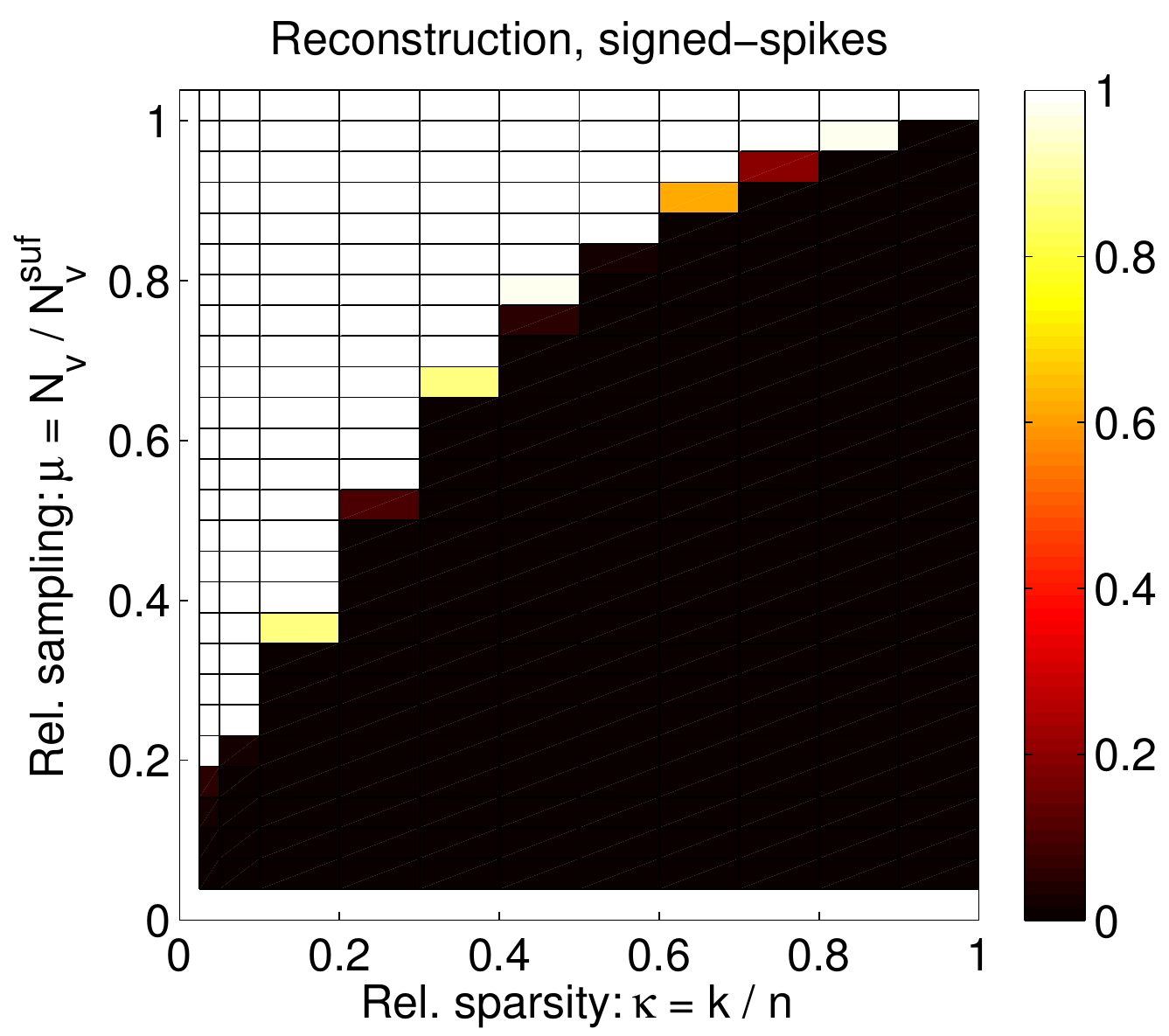}
\includegraphics[width=0.3\linewidth,trim=0cm 0cm 0cm 0cm,clip]{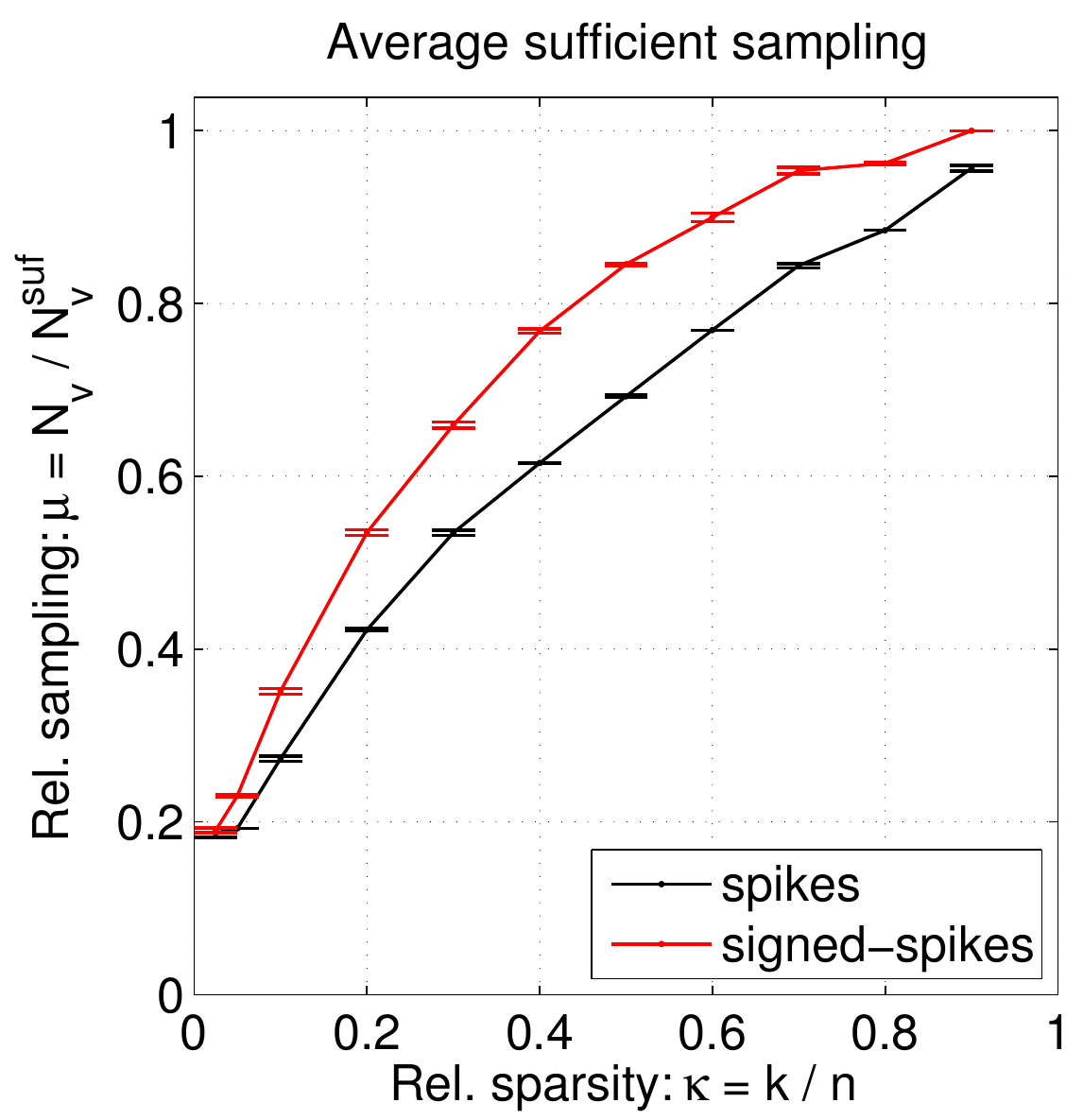}\\
\includegraphics[width=\wq\linewidth,trim=0cm 0cm 0cm 0cm, clip]{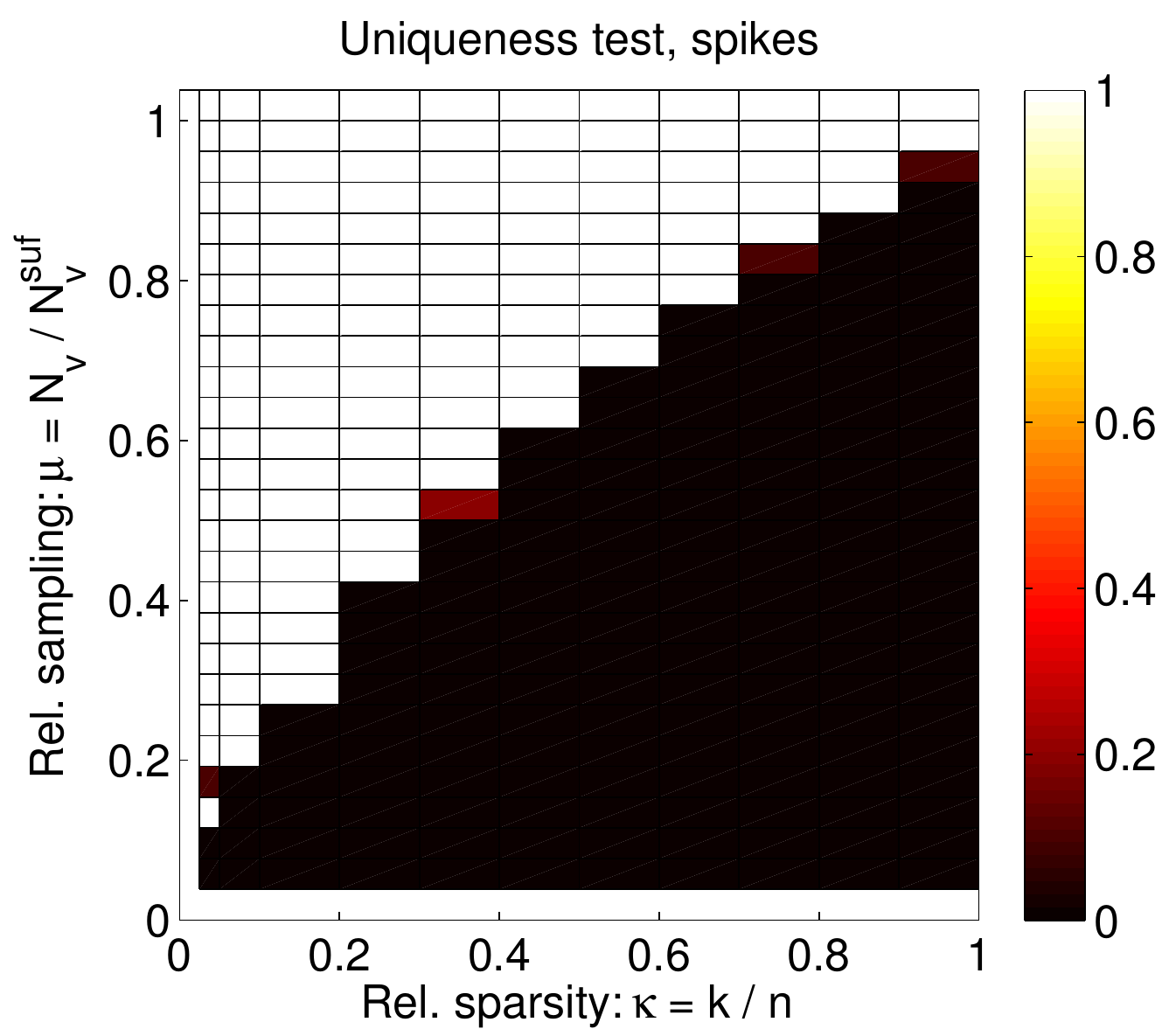}
\includegraphics[width=\wq\linewidth,trim=0cm 0cm 0cm 0cm, clip]{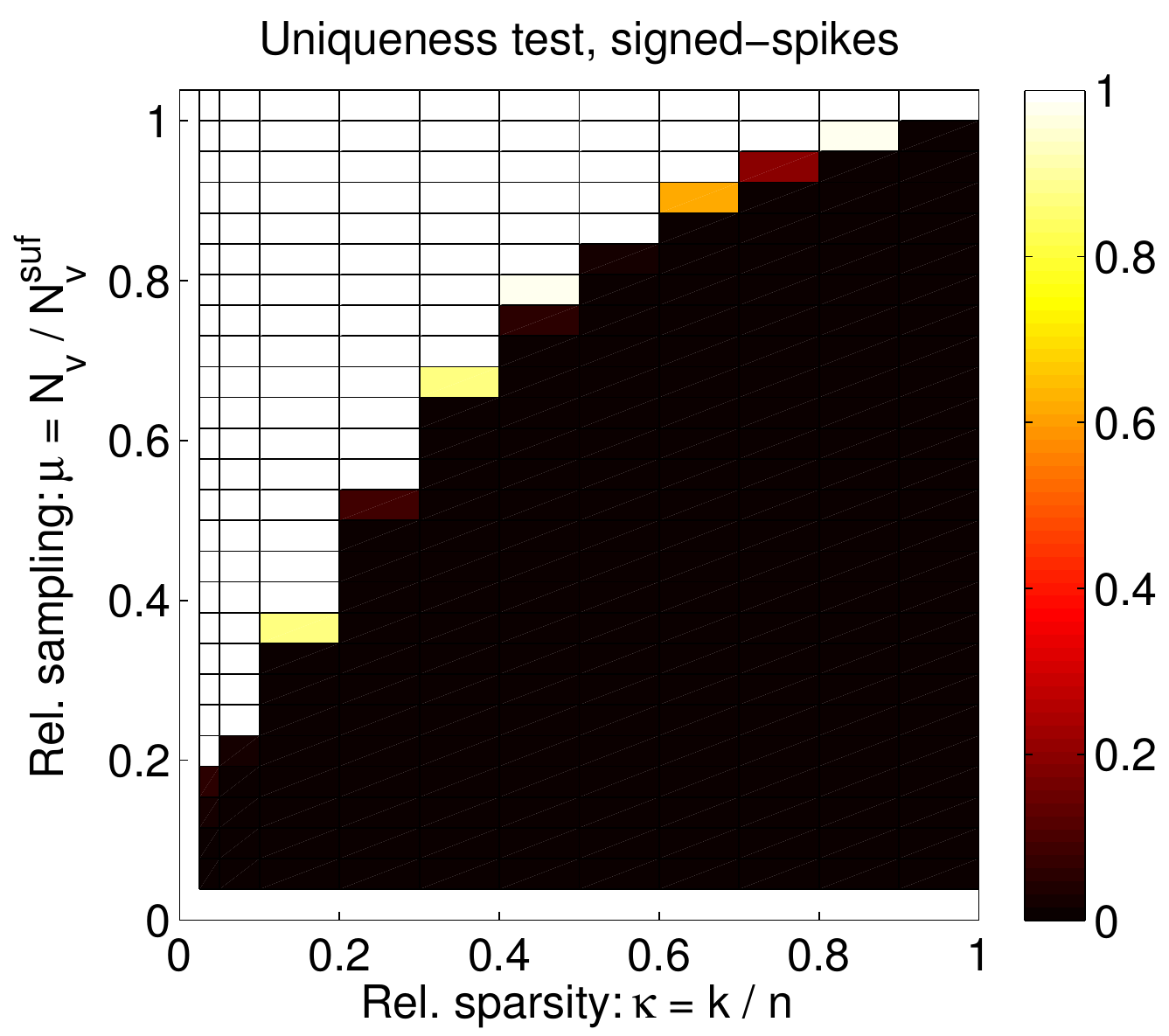}
\caption{Phase diagrams for \ellone{}. Top row: reconstruction, bottom row: uniqueness test. Left: \spikes{}, center: \signedspikes{}, right: average sufficient relative sampling point along with a $99\%$ confidence interval at each $\kappa$-value. For both classes reconstruction and uniqueness testing agree perfectly. For \signedspikes{} the average relative sufficient sampling curve is higher than for \spikes{} meaning that more projections are needed for recovery.\label{fig:bpphasediagrams}}
\end{figure}

Fig. \ref{fig:bpphasediagrams} shows, for both the \spikes{} and the \signedspikes{} classes, reconstruction and uniqueness testing phase diagrams: each rectangle corresponds to the relative sparsity value $\kappa$ at its left and the relative sampling $\mu$ at its bottom and the color indicates the fraction of instances perfectly recovered by reconstruction or deemed the unique solution by the uniqueness test. The phase diagrams are divided into into a `full-recovery' regime, in which all instances are uniquely recovered, and a `no-recovery' regime, where all instances fail to be recovered/be unique. Further, the transition from no-recovery to full-recovery is sharp, in the sense that for all relative sparsity values adding $1$--$2$ projection views changes the recovery rate from $0 \%$ to $100 \%$.

For both \spikes{} and \signedspikes{}, the uniqueness test phase diagrams are identical to the reconstruction phase diagrams, thereby mutually verifying correctness of each method and the attained phase diagrams.

The phase transition occurs at different sampling levels for the \spikes{} and \signedspikes{} classes. This is perhaps more easily seen in the right-most plot in Fig. \ref{fig:bpphasediagrams} in which the average relative sampling sufficient for recovery, i.e., the smallest value of $\mu$ at which all instances of a given sparsity are recovered, is plotted for both classes along with error bars for the $99\%$ confidence intervals. 
On average \signedspikes{} require more projection views for unique recovery. This is perhaps not surprising, as having negative pixel values can lead to negative entries in the data vector $\sino{}$, something that can not happen with non-negative pixel values due to the elements of $\sysmat{}$ being non-negative. Nevertheless, the phase diagram reveals quantitatively how signed entries affect recoverability.

\subsection{Phase diagrams for \atv{}}

In the same way as for \ellone{} we create reconstruction and uniqueness test phase diagrams for \atv{}. We consider first the \altproj{} class as well as its non-negative version 
\altprojnonneg{}. As the sparsity is measured after application of $\dmat{}^T$, the relative sparsity can now be in the range between $0$ and $2$, and in addition to the $\kappa$ values in the previous section, we include now $\kappa = 1.0, 1.1, \dots, 1.9$. 

\begin{figure}[tb]
\centering
\newcommand{\wq}{0.48}
\includegraphics[width=\wq\linewidth,trim=0cm 0cm 0cm 0cm, clip]{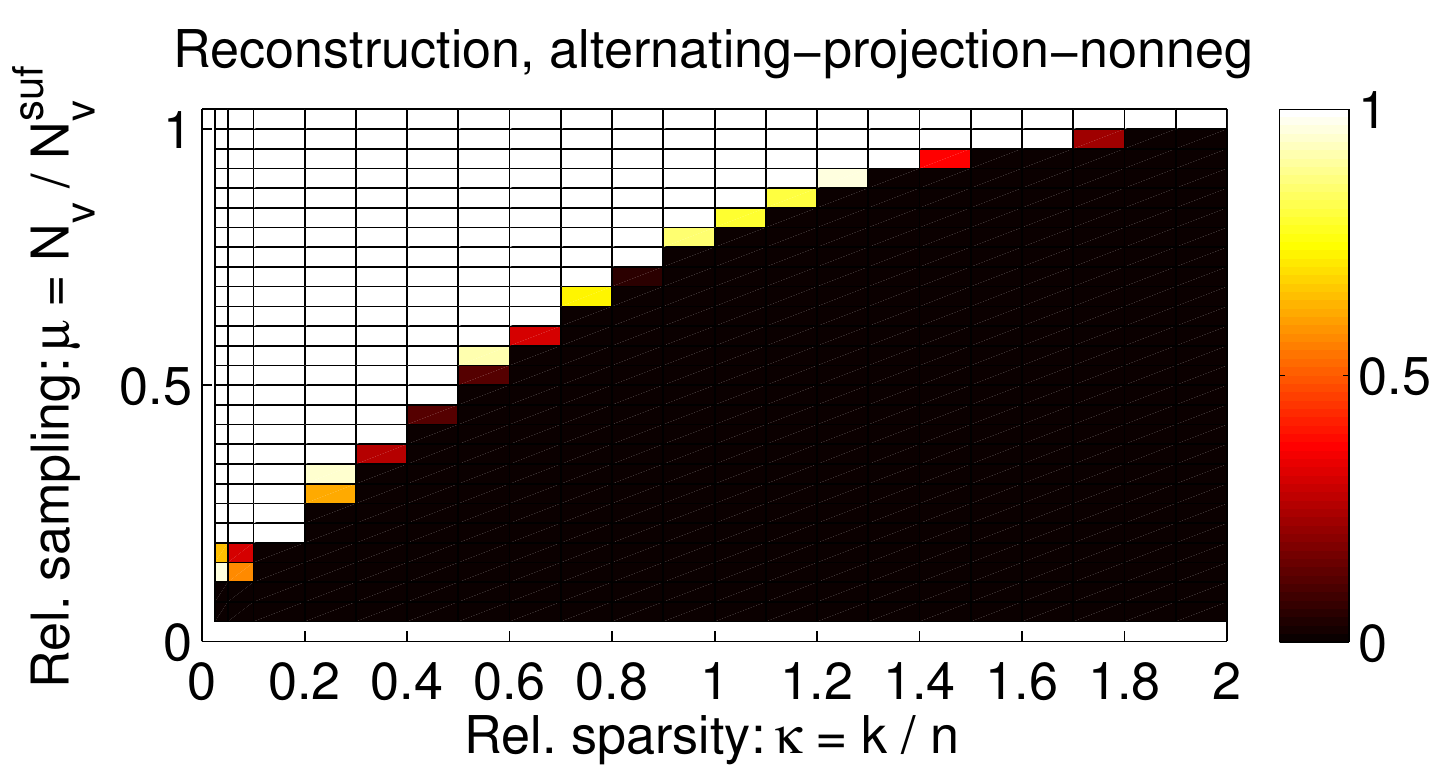}
\includegraphics[width=\wq\linewidth,trim=0cm 0cm 0cm 0cm, clip]{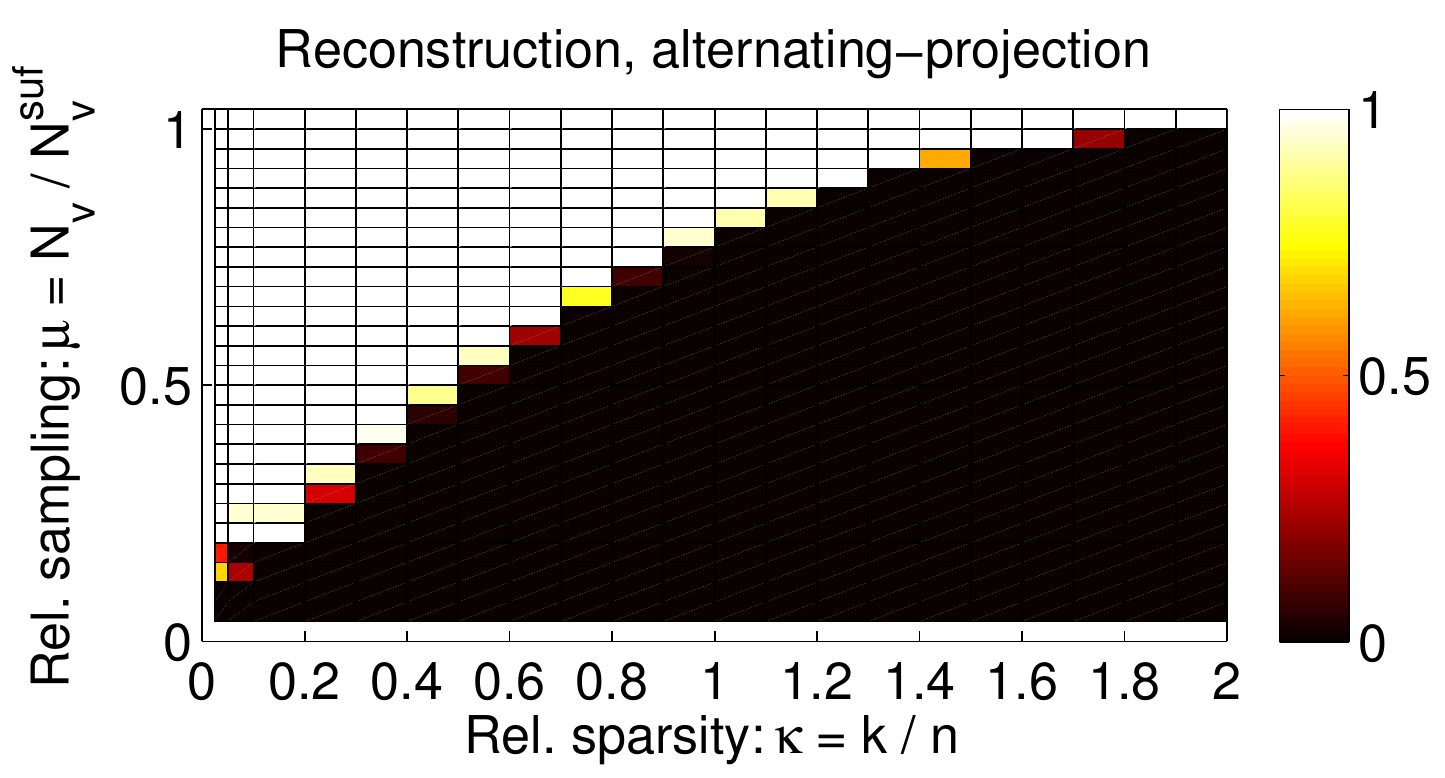}\\
\includegraphics[width=\wq\linewidth,trim=0cm 0cm 0cm 0cm, clip]{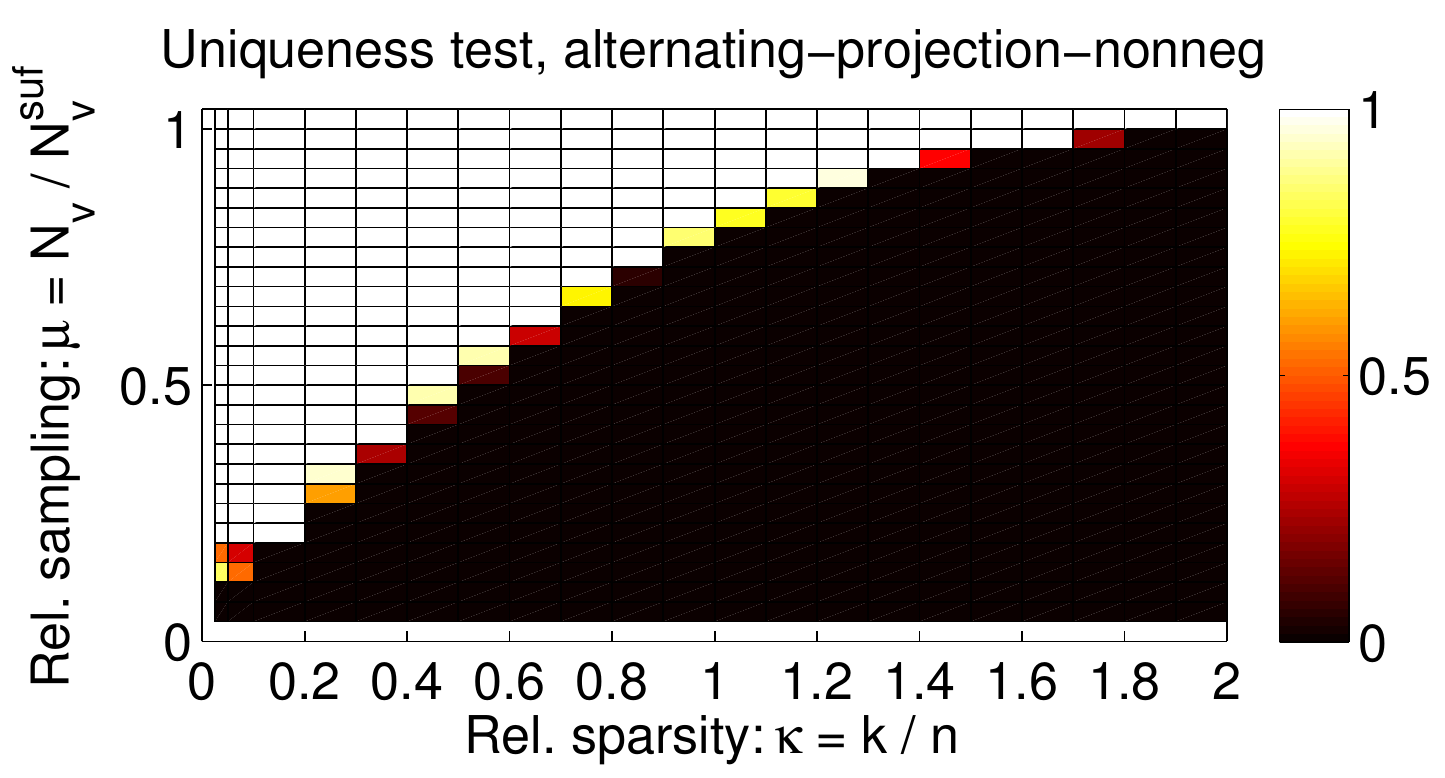}
\includegraphics[width=\wq\linewidth,trim=0cm 0cm 0cm 0cm, clip]{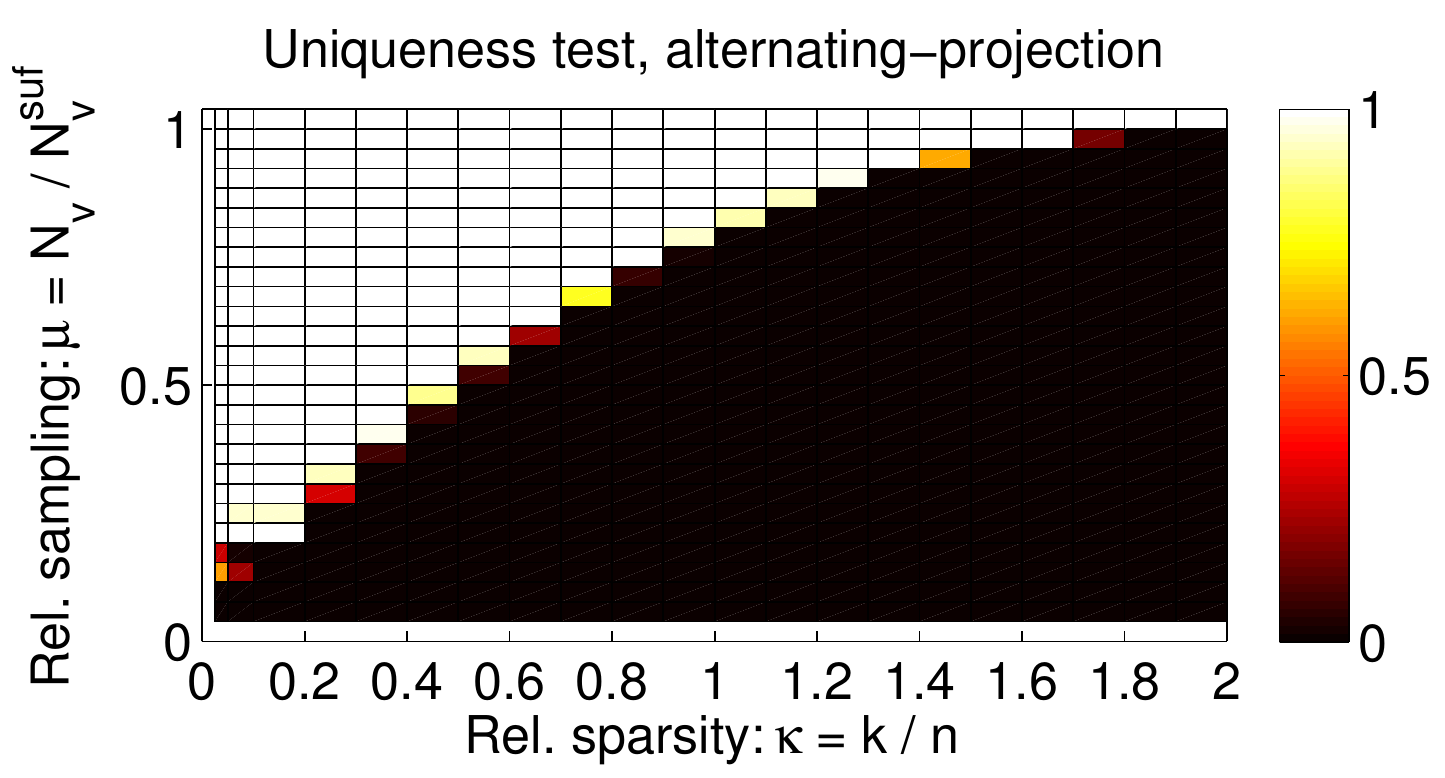}\\
\includegraphics[width=\wq\linewidth,trim=0cm 0cm 0cm 0cm,clip=true]{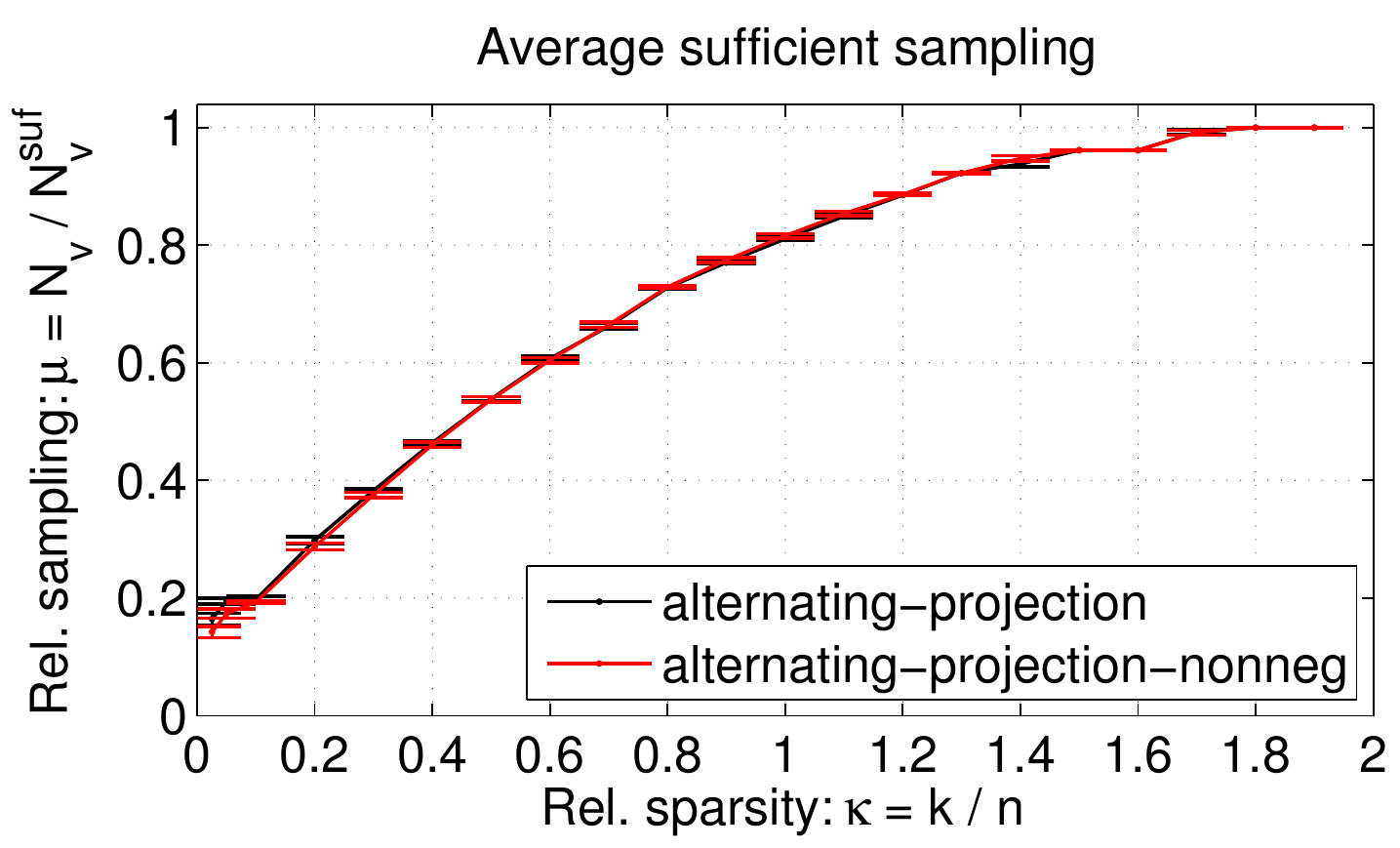}
\caption{Phase diagrams for \atv{}. Top row: reconstruction, middle row: uniqueness test. Left: \altprojnonneg{}, right: \altproj{}. Bottom row: average relative sufficient sampling point along with a $99\%$ confidence interval at each $\kappa$-value. No difference between classes is seen, unlike the \ellone{} case.\label{fig:anisophasediagrams}}
\end{figure}

The resulting reconstruction and uniqueness test phase diagrams are shown in Fig. \ref{fig:anisophasediagrams}. As for \ellone{}, we see a partition into `full-recovery' and `no-recovery' regimes separated by a sharp phase transition across $1$--$2$ projection views. The uniqueness test phase diagrams are identical to the reconstruction phase diagrams, except for a few cases in the transition region, for example for the smallest $\kappa$ values for \altprojnonneg{}. We explain these minor differences by the choice of numerical threshold for assessing recovery that is chosen  a priori to a constant $\epsilon$.

Contrary to the \ellone{} case, the phase diagrams for the signed and non-negative image class are identical, which is clearly seen in the bottom plot of Fig. \ref{fig:anisophasediagrams}, since the transition curves (average relative sampling for recovery) of the two classes coincide. It is interesting to observe such a fundamental difference with respect to having signed or non-negative images between \ellone{} and \atv{}. We can explain the difference by the fact that the constant vector, that is added to an \altproj{} image to obtain the non-negative version, is in the kernel of $\dmat{}^T$, and hence makes no difference for the \atv{} objective value. Further, according to Corollary \ref{cor:nullvector}, if an \altproj{} image is the unique \atv{} minimizer for its data, then the non-negative image obtained by adding a vector in the kernel of $\dmat{}^T$ will be the unique minimizer for its given data.

To study whether the phase diagrams depend on the image class we repeat the experiment for the \trununif{} image class. The resulting reconstruction and uniqueness testing phase diagrams are shown in Fig. \ref{fig:anisotrununifphasediagrams}. Again, the two phase diagrams are identical and show a sharp phase transition from the no-recovery to the full-recovery regime.  As can be seen from the right-most plot of the average relative sampling for recovery for the \trununif{} image class compared to the \altproj{} image class. The two curves are nearly identical at low and high relative sparsity values. However, in the mid-range there is a small difference corresponding to around one more projection view needed on average for the \trununif{} images to be recovered.

Both \altproj{} and \trununif{} image classes are constructed to yield images of a desired target sparsity, but in fundamentally different ways. The fact that the arising phase diagrams are so similar leaves us with the interpretation that the phase diagram and in particular the phase transition curve is governed mainly by the sparsity, while the particular image class has less influence.

\begin{figure}[tb]
 \newcommand{\wq}{0.48}
\begin{minipage}{\wq\linewidth}
\includegraphics[width=\linewidth,trim=0cm 0cm 0cm 0cm, clip]{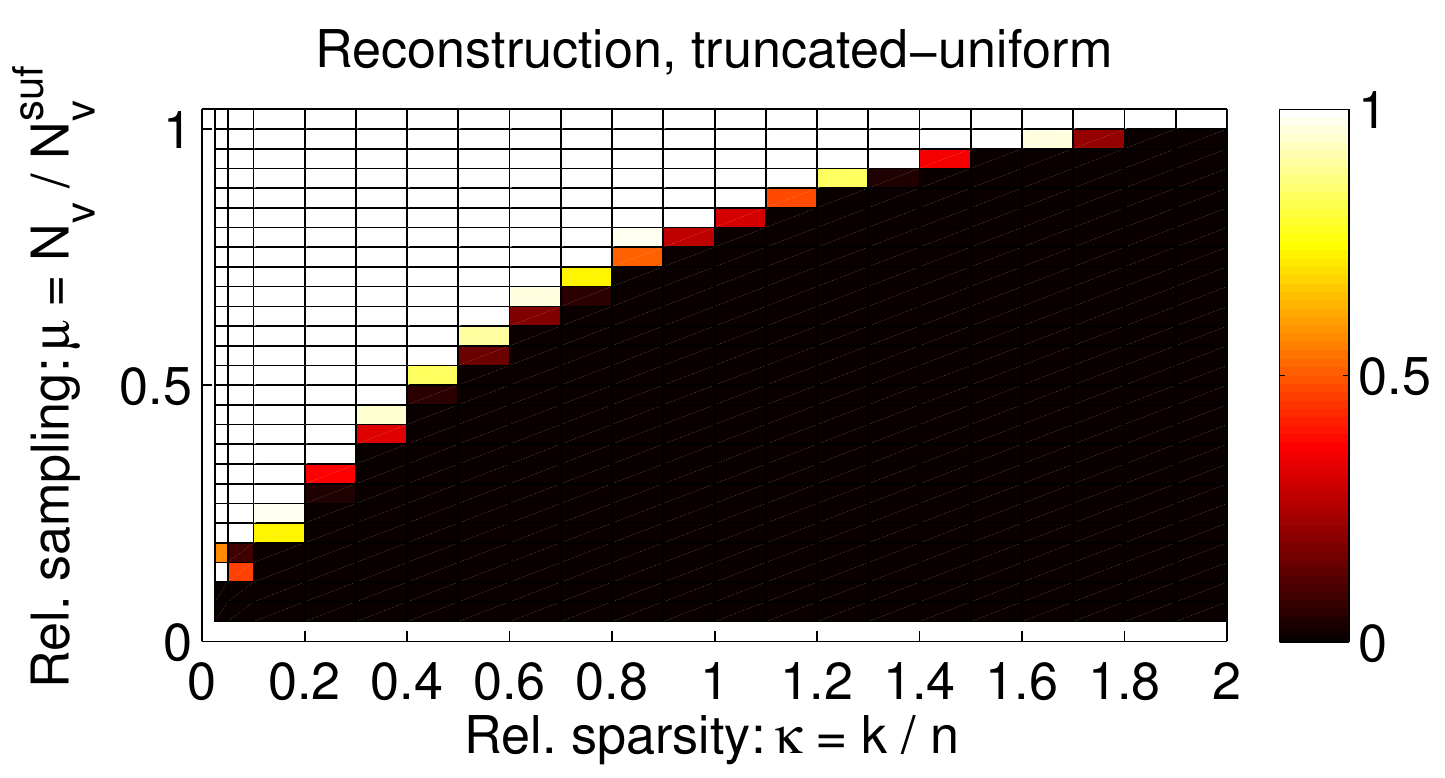}\\
\includegraphics[width=\linewidth,trim=0cm 0cm 0cm 0cm, clip]{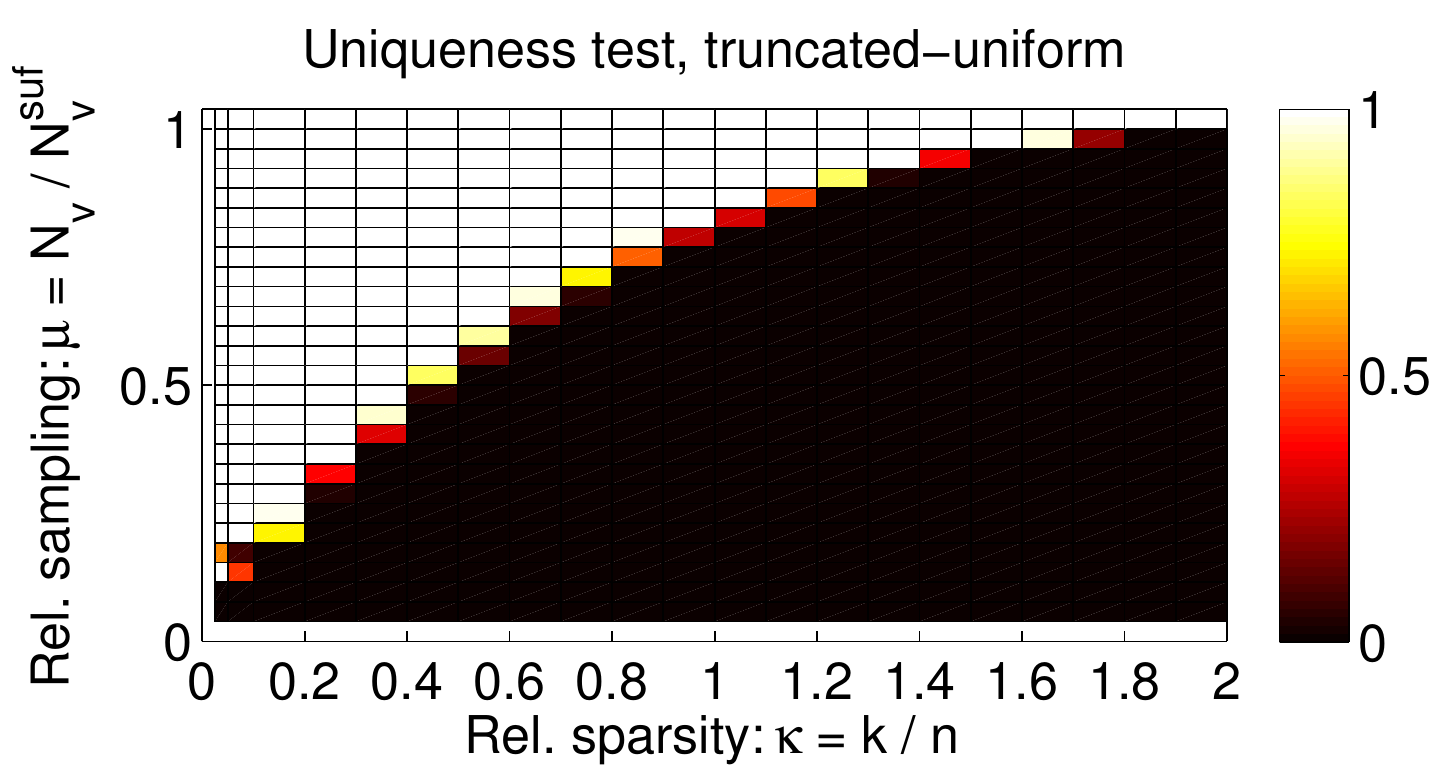}
\end{minipage}
\begin{minipage}{\wq\linewidth}
\includegraphics[width=\linewidth,trim=0cm 0cm 0cm 0cm,clip=true]{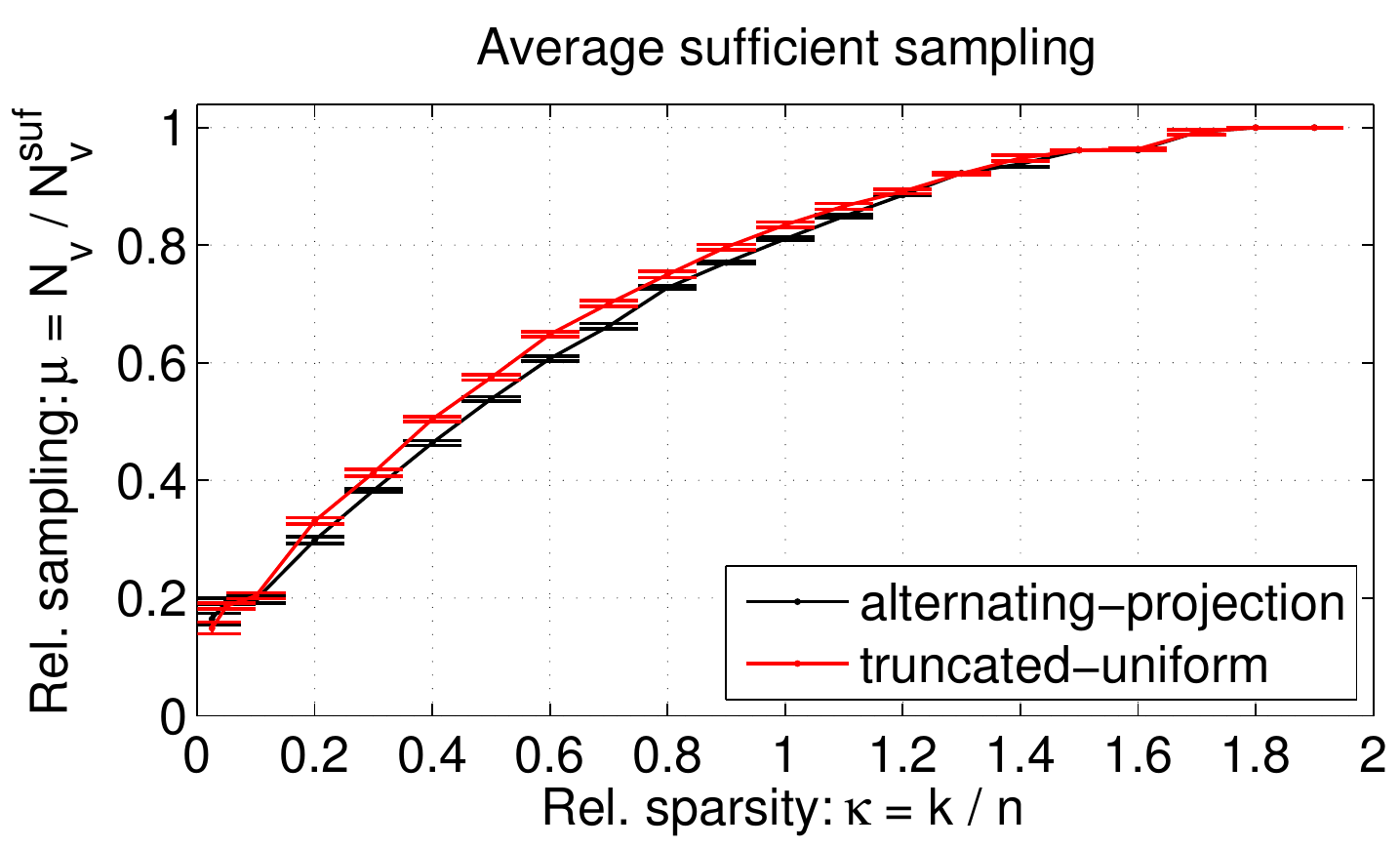}
\end{minipage}
\caption{Phase diagrams for \atv{} and \trununif{}. Top left: reconstruction, bottom left: uniqueness test. Right: 
average relative sufficient sampling point along with a $99\%$ confidence interval at each $\kappa$-value. Very small differences (one view more needed for recovering \trununif{} images) between classes is seen. \label{fig:anisotrununifphasediagrams}}
\end{figure}

\subsection{Phase diagrams for \itv{}}
We repeat the reconstruction and uniqueness test study for \itv{} with the \altproj{} image class designed for \itv{}. Recall that in contrast to the \ellone{} and \atv{} cases, Theorem \ref{thm:isotv} only provides a sufficient condition of solution uniqueness. This means that, in principle, instances that are not shown to be unique solutions still might be.

For \itv{}, the relative sparsity with respect to image size is between $0$ and $1$. 
As for \ellone{} we construct the phase diagram for the values $\kappa = 0.025, 0.05, 0.1, 0.2, ..., 0.9$ and $N_\text{v} = 1,\dots,32$, see Fig. \ref{fig:isotvphasediagrams}. Due to the numerically more challenging conic programs of isotropic TV solution accuracy was smaller than for \ellone{} and \atv{} and as a result we choose a the numerical threshold to $\epsilon = 10^{-3}$.

Once again we observe a partition into `full-recovery' and `no-recovery' regimes clearly separated by a sharp transition. Also, the reconstruction and uniqueness test phase diagram agree almost exactly and we ascribe again the minor differences to the uniform a priori choice of the numerical threshold $\epsilon$ for assessing recovery. The almost exact agreement between the phase diagrams is interesting considering the uniqueness test is only a sufficient condition. One may conjecture that the conditions of Theorem \ref{thm:isotv} are in fact also necessary and that a proof is only to be found.
\begin{figure}[tb]
\centering
\newcommand{\ww}{0.4}
\includegraphics[width=\ww\linewidth]{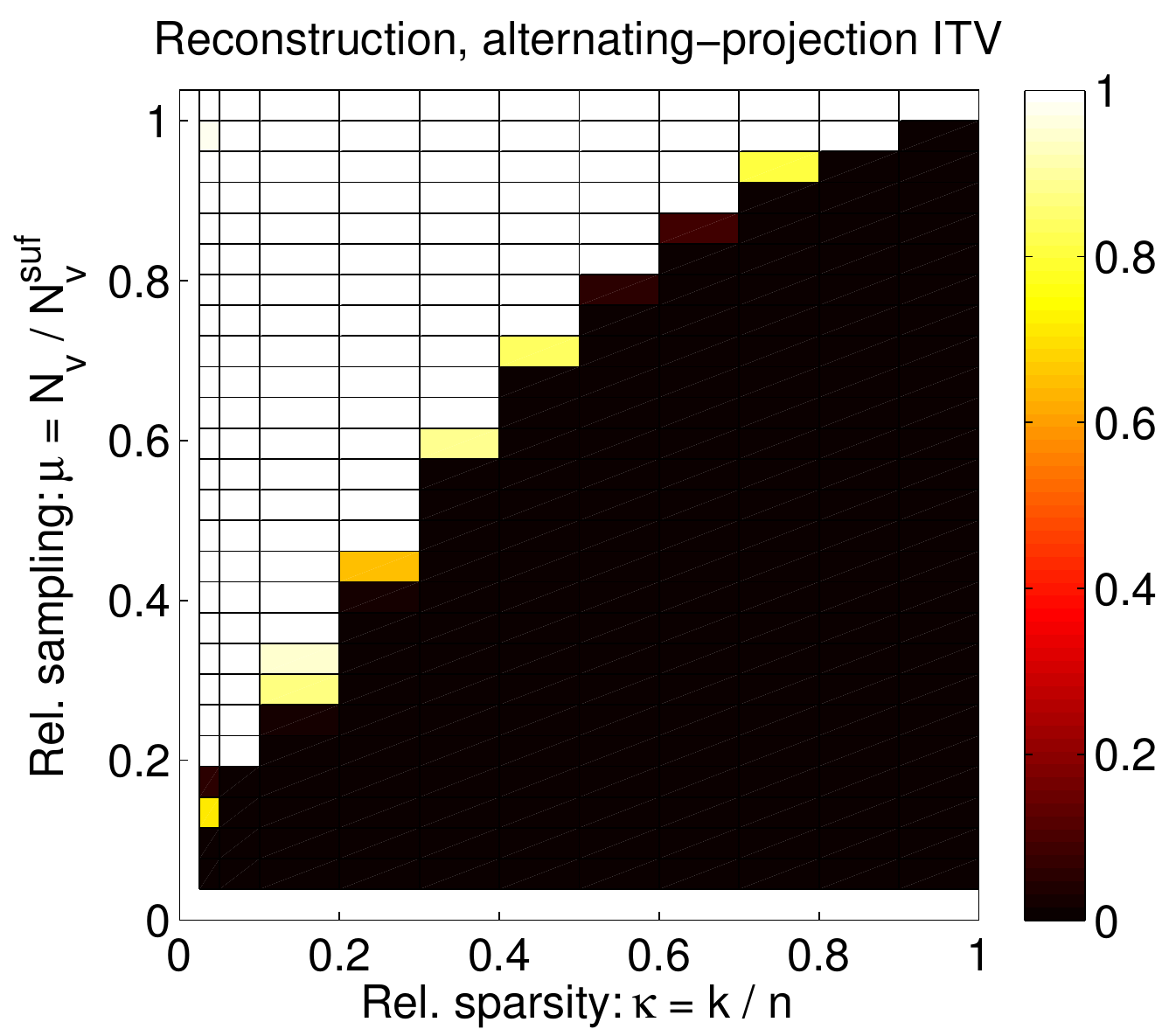}  
\includegraphics[width=\ww\linewidth]{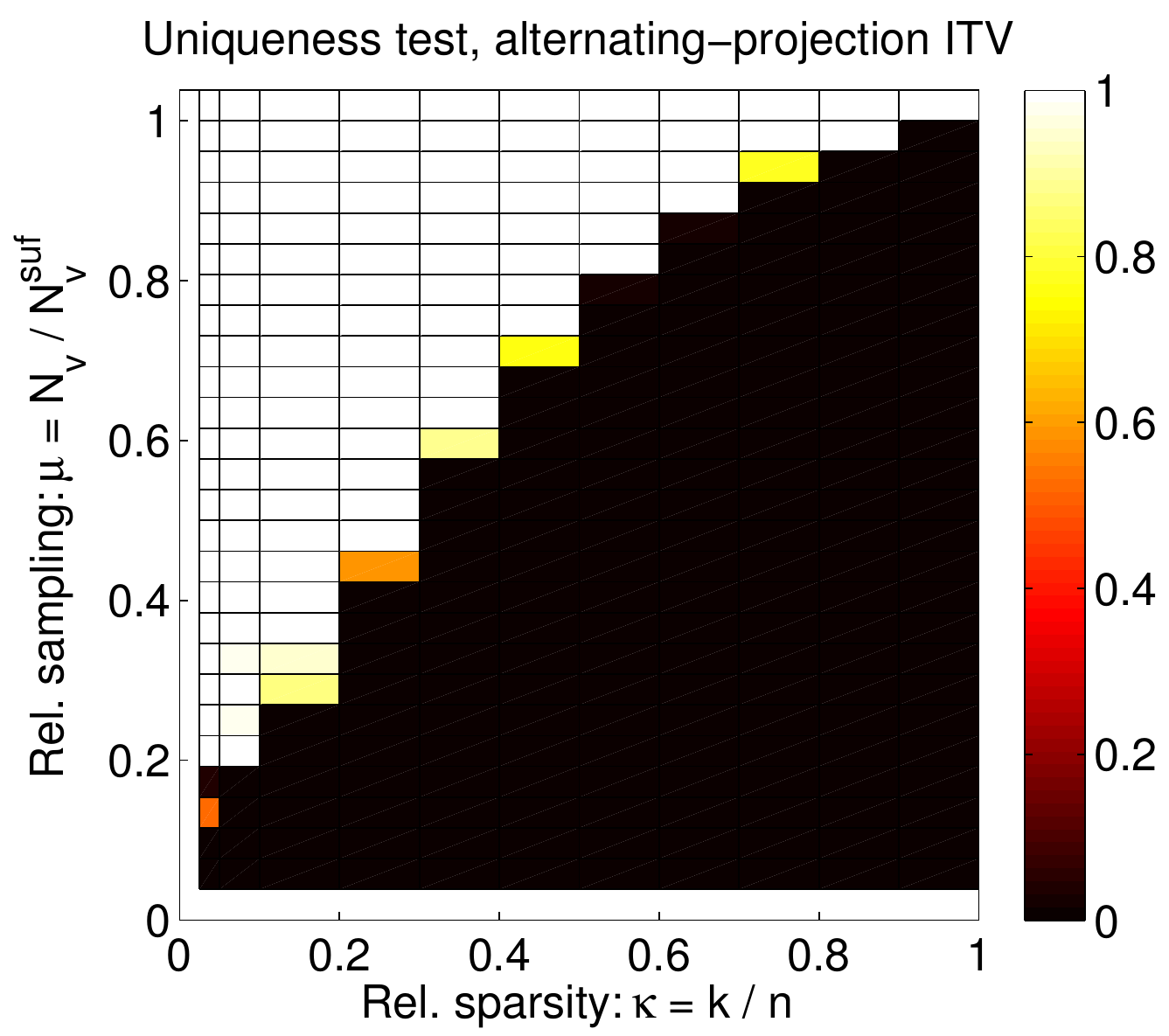}\\
\caption{Phase diagrams for \itv{}. Left:  Reconstruction,  right: uniqueness test.\label{fig:isotvphasediagrams}}
\end{figure}

Even though both \atv{} and \itv{} rely on gradient sparsity, comparing their phase diagrams do not reveal a straightforward conclusion as to which method provides the greatest undersampling potential because of the different sparsity measures. More comparisons of the two methods, for example a `cross-over study' of \itv{} reconstruction applied to the \atv{} image class and vice versa is beyond the scope of the present work, where the goal was to simply document phase transition behavior for CT measurements.

\subsection{Computational time for reconstruction vs. uniqueness testing}

We compared the computational times of reconstruction and uniqueness testing. 
All timing experiments were run in MATLAB 7.13 (R2011b) under Linux using MOSEK 6.0 on a Lenovo ThinkPad T430s with 
Intel Core i5-3320M processor (3 MB cache, up to 3.30 GHz) and 8 GB RAM, restricted to a single core. 

\begin{figure}[tb]
\centering
\newcommand{\ww}{0.37}
\includegraphics[width=\ww\linewidth]{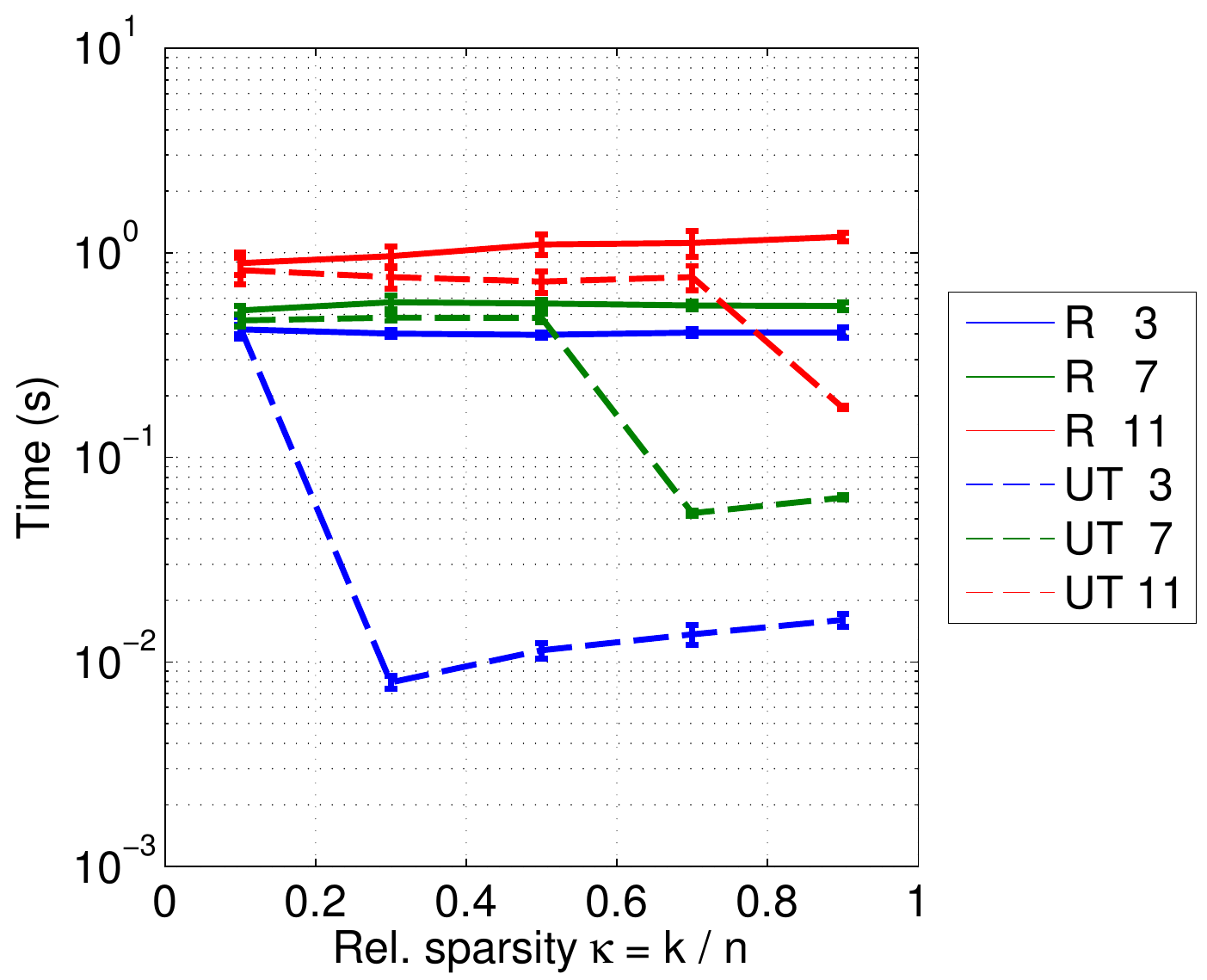}  
\includegraphics[width=\ww\linewidth]{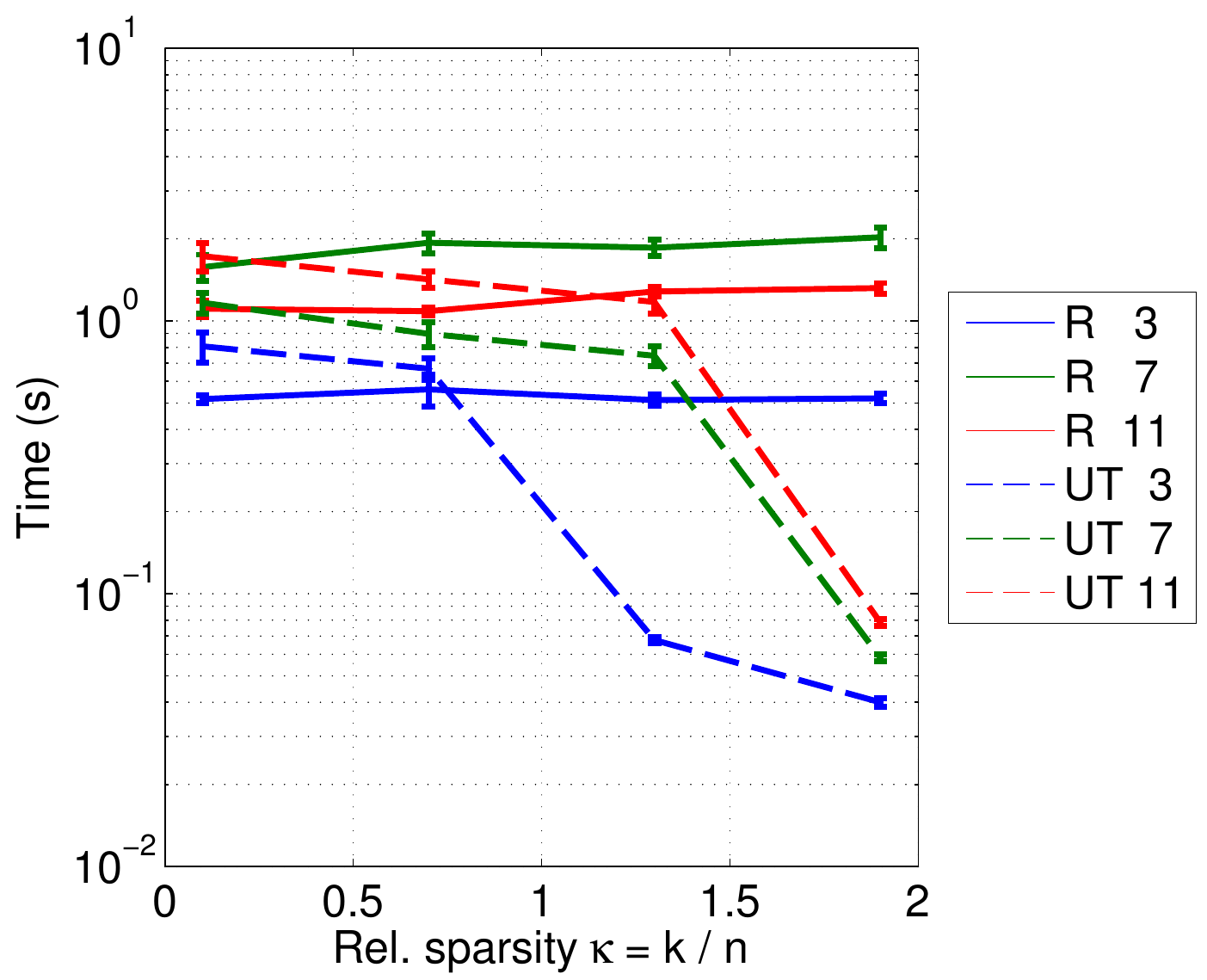}\\
\includegraphics[width=\ww\linewidth]{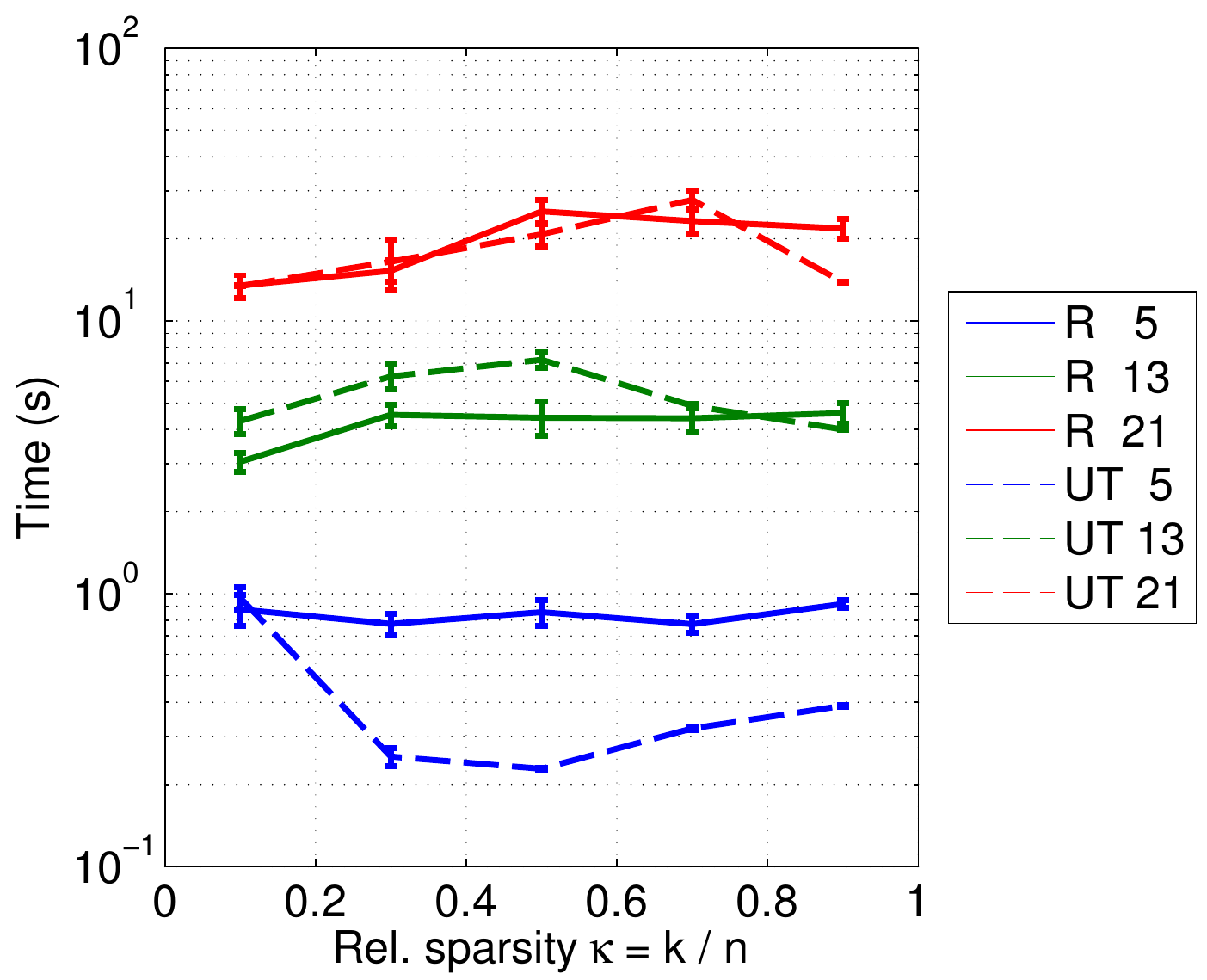} 
\includegraphics[width=\ww\linewidth]{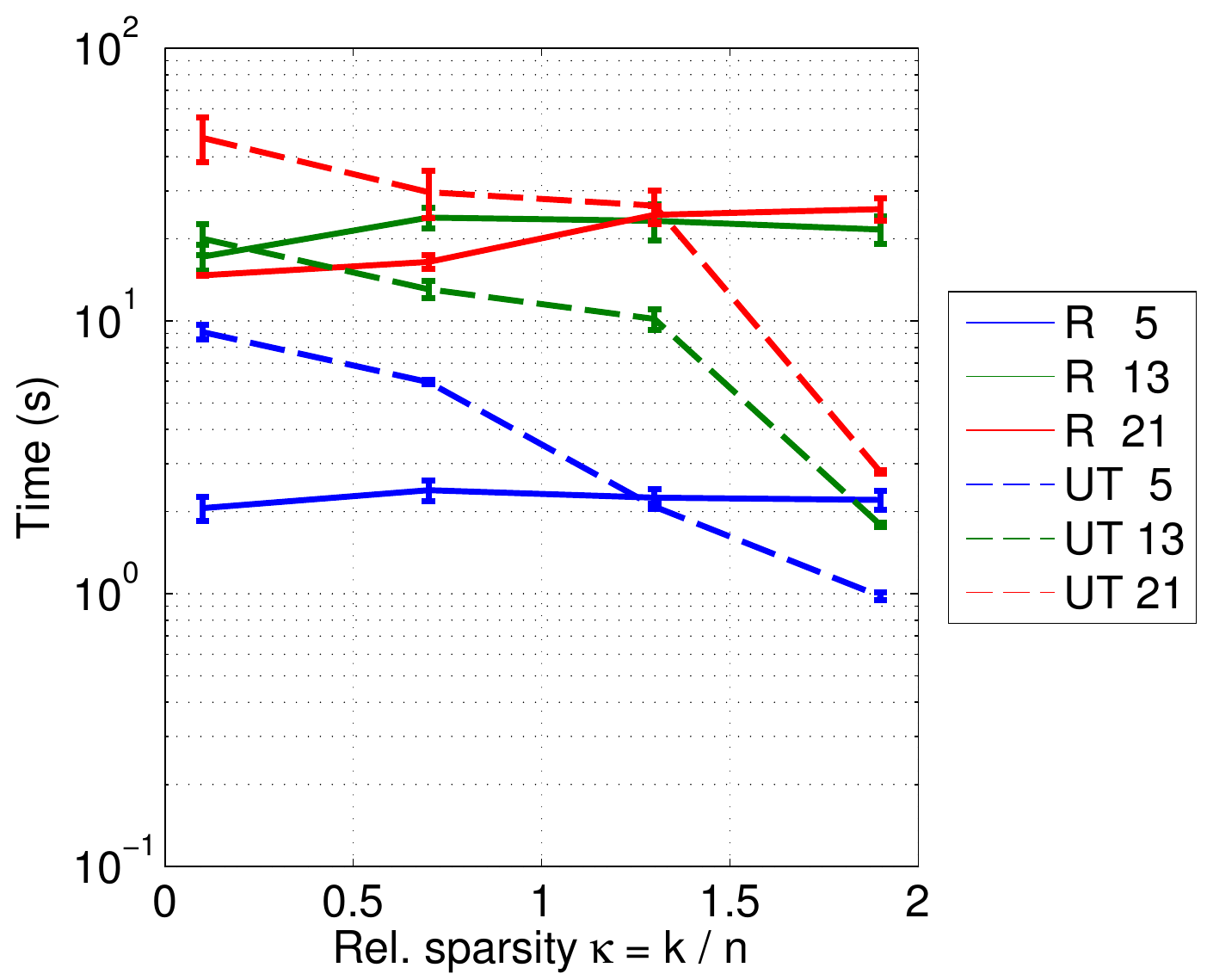}\\
\includegraphics[width=\ww\linewidth]{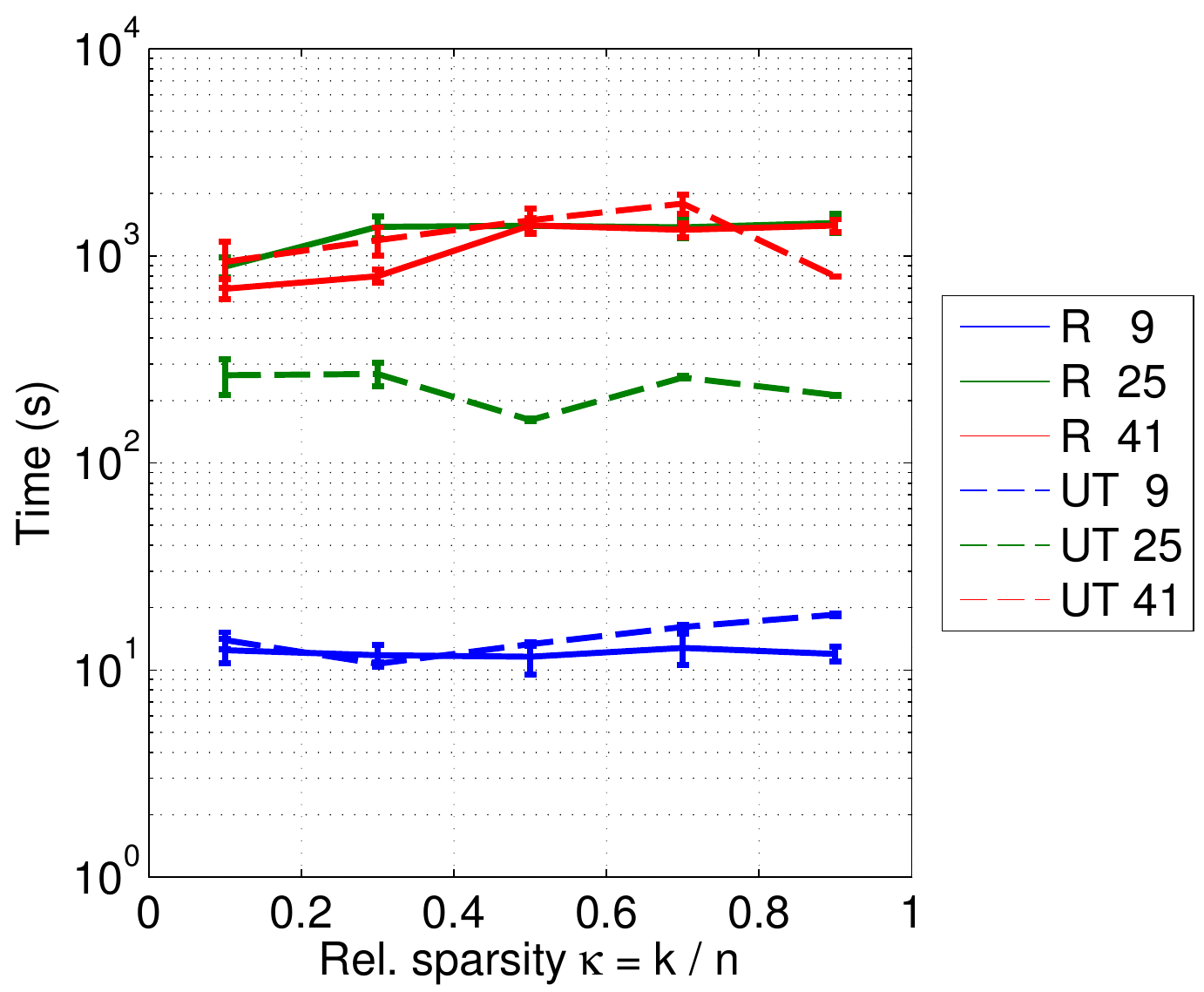}
\includegraphics[width=\ww\linewidth]{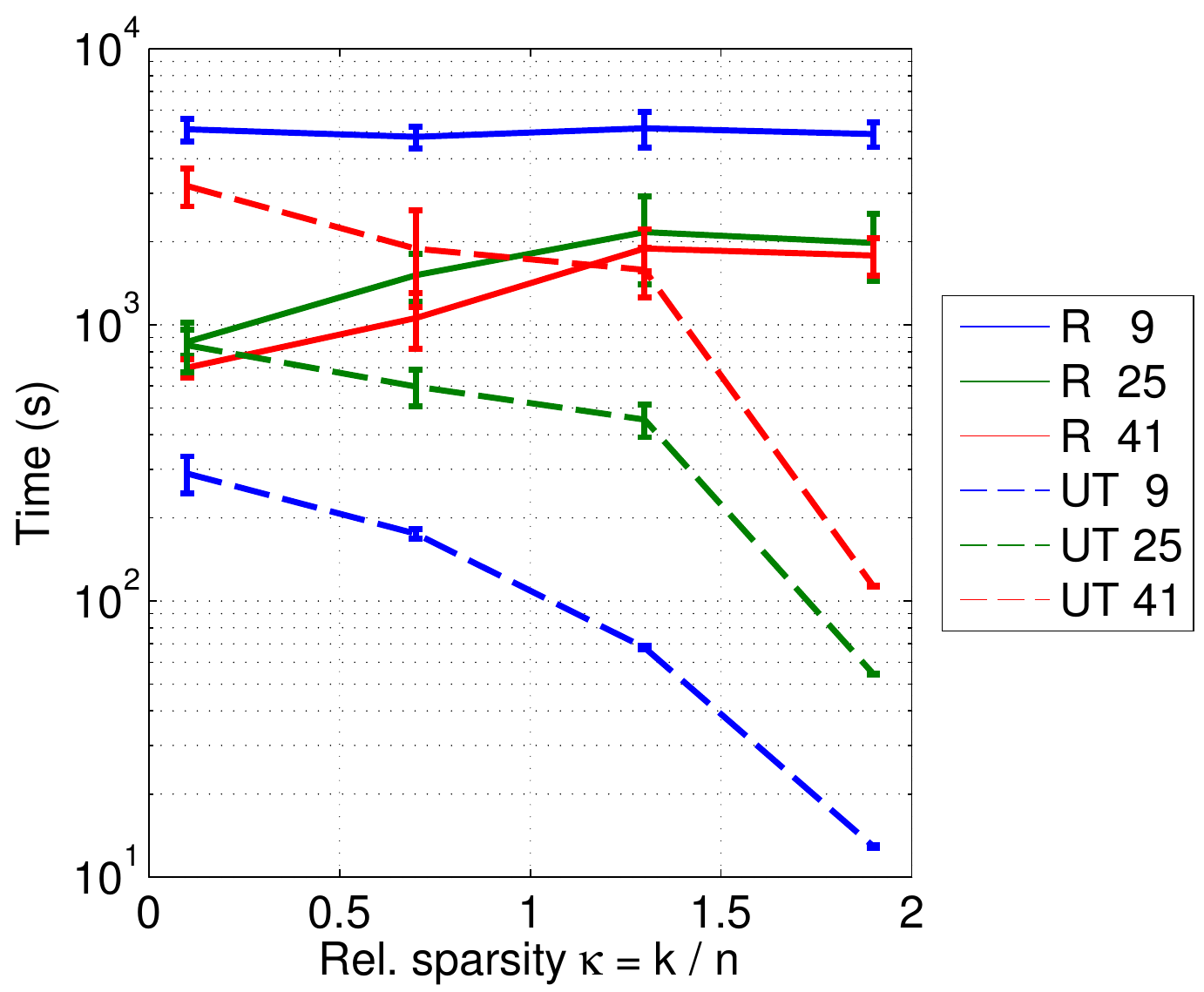}
\caption{Average reconstruction and uniqueness testing computational times with error bars of plus/minus one standard deviation over $10$ repetitions at each relative sparsity and relative sampling value. Left: \ellone{} for \signedspikes{} image class. Right: \atv{} for \altproj{} image class. Top: Size $32$, middle: size $64$, bottom: size $128$. Legend R: reconstruction (full lines), UT: uniqueness test (dashed lines).\label{fig:timing}}
\end{figure}

We choose experiments with low, medium and high relative sparsity and low, medium and high relative sampling cases to measure computational times for.
For the \signedspikes{} class we consider $10$ instances at each of the relative sparsity levels $\kappa = 0.1, 0.3, 0.5, 0.7, 0.9$. For image size $N_\text{side} = 32$, we use $3$, $7$ and $11$ views, for $N_\text{side} = 64$, we use $5$, $13$ and $21$ views and for $N_\text{side} = 128$ we take $N_\text{v} = 9$, $25$ and $41$. For \atv{} reconstruction, we consider the \altproj{} image class, for the same $N_\text{side} = 32, 64, 128$, relative sparsity $\kappa = 0.1, 0.7, 1.3, 1.9$ and the same number of views.

Results are shown in Fig. \ref{fig:timing}. For reconstruction, computational time generally shows little dependence with $\kappa$, if any, increasing $\kappa$ generally gives slightly increasing computational time. Uniqueness testing computational time tends to decrease with increasing $\kappa$. In several cases the uniqueness test is significantly faster than the reconstruction. In some of these cases, the relative sampling is low and the relative sparsity is high, which causes $\sysmat{}_I$ (in the \ellone{} case) to be non-injective, and the infinity-norm minimization problem needs not be solved. In other cases, for example the \ellone{} $N_\text{side} = 128$ case with $N_\text{v} = 25$, uniqueness testing is much faster across the relative sparsity range.
For uniqueness testing, computational time increases with the sampling level. For reconstruction, the low sampling cases are also the fastest, however the medium sampling case is not faster than the high sampling case in all cases, for example in the \ellone{} $N_\text{side} = 128$ case the times are comparable, and in the \atv{} $N_\text{side}=32$ case, the high sampling case is in fact faster.

We conclude that in general uniqueness testing is not slower than doing reconstruction, in most cases the computational times are comparable and in some cases, uniqueness testing is in fact faster. We note that uniqueness testing can be conveniently used in case of larger $\kappa$, where reconstruction tends to be the slower option.

It is clear that the reported computational times rely on our use of MOSEK for solving the optimization problems of reconstruction and uniqueness testing. The use of an interior-point method is what causes the computational time to increase so dramatically from the order of $10^0$ seconds at $N_\text{side}=32$ to $10^1$ seconds at $N_\text{side}=64$ and $10^3$ seconds at $N_\text{side} = 128$. With another optimization algorithm shorter running times may be observed with a different result of the comparison. 
However, our intention with the present study is not an exhaustive algorithm comparison, but merely to demonstrate that uniqueness testing can be accomplished in the same time, or faster, than reconstruction.

\section{Conclusion} \label{sec:conclusion}

The present work was motivated by understanding quantitatively how much undersampling is admitted for sparsity-exploiting reconstruction methods for CT given the lack of theoretical guarantees from compressed sensing. 
Our results demonstrate empirically that sharp average-case phase transitions from no recovery to full recovery as seen in compressed sensing also occur for CT measurements across a range of image classes and sparse reconstruction methods. The location of the phase transition, i.e., the level of sampling sufficient for recovery depends on the reconstruction method and is to a large degree is governed by the image sparsity, quite independent of the particular image class. 

Due to the inherently empirical nature of our study design it is clear that our results do not imply any theoretical guarantee. Further, being average-case results leaves the chance for single instance to require more or fewer samples for recovery than predicted by the average case. Nevertheless, we think the results may be used or extended to serve as guide lines for how to many CT samples to acquire based on prior knowledge about the image class and sparsity. Natural future work would be extensions toward more realistic scenarios including noisy data, model inconsistencies, specialized image classes, etc.

Constructing phase diagrams by reconstruction cannot establish solution uniqueness, which makes the uniqueness test more desirable from a theoretical perspective. However, we observed almost exact agreement between reconstruction and uniqueness test phase diagrams, so in practice the advantage may be negligible. Also, the reconstruction approach has the advantage that it can be run directly on any reconstruction problem with no need to derive specific uniqueness conditions and as such is more easily generalizable.

In our view, the presented empirical evidence suggests that an underlying theoretical explanation of phase transition behavior in CT may exist. Establishing such theory would have large implications for the understanding of sparse reconstruction in CT and we hope that the present results can serve as a step towards this goal.

\section*{Acknowledgements}
The authors are grateful to Emil Y. Sidky for inspiring discussions. This work was supported in part by Advanced Grant 291405 `HD-Tomo' from the European Research Council and by grant 274-07-0065 `CSI:\ Computational Science in Imaging' from the Danish
Research Council for Technology and Production Sciences.
JSJ acknowledges support from The Danish Ministry of Science, Innovation and Higher Education's Elite Research Scholarship.
CK and DAL acknowlege support from the Deutsche Forschungsgemeinschaft (DFG) within the project `Sparse Exact and Approximate Recovery' under grant LO 1436/3-1.

\bibliographystyle{unsrt}
\bibliography{uniqueness_paper.bib}

\end{document}